\documentclass[10pt,a4paper]{amsart}
\usepackage[utf8]{inputenc}
\usepackage{amsmath}
\usepackage{amsfonts}
\usepackage{amssymb}
\usepackage[left=2cm,right=2cm,top=2cm,bottom=2cm]{geometry}
\usepackage{bbm}
\usepackage{mathrsfs}
\usepackage[all]{xy}
\usepackage{bm}
\usepackage{manfnt}
\usepackage{leftindex}
\usepackage[color=blue!30,textsize=tiny]{todonotes}
\usepackage{tikz}
\usepackage[backref=page]{hyperref}
\author{Geoffrey Powell}
\title[Hairy graph complexes]{Cyclic operads, Koszul complexes,  and hairy graph complexes}
\address{Univ Angers, CNRS, LAREMA, SFR MATHSTIC, F-49000 Angers, France}
\email{Geoffrey.Powell@math.cnrs.fr}
\urladdr{https://math.univ-angers.fr/~powell/}
\date{}

\keywords{}
\subjclass[2000]{}
\newtheorem{THM}{Theorem}
\newtheorem{PROP}[THM]{Proposition}

\newtheorem{thm}{Theorem}[section]
\newtheorem{prop}[thm]{Proposition}
\newtheorem{cor}[thm]{Corollary}
\newtheorem{lem}[thm]{Lemma}
\theoremstyle{definition}
\newtheorem{defn}[thm]{Definition}
\newtheorem{exam}[thm]{Example}

\theoremstyle{remark}
\newtheorem{rem}[thm]{Remark}
\newtheorem*{rem*}{Remark}
\newtheorem{nota}[thm]{Notation}
\newtheorem{hyp}[thm]{Hypothesis}

\newcommand{\f}{\mathcal{F}}
\renewcommand{\phi}{\varphi}
\renewcommand{\hom}{\mathrm{Hom}}
\newcommand{\sym}{\mathfrak{S}}
\newcommand{\fs}{{\mathsf{FS}}}
\newcommand{\kmod}{\mathtt{Mod}_\kk}
\newcommand{\calc}{\mathcal{C}}

\newcommand{\nat}{\mathbb{N}}
\newcommand{\zed}{\mathbb{Z}}
\newcommand{\rat}{\mathbb{Q}}
\newcommand{\ext}{\mathrm{Ext}}
\newcommand{\op}{^\mathrm{op}}
\newcommand{\ob}{\mathrm{Ob}\hspace{2pt}}

\newcommand{\fb}{\mathsf{FB}}
\newcommand{\finj}{{\mathsf{FI}}}

\newcommand{\id}{\mathrm{Id}}

\newcommand{\triv}{\mathsf{triv}}
\newcommand{\sgn}{\mathsf{sgn}}
\newcommand{\aut}{\mathrm{Aut}}
\newcommand{\fin}{\mathsf{FA}}
\newcommand{\n}{\mathbf{n}}
\newcommand{\kk}{\mathbbm{k}}

\newcommand{\dash}{\text{-}}
\newcommand{\modules}{\mathsf{mod}}
\newcommand{\ub}{\mathsf{ub}}
\newcommand{\db}{\mathsf{db}}
\newcommand{\vo}{\mathscr{V}_{\mathsf{O}}}
\newcommand{\vsp}{\mathscr{V}_{\mathsf{Sp}}}
\newcommand{\cala}{\mathscr{A}}
\newcommand{\calb}{\mathscr{B}}
\newcommand{\g}{\mathfrak{g}}

\newcommand{\tor}{\mathrm{Tor}}
\newcommand{\tw}{_{(+,-)}}

\newcommand{\fiord}{\finj^{\mathsf{ord}}}

\newcommand{\dub}{\overrightarrow{\ub}}
\newcommand{\dubord}{\dub^\mathsf{ord}}

\newcommand{\m}{\mathbf{m}}

\newcommand{\domega}{\mathfrak{d}}

\newcommand{\cpd}{\mathscr{C}}
\newcommand{\sfb}{S_\odot}
\newcommand{\lfb}{\Lambda_\odot}

\newcommand{\copds}{\mathsf{CyclOpd}}

\newcommand{\nuco}{\copds^{\mathrm{nu}}}
\newcommand{\kdbmm}{(\kk \db)_{(-;-)}}
\newcommand{\kubg}{(\kk\ub)_{(\pm; \mp)}}

\newcommand{\torskubg}{\kubg\dash\mathbf{Tors}}
\newcommand{\loc}{\mathbf{\pi}}
\newcommand{\sat}{\mathbf{s}}
\newcommand{\tors}{\mathsf{tors}\ }
\newcommand{\sbf}{\mathsf{b}}

\newcommand{\schur}{\mathbf{S}}
\newcommand{\ogp}{\mathbf{O}}
\newcommand{\spgp}{\mathbf{Sp}}
\newcommand{\rep}{\mathrm{Rep}}
\newcommand{\ot}{T_\mathsf{O}}
\newcommand{\spt}{T_\mathsf{Sp}}
\newcommand{\falg}{\f^\mathrm{alg}}
\newcommand{\stabvo}{\mathrm{Stab}_\mathsf{O}}
\newcommand{\stabvsp}{\mathrm{Stab}_\mathsf{Sp}}
\newcommand{\ocx}{\mathfrak{C}_\mathsf{O}}
\newcommand{\spcx}{\mathfrak{C}_\mathsf{Sp}}
\numberwithin{equation}{section}
\begin{document}

\begin{abstract}
In this paper, we revisit the construction of the hairy graph complexes associated to a cyclic operad, by exploiting modules over the appropriate twisted $\kk$-linearization of the downward Brauer category (working over a field $\kk$ of characteristic zero). The different flavours (even or odd) of complexes appear as forms of Koszul complexes; the Koszul property of the $\kk$-linear category provides an elegant homological interpretation of their homology.

This approach allows a second form of Koszul complex to enter the picture. For the `even' flavour, this corresponds to a precursor of the Chevalley-Eilenberg complex of the Conant-Vogtmann Lie algebra associated to a cyclic operad and a symplectic vector space (generalizing Kontsevich's Lie algebras). Again, the cohomology of the  Koszul complex has an elegant interpretation.  

This sheds light on the relationship between the unstable case and Kontsevich's identification (generalized by Conant and Vogtmann) of the homology in the infinite-dimensional case with a form of graph homology (in the even case).  We observe that this is already of interest in the case of an algebra with involution, viewed as a cyclic operad. 
\end{abstract}

\maketitle

\section{Introduction}
\label{sect:intro}

For $(B, \sigma)$ a unital $\kk$-algebra with involution over a field $\kk$ of characteristic zero and $(V,\omega)$ a  symplectic $\kk$-vector space, one has the sub Lie algebra of symmetric matrices 
\[
sp_{(V,\omega)} (B,\sigma) \subset gl _V (B).
\]
(see \cite[Chapter 10.5]{MR1600246}, for example).
One can  stabilize with respect to $(V, \omega)$ to obtain the infinite dimensional Lie algebra $sp(B,\sigma)$. 

For $A$ a unital associative $\kk$-algebra and $gl(A)$ the corresponding Lie algebra, the  Loday-Quillen-Tsygan theorem identifies the Lie algebra homology  $H^{\mathsf{CE}} _* (gl(A))$ in terms of the cyclic homology $HC_* (A)$ of $A$ (see  \cite[Chapter 10]{MR1600246} for this and the following). Loday and Procesi \cite{MR937318} generalized this to the symplectic case, establishing the isomorphism
\[
H^{\mathsf{CE}} _* (sp(B,\sigma)) 
\cong 
S^* (HD_* (B,\sigma)[1]), 
\]
where the left hand side is Lie algebra homology and the right hand side is the free graded commutative algebra on the shift of the  dihedral homology of $(B,\sigma)$.  This  identifies the stable Lie algebra homology in terms of $HD_* (B,\sigma)$, which does not  depend explicitly on the Lie algebra $sp (B,\sigma)$. Moreover, one can identify dihedral homology as a form of graph homology. (This may be compared with the $\kk$-algebra case (without involution); in recent work, Dotsenko \cite{MR4945404} has shown how the Loday-Quillen-Tsygan theorem fits into a similar graph-homology type framework, generalizing work of Fuks \cite{MR874337}.)

One can also consider, for each (finite dimensional) symplectic vector space $(V,\omega)$, the Lie algebra homology 
$$
H_*^\mathsf{CE}(sp_{(V,\omega)}(B,\sigma));
$$ 
this is natural with respect to  $(V, \omega)$.  One can thus seek the precursor to this structure by exploiting this naturality. This leads to the following two  questions: what is the algebraic structure involved and what is the appropriate homology theory? Moreover, one can also ask: what happens when $B$ is not unital? 

The above has a far-reaching generalization, by considering $\kk$-algebras with involution as a very particular case of cyclic operads in $\kk$-vector spaces, in the same way that associative $\kk$-algebras are a very particular case of operads. (We will not require that  a cyclic operad has a unit.)

 Consider a cyclic operad $\cpd$; Conant and Vogtmann \cite{MR2026331} showed that, for $(V, \omega)$ a symplectic vector space, the Schur functor $\cpd (V)$ (depending only on the underlying vector space $V$) has a natural Lie algebra structure, induced by the composition structure map of the cyclic operad together with the symplectic form.  This generalizes the Lie algebras considered by Kontsevich \cite{MR1247289,MR1341841} for the `associative', `commutative', and `Lie' cyclic operads. When the cyclic operad is $\cpd_{(B,\sigma)}$, corresponding to the $\kk$-algebra with involution $(B,\sigma)$, then $\cpd _{(B,\sigma)} (V)$ is naturally isomorphic to the Lie algebra $sp_{(V,\omega)} (B,\sigma)$ considered above.

One can then form the Chevalley-Eilenberg complex $(\Lambda^* \cpd (V), d_\mathsf{CE})$ and consider its homology. Exploiting the naturality with respect to the symplectic vector space, Conant and Vogtmann (generalizing  Kontsevich's results), showed that the  homology of the stable Lie algebra is related to the even graph homology for the cyclic operad $\cpd$. This can be seen as a far-reaching generalization of the Loday-Procesi theorem. In the other direction, Conant, Kassabov, and Vogtmann \cite{MR3347586} have explored the relationship between dihedral homology and hairy graph homology by focusing upon genus one graphs.

Here we exploit fully the naturality with respect to the symplectic vector space; for this, we use a form of Brauer-Schur-Weyl duality. The key input is that, for $(V,\omega)$ a symplectic vector space, the association $V^{\otimes \bullet} \colon n \mapsto V^{\otimes n}$ has two types of naturality: the action of the symmetric group $\sym_n$ by place permutations, as well as the maps induced by the form $\omega : V^{\otimes 2} \rightarrow \kk$. This structure is encoded in the fact that $V^{\otimes \bullet}$ is a module over a twisted $\kk$-linearization of the downward Brauer category $\db$; this twisted variant is denoted here   by $(\kk \db)_{(-;+)}$. (See Sections \ref{sect:fiord} and \ref{sect:twist} for a review of the upward and downward Brauer categories and their twisted $\kk$-linearizations.) Moreover, writing $\vsp$ for the category of symplectic vector spaces, this defines a functor 
$$\spt^\bullet \colon \vsp \rightarrow (\kk \db)_{(-;+)}\dash\modules
$$
 to $(\kk \db)_{(-;+)}$-modules (here $\kmod$ is the category of $\kk$-vector spaces), given by $\spt^\bullet (V, \omega) :=V^{\otimes \bullet}$, reflecting the naturality of the construction.

The importance of these structures is well-established. For instance, the fundamental theorems of symplectic invariant theory were exploited by Loday and Procesi in their identification of the  Lie algebra homology of stable symplectic Lie algebras; these theorems are intimately related to Brauer-Schur-Weyl duality. Similarly, Sam and Snowden have shown the importance of the downward Brauer categories and their twisted variants in considering stability in representation theory \cite[Section 4]{MR3376738} and \cite{MR3876732}.

As above, one can ask what are the appropriate algebraic structures that correspond to the natural Lie algebra $\cpd (V)$ (respectively $(\Lambda^* \cpd (V), d_\mathsf{CE})$) via Brauer-Schur-Weyl duality?  The answer brings into play a further ingredient: the relationship between cyclic operads and modules over the downward Brauer category.

Recall that a cyclic operad can be considered as a particular form of modular operad, in the sense of Getzler and Kapranov \cite{MR1601666}. Stoll \cite{MR4541945} has given an  elegant characterization of modular operads as algebras over the Brauer {\em properad}. This can be restricted to give a characterization of cyclic operads.

Here we use a weaker (but related) result, avoiding  working with properads and their algebras. (This corresponds to the fact that we allow non-connected graphs, whereas properads   impose connectivity, by design.) To explain this, recall that the category of $\kk \fb$-modules (where $\fb$ is the category of finite sets and bijections),  has  a symmetric monoidal structure provided by the Day convolution product, denoted here by $\odot$. One can thus form the associated symmetric and exterior algebras, denoted $\sfb^*$ and $\lfb^*$, on any $\kk \fb$-module.  

 These constructions can be applied to the underlying $\kk \fb$-module of a non-unital cyclic operad $\cpd$. The following statement uses $(\kk \db)_{(-;-)}$, another twisted $\kk$-linearization of $\db$.

\begin{THM}(Theorem \ref{thm:cpd_db_kdbm}.)
For $\cpd$ a non-unital cyclic operad, the composition induces 
\begin{enumerate}
\item 
a natural $\kk\db$-module structure on $\sfb^* (\cpd)$; 
\item 
a natural $(\kk \db)_{(-;-)}$-module structure on $\lfb^* (\cpd)$.
\end{enumerate}
\end{THM}

This does not yet answer our questions. The missing ingredient is provided by  the fact that $(\kk \db)_{(-;-)}$ is a Koszul $\kk$-linear category over $\kk \fb$. In particular, one has a Koszul complex, denoted here by
\begin{eqnarray}
\label{eqn:K_-otimes_lfb_cpd} 
\mathscr{K}_- \otimes_{(\kk \db)_{(-;-)}} \lfb^* (\cpd)
.
\end{eqnarray}
This is a complex of $(\kk \ub)_{(-;+)}$-modules, where the twisted form $(\kk \ub)_{(-;+)}$ of $\kk$-linearization of $\ub$, the upward Brauer category, is the opposite of $(\kk \db)_{(-;+)}$. 

This leads to  a direct construction of the Chevalley-Eilenberg complex of $\cpd (V)$, without explicitly using its Lie algebra structure:

\begin{THM} (Theorem \ref{thm:vker_K-_versus_CE}.)
\label{THM:CE}
For $(V, \omega)$ a symplectic vector space and $\cpd$ a non-unital cyclic operad, the complex 
$$
V^{\otimes \bullet} \otimes_{(\kk\ub)_{(-;+)}} \mathscr{K}_- \otimes_{(\kk \db)_{(-;-)}} \lfb^* (\cpd) 
$$
is naturally isomorphic to $(\Lambda^* \cpd (V), d_\mathsf{CE})$.
\end{THM}

Moreover, there is an interesting interpretation of the homology of  (\ref{eqn:K_-otimes_lfb_cpd}), which is a direct consequence of the Koszul property (see Corollary \ref{cor:homology_complexes_cpd}):

\begin{THM}
For $\cpd$ a non-unital cyclic operad, the homology of $\mathscr{K}_- \otimes_{(\kk \db)_{(-;-)}} \lfb^* (\cpd)$ is naturally isomorphic to 
$$
\ext^* _{(\kk \db)_{(-;-)}} (\kk \fb, \lfb^* (\cpd)).
$$
This isomorphism respects the natural $(\kk\ub)_{(-;+)}$-module structures.
\end{THM}

This  gives a relationship between the Lie algebra homology of $\cpd (V)$ and the homology of the complex (\ref{eqn:K_-otimes_lfb_cpd}), which can be encoded in  a universal coefficients-type result. Namely, since $V^{\otimes \bullet} \otimes_{(\kk\ub)_{(-;+)}} -$ is not exact, there is a universal coefficients spectral sequence 
\[
\tor^{(\kk \ub)_{(-;+)} } _* \big(
\spt^\bullet, \ext^*_{(\kk \db)_{(-;-)} } (\kk \fb , \lfb^*(\cpd)) 
\big) 
\Rightarrow 
H^\mathsf{CE} _*(\cpd (-))
\]
as functors on $\vsp$,  where $H^\mathsf{CE} _*(\cpd (-))$ is the Lie algebra homology of the Conant-Vogtmann Lie algebra $\cpd (-)$. In particular, one has the  edge homomorphism:
\[
\spt^\bullet 
\otimes_{(\kk \ub)_{(-;+)} }
\ext^*_{(\kk \db)_{(-;-)} } (\kk \fb , \lfb^*(\cpd)) 
\rightarrow 
H^\mathsf{CE} _*(\cpd (-)).
\]
Clearly the above depends on the natural $(\kk \ub)_{(-;+)}$-module structure on $\ext^*_{(\kk \db)_{(-;-)} } (\kk \fb , \lfb^*(\cpd)) $.

We now turn to the `stable' side of the story. This is based on Sam and Snowden's approach to stabilization \cite{MR3376738}, which is reviewed in Section \ref{sect:forms}. The functors from $\vsp$ to $\kk$-vector spaces that we are considering are {\em algebraic}, lying in the full subcategory $\falg (\vsp)$ of the category of functors $\f (\vsp)$.
There is a stabilization functor $\stabvsp : \falg (\vsp) \rightarrow \rep (\spgp)$ to Sam and Snowden's category of algebraic representations of the `infinite' symplectic group $\spgp_\infty$ and this induces an equivalence of categories
$$
\stabvsp :
\falg (\vsp)/ \falg_\tors (\vsp) \stackrel{\simeq}{\rightarrow } \rep (\spgp),
$$
where the domain is obtained by localizing away from the torsion functors.

\begin{rem*}
The theory has a counterpart for the `orthogonal' case. One works with the appropriate category $\vo$ of orthogonal vector spaces and there is a stabilization functor 
$\stabvo : \falg (\vo) \rightarrow \rep (\ogp)$, where $\rep (\ogp)$ is the category of algebraic representations of the `infinite' orthogonal group $\ogp_\infty$. As in the symplectic case, this induces an equivalence of categories 
$$
\stabvo :
\falg (\vo)/ \falg_\tors (\vo) \stackrel{\simeq}{\rightarrow } \rep (\ogp).
$$
For the remainder of this introduction, we focus upon the symplectic case.
\end{rem*}

We use the approximation to stabilization that is given by the functor $
\hom_{\f (\vsp)} (- , \spt^\ast), 
$ from $\f (\vsp)\op$ to $(\kk \db)_{(-;+)}$-modules. This induces the exact functor 
$$
\hom_{\f (\vsp)} (- , \spt^\ast)
: 
\big(\falg (\vsp)/ \falg_\tors (\vsp)\big) \op
\longrightarrow 
(\kk \db)_{(-;+)} \dash \modules 
$$
which restricts to an equivalence between the respective full subcategories of finite length objects. The latter property explains why this can be considered as an approximation to $\stabvsp$.

This is of particular interest when applied to the image of the generalized Schur functor $\spt^\bullet \otimes_{(\kk \ub)_{(-;+)}} -$. As explained in Proposition \ref{prop:natural_iso_hom_vker_vker}, for  $N$ a $(\kk \ub)_{(-;+)}$-module, there is a natural isomorphism of $(\kk \db)_{(-;+)}$-modules:
\[
\hom_{\f (\vsp)} (\spt^\bullet \otimes_{(\kk \ub)_{(-;+)}} N, \spt^\ast) 
\cong 
\hom_\kk ((\kk \ub)_{(-;+)} ^\sharp \otimes _{(\kk \ub)_{(-;+)}} N, \kk),
\]
where $(-)^\sharp$ denotes vector space duality.

This motivates the consideration of the functor $(\kk \ub)_{(-;+)} ^\sharp \otimes _{(\kk \ub)_{(-;+)}}-$ on $(\kk \ub)_{(-;+)}$-modules. This  fits into the adjunction 
 \begin{equation*}
 \resizebox{\hsize}{!}{
 $
\xymatrix{
\kubg^\sharp \otimes_{\kubg} - 
\ar@{} [r]|(.55):
&
\kubg\dash\modules
\ar@<1ex>[rr]
\ar@{}[rr]|(.5)\perp
&&
\torskubg 
\ar@<1ex>[ll]
&
\hom_{\kubg} (\kubg^\sharp , -),
\ar@{}[l]|(.6) :
}
$
}
\end{equation*} 
where $\torskubg$ is the full subcategory of torsion $\kubg$-modules. The functor  $\kubg^\sharp \otimes_{\kubg} - $ vanishes on torsion modules (see Section \ref{sect:torsion} for this and more).

The above  applies to the complex appearing in Theorem \ref{THM:CE}. Namely, applying the functor $\hom_{\f (\vsp)} (- , \spt^\ast)$ to the complex 
$$
\spt^\bullet \otimes_{(\kk\ub)_{(-;+)}} \mathscr{K}_- \otimes_{(\kk \db)_{(-;-)}} \lfb^* (\cpd) 
$$
yields the $\kk$-linear dual of the following complex 
\begin{eqnarray}
\label{eqn:sharp_complex}
(\kk \ub)_{(-;+)}^\sharp \otimes_{(\kk \ub)_{(-;+)} } \mathscr{K}_- \otimes_{(\kk \db)_{(-;-)} } \lfb^* (\cpd),
\end{eqnarray}
which we consider as  a second form of Koszul complex. 

The significance of (\ref{eqn:sharp_complex}) is shown by the following (see Proposition \ref{prop:weakly_stabilize_CE_complex}):

\begin{PROP}
\label{PROP:H^CE}
For $\cpd$ a non-unital cyclic operad, there is a natural isomorphism of graded $(\kk \db)_{(-;+)}$-modules 
$$
\hom_{\falg (\vsp)} (H_* ^\mathsf{CE}( \cpd (-)) , \spt^\ast ) 
\cong 
\big( H_* ((\kk \ub)_{(-;+)}^\sharp \otimes_{(\kk \ub)_{(-;+)} } \mathscr{K}_- \otimes_{(\kk \db)_{(-;-)} } \lfb^* (\cpd)
\big)^\sharp .
$$ 
\end{PROP}

Since the left hand side is our approximation to the stabilization of $H_* ^\mathsf{CE} (\cpd (-))$ (considered as a graded object of $\falg (\vsp)$), it follows that the homology of 
(\ref{eqn:sharp_complex}) calculates this approximation (up to the vector space duality).

Moreover,  (\ref{eqn:sharp_complex}) is  a complex of $(\kk \ub)_{(-;+)} $-modules and its homology has a nice interpretation (see Corollary \ref{cor:homology_complexes_cpd} again):

\begin{THM}
For $\cpd$ a non-unital cyclic operad, the homology of $(\kk \ub)_{(-;+)}^\sharp \otimes_{(\kk \ub)_{(-;+)} }\mathscr{K}_- \otimes_{(\kk \db)_{(-;-)}} \lfb^* (\cpd)$ is naturally isomorphic to 
$$
\tor_* ^{(\kk \db)_{(-;-)}} (\kk \fb, \lfb^* (\cpd)).
$$ 
This isomorphism respects the natural $(\kk\ub)_{(-;+)}$-module structures.
\end{THM}

The complex (\ref{eqn:sharp_complex}) can be related to the even hairy graph complex. Thus, Proposition \ref{PROP:H^CE} can be viewed as a  generalization to the case of graphs with hairs of  Kontsevich's theorem  relating the homology of his infinite dimensional Lie algebras to graph homology, as generalized to cyclic operads by Conant and Vogtmann \cite{MR2026331}. 

For this, so as to relate to the hairy graph complex considered in \cite{MR3029423}, for a cyclic operad $\cpd$, we restrict to the subobject $\cpd_{\geq 3}$ of $\cpd$ with underlying $\kk \fb$-module that is supported on sets of cardinality greater than two. A proof of the following result is sketched in Section \ref{sect:complexes}:

\begin{THM}
\label{THM:hairy}
(Theorem \ref{thm:hairy-graph_complexes}.)
For $\cpd$ a non-unital cyclic operad and $l \in \nat$, the complex
\begin{eqnarray*}
(\kk\ub)_{(-;+)}(-,\mathbf{l}) ^\sharp  \otimes_{(\kk\ub)_{(-;+)}} \mathscr{K}_- \otimes_{(\kk\db)_{(-;-)}} \lfb^* (\cpd_{\geq 3})
&\cong &
(\kk\db)_{(-;+)}(-, \mathbf{l}) ^\sharp  \otimes_{\kk \fb} \lfb^* (\cpd_{\geq 3})
\end{eqnarray*}
is isomorphic to the even  hairy graph complex with legs labelled by $\mathbf{l}$.
\end{THM}

\begin{rem*}
One can drop the restriction to $\cpd_{\geq 3}$. However, if $\cpd$ has a unit, the homology of the hairy graph complex is supported on $\mathbf{0}$, by the acyclicity result Theorem \ref{thm:acyclicity}. (This can also be seen as a consequence of the behaviour of the stabilization of $H_*^\mathsf{CE} (\cpd (V))$, by using Theorem \ref{thm:stabilization_trivial_action}.) In this case, if $\cpd$ is augmented, the natural solution is to restrict to the `augmentation ideal' $\overline{\cpd}$ of $\cpd$; the hairy graph homology can then be non-trivial in the presence of hairs. 
\end{rem*}

\begin{rem*}
A result such as Theorem \ref{THM:hairy} is to be expected. Indeed, graph complexes can be constructed  by using Getzler and Kapranov's Feynman transform \cite{MR1601666}. Stoll \cite{MR4541945} showed how the Feynman transform can be constructed using his results, which correspond to `connected graph' versions of the results here.
\end{rem*}

One advantage of working with hairy graphs is that one has the $(\kk \ub)_{(-;+)}$-structure to work with; this gives naturality with respect to $\mathbf{l}$  in Theorem \ref{THM:hairy}. (This is related to  considerations of Conant, Kassabov, and Vogtmann in \cite[Section 4]{MR3029423}, for example.) This allows us to consider the relationship between the $\ext$ and $\tor$ structures introduced above. A full analysis  requires the usage of the appropriate universal coefficients spectral sequences; a first approximation is given by the respective edge morphisms. 

For any $(\kk \ub)_{(-;-)}$-module $N$, one has the natural morphism
\begin{eqnarray}
\label{eqn:edge1_INTRO}
(\kk \ub)^\sharp _{(-;+)}
\otimes_{
(\kk \ub)_{(-;+)}
}
 \ext^*_{(\kk \db)_{(-;-)} } (\kk \fb , N) 
\rightarrow 
\tor_*^{(\kk \db)_{(-;-)} } (\kk \fb , N).
\end{eqnarray}  
Here both domain and codomain are torsion $(\kk \ub)_{(-;+)}$-modules;  the domain only depends on the `torsion-free' part of $\ext^*_{(\kk \db)_{(-;-)} } (\kk \fb , N) $ as a  $(\kk \ub)_{(-;+)}$-module.

This begs the question as to whether the torsion-free part of $\ext^*_{(\kk \db)_{(-;-)} } (\kk \fb , \lfb^* (\cpd))$ is ever non-trivial. As a first example, in Section \ref{sect:examples_alg_inv} we consider the case of $\cpd _{(\kk , \id)}$ (which corresponds to the initial cyclic operad with unit) and establish the following:

\begin{THM}
(Theorem \ref{thm:non-torsion}.)
For $2 \leq \ell \in \nat$ and $2n := \ell (\ell -1)$, the  cohomology given by the $\kk \sym_{2n}$-module
$$
S_{(\ell^{\ell -1})} \subseteq 
\ext^0_{(\kk \db)_{(-;-)} }(\kk \fb, \lfb^* (\cpd_{(\kk, \id)}) )(\mathbf{2n})
$$
is not torsion with respect to the  $(\kk \ub)_{(-;+)}$-module structure. 
\end{THM}

The second edge morphism should then correspond to the mate of (\ref{eqn:edge1_INTRO}): 
\begin{eqnarray}
\label{eqn:edge2_INTRO}
 \ext^*_{(\kk \db)_{(-;-)} } (\kk \fb , N) 
\rightarrow 
\hom_{(\kk \ub)_{(-;+)}} ((\kk \ub)_{(-;+)}^\sharp, 
\tor_*^{(\kk \db)_{(-;-)} } (\kk \fb , N)
).
\end{eqnarray}

Heuristically, in the case $N = \lfb^* (\cpd)$, these morphisms relate `unstable homology' (corresponding to $\ext^* _{(\kk \db)_{(-;-)}} (\kk \fb, \lfb^* (\cpd))$) with `stable homology' (corresponding to $\tor_* ^{(\kk \db)_{(-;-)}} (\kk \fb, \lfb^* (\cpd))$). 

There are also `odd' versions of the above results, using the $\kk \db$-module $\sfb^* (\cpd)$ in place of $\lfb^*(\cpd)$. These are presented in the body of the paper.

\begin{rem*}
The above results rely essentially upon the fact that the $\kk$-linear category $\kk \db$ (as well as its twisted variants) are Koszul $\kk$-linear categories. In the work of Stoll \cite{MR4541945}, this is subsumed in the fact that the Brauer properad is a Koszul properad. This is also related to Ward's  theorem \cite{MR4425832} that the operad encoding modular operads is Koszul. There are also shades of such results in the work of Hinich and Vaintrob \cite{MR1913297}, notably the relationship between cyclic operads and chord diagrams.
\end{rem*}

There is an analogous story for operads, using the Lie algebra of derivations of free algebras in place of the Conant-Vogtmann-Kontsevich Lie algebra. In this context, the idea of naturality was already present in \cite{MR4881588}, but without explicitly exploiting the appropriate forms of (twisted) downward walled Brauer categories. In the companion paper \cite{P_opd_wall}, the analogue for operads of the cyclic case considered in this paper is explained, also indicating the relationship with the work of Dotsenko \cite{MR4945404}. The relationship between the cyclic operad case and the operad case is considered in \cite{2026arXiv260413750P}.

\bigskip
\paragraph{\bf Organization:}
A considerable part of the paper is devoted to presenting background material. Whilst almost all of this is available in the literature in some shape or form, the author has preferred to give a self-contained presentation adapted to the applications at hand. 

Sections \ref{sect:fiord} and \ref{sect:twist} give the definitions and first properties of the twisted upward (and downward) Brauer categories. Section \ref{sect:torsion} reviews properties of torsion for modules over the upward variants (this material is not strictly-speaking essential, but helps explain some conceptual features). 

The long Section \ref{sect:koszul} serves to cover the Koszul property in detail. This has two purposes: to make  explicit the different Koszul complexes that we use, and to give the (co)homological interpretation of these in terms of $\ext$ and $\tor$ in modules over the relevant twisted $\kk$-linearization of a  Brauer category.

Section \ref{sect:forms} then explains how symplectic (respectively orthogonal) vector spaces enter the picture. This  yields our approximation to stabilization. This is applied in Section \ref{sect:forms_koszul} to the Koszul complexes that interest us. 

The relationship between cyclic operads and modules over twisted Brauer categories is explained in Section \ref{sect:cyclic_brauer}. This is the fundamental ingredient that allows the theory presented in the previous sections to be applied.  Section \ref{sect:kontsevich_lie} then gives a  explanation of the relationship between the Koszul complex (for the even case) and the Chevalley-Eilenberg complex of the Lie algebra constructed by Conant and Vogtmann. Section \ref{sect:acyclicity} shows  that hairy graph homology in the presence of a unit is trivial in the presence of hairs. The theory is illustrated further in Section \ref{sect:examples_alg_inv}, where the case of algebras with involution is considered. 

In Section \ref{sect:complexes}, the relationship with hairy graph complexes is explained.

\bigskip
\paragraph{\bf Acknowledgement:}
This project was inspired in part by the work of Vladimir Dotsenko, and the author is grateful for his interest. 

The author would also like to thank Kazuo Habiro and Mai Katada for their interest, for informing him of their related work in progress, and also for useful comments (and corrections) on an earlier version of this work.

The author is also grateful to José S\~ao Jo\~ao for his interest and for comments which helped clarify the behaviour of the homology of (generalized) hairy graph complexes in the presence of a unit. 

Part of this work was presented by the author at the workshop \href{https://sites.google.com/g.ecc.u-tokyo.ac.jp/alg-approach-mcg/}{{\em Algebraic approaches to mapping class groups of surfaces}} at the University of Tokyo in May 2025. The author is very grateful to have been invited to speak there and wishes to acknowledge the financial support for his visit provided by  Professor Nariya Kawazumi's grant.

\bigskip
\paragraph{\bf Financial Support:}
This research was supported in part by the ANR project KAsH  \href{https://anr.fr/Projet-ANR-25-CE40-2861}{ANR-25-CE40-2861}.

\tableofcontents


\section{The categories $\fiord$, $\dubord$, and $\ub$}
\label{sect:fiord}

The purpose of this section is to recall the definition of the upward Brauer category $\ub$ (and its opposite, $\db$, the downward Brauer category). The presentation is designed to  facilitate the introduction of the twisted variants of their respective $\kk$-linearizations $\kk \ub$ and $\kk \db$ in Section \ref{sect:twist}, where $\kk$ is a unital commutative ring.

\subsection{Preliminaries}

The category of finite sets (and all maps) is denoted $\fin$ and its wide subcategories defined by the class of injective (respectively bijective) maps by $\finj$ (resp. $\fb$). The disjoint union of sets is the coproduct of $\fin$; it restricts to a symmetric monoidal structure on both $\finj$ and $\fb$. The category $\fin$ has skeleton given by $\{ \n  \mid n \in \nat\}$, where $\n := \{ 1, \ldots, n \}$ (understood to be $\emptyset$ for $n=0$). This yields skeleta for $\finj$ and for $\fb$.

\subsection{Introducing $\fiord$}

The category $\fiord$ is a variant of $\finj$ that is described informally as follows. Objects are finite sets; morphisms are injective maps together with a total order on the complement of the image of this injection; composition is given by that of $\finj$ together the monoidal structure on ordered sets given by disjoint union.  There is a functor 
 $ 
\fiord \rightarrow \finj
$  
that is the identity on objects; on morphisms, it forgets the total order on the complement.

The following is the formal definition of $\fiord$:

\begin{defn}
\label{defn:fiord}
Let $\fiord$ be the category with objects finite sets. Morphisms are given by 
\[
\fiord (X,Y) = 
\left\{ 
\begin{array}{ll}
\emptyset & |X| >  |Y| \\
\fb (X \amalg \m , Y) & |Y|-|X|= m \in \nat.
\end{array}
\right.
\]

Given $f \in \fiord (X,Y)$ represented by the bijection $f' : X \amalg \m \stackrel{\cong}{\rightarrow} Y$ and $g \in \fiord (Y, Z)$ represented by $g' : Y \amalg \n \stackrel{\cong}{\rightarrow} Z$, the composite $g\circ f$ is represented by the composite
\[
X \amalg (\mathbf{m+n})
\cong
X \amalg \m \amalg \n 
\stackrel{f' \amalg \id_\n}{\longrightarrow}
Y \amalg \n 
\stackrel{g'}{\rightarrow} 
Z.
\]
\end{defn}

\begin{rem}
\label{rem:fiord_graphical}
The following graphical presentation of morphisms can be useful and prepares the way for the case of the upward Brauer categories.
For example, consider $\fiord (\mathbf{4}, \mathbf{8})$, using the canonical orders on $\mathbf{4}$, $\mathbf{8}$ (viewed as the ordering from left to right on the top and bottom edges in the diagram). The following diagram represents a morphism: 

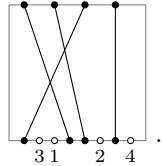
\begin{center}
\begin{tikzpicture}[scale = .15]
\draw [gray] (0,0) -- (0,-9) -- (9,-9) -- (9,0) -- (0,0); 
\draw (1,-9) -- (5,0); 
\draw (4,-9) -- (1,0); 
\draw (5,-9) -- (3,0); 
\draw (7,-9)-- (7,0);

\draw [fill=black] (1,-9) circle (0.2);
\draw [fill=white] (2,-9) circle (0.2);
\draw [fill=white] (3,-9) circle (0.2);
\draw [fill=black] (4,-9) circle (0.2);
\draw [fill=black] (1,0) circle (0.2);
\draw [fill=black] (3,0) circle (0.2);
\draw [fill=white]  (6,-9) circle (0.2);
\draw [fill=black]  (7,-9) circle (0.2);
\draw [fill=white]  (8,-9) circle (0.2);
\draw [fill=black]  (7,0) circle (0.2);
\draw [fill=black] (5,-9) circle (0.2);
\draw [fill=black] (5,0) circle (0.2);
\node [right] at (9,-9) {.};
    
\node [below] at (2,-9) {$\scriptscriptstyle{3}$};
\node [below] at (3,-9) {$\scriptscriptstyle{1}$};
\node [below] at (6,-9) {$\scriptscriptstyle{2}$}; 
\node [below] at (8,-9) {$\scriptscriptstyle{4}$};    
\end{tikzpicture}
\end{center}
Here, the injection is represented by the black lines (for example $1\mapsto 4$ and $4 \mapsto 7$); the complement of the image is represented by the white nodes  and the total order is indicated by the numbering of these. 

Composition is by stacking diagrams vertically, carrying through the total orders. For example, consider the morphism in $\fiord (\mathbf{2}, \mathbf{4})$ represented by 
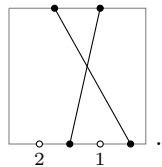
\begin{center}
\begin{tikzpicture}[scale = .15]
\draw [gray] (0,0) -- (0,-9) -- (9,-9) -- (9,0) -- (0,0); 
\draw (8,-9) -- (3,0); 
\draw (4,-9) -- (6,0);

\draw [fill=white] (2,-9) circle (0.2);
\draw [fill=black] (4,-9) circle (0.2);
\draw [fill=white]  (6,-9) circle (0.2);
\draw [fill=black]  (8,-9) circle (0.2);

\draw [fill=black] (3,0) circle (0.2);
\draw [fill=black] (6,0) circle (0.2);

\node [right] at (9,-9) {.};
    
\node [below] at (2,-9) {$\scriptscriptstyle{2}$};
\node [below] at (6,-9) {$\scriptscriptstyle{1}$};
\end{tikzpicture}
\end{center}

Composing the above two morphisms gives the morphism in $\fiord (\mathbf{2}, \mathbf{8})$ represented by the diagram 
\begin{center}
\begin{tikzpicture}[scale = .15]
\draw [gray] (0,0) -- (0,-9) -- (9,-9) -- (9,0) -- (0,0);  
\draw (7,-9) -- (3,0); 
\draw (5,-9) -- (6,0); 

\draw [fill=black] (6,0) circle (0.2);
\draw [fill=black] (3,0) circle (0.2);

\draw [fill=white] (1,-9) circle (0.2);
\draw [fill=white] (2,-9) circle (0.2);
\draw [fill=white] (3,-9) circle (0.2);
\draw [fill=white] (4,-9) circle (0.2);
\draw [fill=black] (5,-9) circle (0.2);
\draw [fill=white]  (6,-9) circle (0.2);
\draw [fill=black]  (7,-9) circle (0.2);
\draw [fill=white]  (8,-9) circle (0.2);

\node [right] at (9,-9) {.};
    
\node [below] at (1, -9){$\scriptscriptstyle{1}$} ;
\node [below] at (2,-9) {$\scriptscriptstyle{5}$};
\node [below] at (3,-9) {$\scriptscriptstyle{3}$};
\node [below] at (4, -9) {$\scriptscriptstyle{2}$};
\node [below] at (6,-9) {$\scriptscriptstyle{4}$}; 
\node [below] at (8,-9) {$\scriptscriptstyle{6}$};    
\end{tikzpicture}
\end{center}
\end{rem}

\begin{rem}
\label{rem:finj_=_fiord_mod_sym}
For finite sets $X$, $Y$ such that $|Y|-|X|=n \in \nat$, by definition one has $\fiord (X,Y) = \fb (X \amalg \n, Y)$, hence there is a natural right action of the symmetric group $\sym_n$ on $\fiord (X,Y)$ induced by the canonical (left) action on $\n$. This action is free and  there is an isomorphism of sets 
\[
\fiord (X, Y)/ \sym_n \cong \finj( X,Y).
\]
\end{rem}

The following is immediate:

\begin{lem}
\label{lem:fiord_vs_finj}
For finite sets $X$, $Y$, the map 
$
\fiord (X,Y) 
\rightarrow 
\finj (X,Y)
$ 
given by the functor $\fiord \rightarrow \finj$ is an isomorphism if and  only if $|Y|-|X| \leq 1$. 
 In particular, if $|Y|-|X| \in \{0, 1\}$, a morphism of $\fiord (X,Y)$ is uniquely determined by its underlying injective map.
\end{lem}

This lemma is used implicitly to simplify the statement of the following Proposition.

\begin{prop}
\label{prop:factor_fiord}
Suppose that $X$, $Y$ are finite sets with $|Y|-|X| = n \in \nat$. A morphism $f \in \fiord (X,Y)$ determined by $f : X \hookrightarrow Y$ and the total order $y_1 \prec y_2 \prec \ldots \prec y_n$ on $Y \backslash f(X)$ 
admits a canonical factorization in $\fiord$: 
\[
X 
\stackrel{\cong, f}{\longrightarrow}
f(X) 
\subset 
f(X) \amalg \{y_1\} 
\subset 
 f(X) \amalg \{y_1, y_2 \}
 \subset 
 \ldots 
 \subset Y = f(X) \amalg \{y_1, y_2, \ldots, y_n \},  
\]
where the inclusions correspond to the flag of subsets of $Y\backslash f(X)$ given by the total order.

In particular, the morphisms of $\fiord$ are generated under composition by the bijections (arising from $\fb$)  and by the canonical inclusions $\m \subset \mathbf{m+1}$ (considered as a morphism in  $\fiord (\m,\mathbf{m+1})$), for $m \in \nat$.
\end{prop}

We record the following: 

\begin{prop}
\label{prop:fiord_EI}
The category $\fiord$ is an EI-category (every endomorphism is an isomorphism) and the maximal subgroupoid identifies with $\fb$. Moreover, the forgetful functor $\fiord \rightarrow \finj$ fits into the commutative diagram of functors that are the identity on objects:
\[
\xymatrix{
\fb 
\ar@{^(->}[r]
\ar@/_1pc/@{^(->}[rr]
&
\fiord
\ar[r]
&
\finj.
}
\]
\end{prop}

\subsection{The upward (and downward) Brauer category}
\label{subsect:ub_db}

As a first step towards introducing the upward Brauer category, 
we restrict to the wide subcategory of $\fiord$ in which we only allow morphisms between sets of the same parity. The notation $\dubord$ reflects the fact that is a `directed' and `ordered' version of the upward Brauer category (see \cite{2026arXiv260413750P} for more context).

\begin{defn}
\label{defn:fiordev}
Let $\dubord$ be the wide subcategory defined by 
\[
\dubord (X,Y) = 
\left\{
\begin{array}{ll}
\emptyset & |X|\not \equiv |Y| \mod 2 \\
\fiord (X, Y) & \mbox{otherwise.}
\end{array}
\right.
\]
\end{defn}

\begin{rem}
\ 
\begin{enumerate}
\item 
Clearly $\dubord$ is a sub EI-category of $\fiord$ and the maximal subgroupoid identifies with $\fb$. 
\item 
The category $\dubord$ has two full subcategories, corresponding to the objects with even (respectively odd) cardinality. Moreover, $\dubord$ is the disjoint union of these two subcategories, since there are no morphisms between objects in different subcategories.
\item 
There is a natural `$\nat$-grading' of morphisms of $\dubord$. Namely, assuming $|Y| - |X|\equiv 0 \mod 2$, $\dubord (X, Y)$ is taken to have degree $\frac{1}{2} (|Y| - |X|) \in \nat$. Thus the degree $0$ morphisms correspond to the maximal subgroupoid $\fb \subset \dubord$. Analogously to Proposition \ref{prop:fiord_EI}, the morphisms of $\dubord$ are generated under composition by the degree zero and the degree one morphisms.
\end{enumerate}
\end{rem}

Restricting to complements with even parity allows us to use chord diagrams:

\begin{rem}
\label{rem:introducing_chords}
For $n=2t$ an even natural number, there is a bijection 
$\mathbf{2} \times \mathbf{t}  \stackrel{\cong}{\rightarrow} \mathbf{n}$ given by  
$(1, i)  \mapsto 2i -1$ and $(2, i )  \mapsto  2i$. Using this bijection, a total ordering of $\mathbf{n}$ is equivalent to an ordered set $\{ (x_1^i, x_2^i) \mid i \in \mathbf{t} \}$ of ordered pairs of elements of $\n$ such that $\n = \amalg_i \{ x_1^i, x_2^i \}$. 

Thus a total ordering can be represented by a decorated chord diagram, in which each chord is directed (aka.  oriented) and the chords are ordered;  the underlying undecorated chord diagram is given by  forgetting the directions of the chords and their order.  

For example, for $n=6$, the following is an undecorated chord diagram:

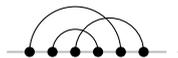
\begin{center}
\begin{tikzpicture}[scale = .3]

\begin{scope}
    \clip (0,-7) rectangle (7,-4.8);
    \draw (3,-7) circle(2);
    \draw (3,-7) circle (1); 
    \draw (4.5, -7) circle (1.5);
\end{scope}

\draw [lightgray, thick] (0,-7) -- (7,-7);
\draw [fill=black] (1,-7) circle (0.1);
\draw [fill=black] (2,-7) circle (0.1);
\draw [fill=black] (3,-7) circle (0.1);
\draw [fill=black] (4,-7) circle (0.1);
\draw [fill=black] (5,-7) circle (0.1);
\draw [fill=black] (6,-7) circle (0.1);
\node [right] at (7,-7) {.};
\end{tikzpicture}
\end{center}
We choose the following decoration: all chords are oriented clockwise and the order of the chords is given by the order of the left node of each chord (with respect to the usual order $<$ of $\mathbf{6}$). With this decoration, the above diagram encodes the total order $\prec$ on $\mathbf{6}$ given by  $1 \prec 5 \prec 2 \prec 4 \prec 3 \prec 6$.

The canonical order on $\mathbf{6}$ corresponds to the chord diagram 
\begin{center}
\begin{tikzpicture}[scale = .3]
\begin{scope}
    \clip (0,-7) rectangle (7,-6.3);
    \draw (1.5,-7) circle(.5);
    \draw (3.5,-7) circle (.5); 
    \draw (5.5, -7) circle (.5);
\end{scope}

\draw [lightgray, thick] (0,-7) -- (7,-7);
\draw [fill=black] (1,-7) circle (0.1);
\draw [fill=black] (2,-7) circle (0.1);
\draw [fill=black] (3,-7) circle (0.1);
\draw [fill=black] (4,-7) circle (0.1);
\draw [fill=black] (5,-7) circle (0.1);
\draw [fill=black] (6,-7) circle (0.1);
\node [right] at (7,-7) {,};
\end{tikzpicture}
\end{center}
with decoration with chords oriented clockwise and ordered from left to right.

The free transitive action of $\sym_n$  induces a free transitive action on the set of decorated chord diagrams. If one forgets the decorations, the action is no longer free: the stabilizer of a given (undecorated) chord diagram is isomorphic to the wreath product $\sym_2 \wr \sym_t \cong \sym_2 ^{\times t} \rtimes \sym_t$.  The set of decorated chord diagrams with given underlying undecorated diagram is thus a free transitive $\sym_2 \wr \sym_t$-set.
\end{rem}


Recall from Remark \ref{rem:finj_=_fiord_mod_sym} that, if $X$ and $Y$ are finite sets with $|Y|-|X| = 2t$, for $t \in \nat$, then $\dubord (X,Y) = \fb (X \amalg \mathbf{2t}, Y)$ is equipped with a natural $\sym_{2t}$-action; thus the subgroup $\sym_2 \wr \sym_t \subset \sym_{2t}$ acts on the right on $\dubord (X,Y)$.

\begin{defn}
\label{defn:ub}
The upward Brauer category $\ub$ has objects finite sets and morphisms given by 
\[
\ub (X, Y) := \left\{
\begin{array}{ll}
\dubord (X,Y)/ \sym_2 \wr \sym_t & |Y|-|X|= 2t,  t \in \nat
\\
\emptyset & \mbox{otherwise.}  
\end{array}
\right.
\]
Composition is induced by that of $\dubord$.

The downward Brauer category $\db$ is the opposite category $\db := \ub \op$. 
\end{defn}

\begin{rem}
The graphical representation of morphisms of $\fiord$ leads to one for $\ub$, by using (undecorated) chord diagrams. For example, the following diagram represents an element of 
 $\ub (\mathbf{4}, \mathbf{8})$:

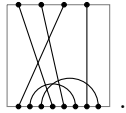
\begin{center}
\begin{tikzpicture}[scale = .15]
\begin{scope}
    \clip (1.5,-9) rectangle (6.5,-6);
    \draw (4,-9) circle(2);
\end{scope}

\begin{scope}
    \clip (2.5,-9) rectangle (8.5,-6);
    \draw (5.5,-9) circle(2.5);
\end{scope}

\draw [gray] (0,0) -- (0,-9) -- (9,-9) -- (9,0) -- (0,0); 
\draw (1,-9) -- (5,0); 
\draw (4,-9) -- (1,0); 
\draw (5,-9) -- (3,0); 
\draw (7,-9)-- (7,0);

\draw [fill=black] (1,-9) circle (0.2);
\draw [fill=black] (2,-9) circle (0.2);
\draw [fill=black] (3,-9) circle (0.2);
\draw [fill=black] (4,-9) circle (0.2);
\draw [fill=black] (1,0) circle (0.2);
\draw [fill=black] (3,0) circle (0.2);
\draw [fill=black]  (6,-9) circle (0.2);
\draw [fill=black]  (7,-9) circle (0.2);
\draw [fill=black]  (8,-9) circle (0.2);
\draw [fill=black]  (7,0) circle (0.2);
\draw [fill=black] (5,-9) circle (0.2);
\draw [fill=black] (5,0) circle (0.2);
    \node [right] at (9,-9) {.};
\end{tikzpicture}
\end{center} 
Composition is again given by stacking diagrams.
\end{rem}

The following is clear:

\begin{prop}
\label{prop:properties_ub}
\ 
\begin{enumerate}
\item
The category $\ub$ is an EI-category with maximal subgroupoid that identifies with $\fb$.
\item 
There is a commutative diagram of  functors that are the identity on objects:
\[
\xymatrix{
\dubord 
\ar[r]
\ar@/_1pc/[rr]
&
\ub 
\ar[r]
&
\finj.
}
\]
On morphisms, for $|Y|= |X|+2t$, $t \in \nat$, this corresponds to the canonical surjections 
\[
\dubord (X,Y) 
\twoheadrightarrow 
\dubord (X,Y) / \sym_2 \wr \sym_t 
\twoheadrightarrow 
\dubord (X,Y) / \sym_{2t}.
\]
\item 
The $\nat$-grading of morphisms of $\dubord$ induces one of morphisms of $\ub$. Under composition, the morphisms of $\ub$ are generated by the degree zero  and the degree one morphisms.
\end{enumerate}
\end{prop}

\section{Twisted $\kk$-linearizations}
\label{sect:twist}

The main purpose of this section is to introduce the twisted forms $(\kk \ub)_{(\pm; \mp)}$ of the $\kk$-linearization of the category $\ub$. We also establish some basic properties that are required later. No claim to originality is made; for example, twisted forms of the upward Brauer category occur in \cite{MR3376738}.

Throughout this section, $\kk$ is a unital commutative ring.

\subsection{$\kk$-linear categories}

We work with (essentially) small $\kk$-linear categories; these can be viewed as a multi-object generalization of $\kk$-algebras. If $\cala$ is such a category, the category $\cala\dash\modules$ of $\cala$-modules is the category of $\kk$-linear functors from $\cala$ to $\kmod$, the category of $\kk$-modules. This corresponds to the category of `left modules'; the category of `right modules' is $\cala\op\dash\modules$, where the opposite category $\cala\op$ is equipped with the $\kk$-linear structure inherited from $\cala$.

\begin{exam}
\label{exam:projective_A_mod}
Suppose that $\cala$ is an essentially small $\kk$-linear category and that $X$ is an object of $\cala$. Then $\cala (X, -)$ has a natural $\cala$-module structure. Moreover, it is projective in $\cala \dash\modules$; it corepresents evaluation on $X$. As $X$ runs over a set of isomorphism class representatives of objects of $\cala$, these form a set of projective generators of $\cala\dash\modules$. Likewise, $\cala (-, X)$ has a natural right $\cala$-module structure; it is projective in $\cala\op\dash\modules$.
\end{exam}

For later usage, we recall:

\begin{defn}
\label{defn:fg_proj}
In a $\kk$-linear category $\cala$, an object is said to be finitely-generated projective if it is a direct summand of a finite direct sum $\bigoplus_ i \cala (X_i, -)$ for some objects $X_i$.
\end{defn}

\begin{rem}
\label{rem:kk-linearization}
If $\calc$ is an essentially small category, then its $\kk$-linearization $\kk \calc$ is (tautologically) a $\kk$-linear category. The category $\kk \calc \dash\modules$ is equivalent to the category of functors from $\calc$ to $\kmod$, by the universal property of $\kk$-linearization (later, this category will also be denoted $\f (\calc)$). 
\end{rem}

If $M$ is a right $\cala$-module and $N$ is a left $\cala$-module, one has $M \otimes_\cala N$,  the multi-object generalization of the tensor product over a $\kk$-algebra. For $X$ an object of $\cala$, there are natural isomorphisms:
\begin{eqnarray*}
\cala (-, X) \otimes_\cala N & \cong & N (X) \\
M \otimes_\cala \cala (X, -) & \cong & M (X).
\end{eqnarray*}

\begin{rem}
For $\calc$ an essentially small category, $\kk \calc\dash\modules$ is equipped with a symmetric monoidal structure induced by the tensor product $\otimes$ on $\kmod$. Namely, for two objects $F$, $G$ (considered as functors from $\calc$ to $\kmod$), their tensor product is given by $(F \otimes G) (X):= F(X) \otimes G(X)$, with morphisms acting diagonally. The unit is the constant functor $\kk$. 

This does not hold in general for a $\kk$-linear category $\cala$ in place of $\kk \calc$. (Indeed,  there need not even be a `constant' $\cala$-module playing the rôle of $\kk$.) 
 There is, however, an external tensor product. If $\cala$ and $\calb$ are (essentially) small $\kk$-linear categories, one can form  their tensor product $\cala \otimes \calb$, generalizing the tensor product of algebras; this has set of objects $\ob \cala \times \ob \calb$ and morphism modules formed by tensor products. For $M$ an $\cala$-module and $N$ a $\calb$-module, $M \otimes N$ is naturally an $\cala \otimes \calb$-module with values given by $(M \otimes N) (A, B) = M(A) \otimes N (B)$, for $(A,B) \in \ob \cala \times \ob \calb$.
\end{rem}

\begin{rem}
\label{rem:cala-bimodules}
The tensor product of $\kk$-linear categories allows us to introduce the notion of an $\cala$-bimodule,  namely an $\cala \otimes \cala\op$-module, as usual. 
 For example, $\cala (-, -)$ has a canonical $\cala$-bimodule structure, as does $\cala\op (-,-)$; these are essentially the same structures, using the equality $\cala (X, Y) = \cala\op (Y, X)$. 
\end{rem}

\subsection{Twisted versions of $\kk \ub$ and of $\kk \db$}

We start from a fundamental observation that follows from the construction of $\ub$ as given in Definition \ref{defn:ub}. For finite sets $X, Y$ such that $|Y|- |X| = 2t \in \nat$, there is a natural isomorphism of $\kk$-modules:
\[
\kk \ub (X, Y) 
\cong 
\kk \dubord (X,Y) \otimes_{\kk \sym_2 \wr \sym_t} 
\kk,
\]
using the natural right action of $\sym_2 \wr \sym_t$ on $\dubord (X,Y)$, where $\kk$ is considered as a trivial $\sym_2 \wr \sym_t$-module. 

This allows us to introduce {\em twisted} versions of $\kk \ub$ by replacing the trivial representation $\kk$ by other appropriate rank one representations of $\sym_2 \wr \sym_t$. These representations could be introduced using generalities on representations of wreath products (see \cite[Chapter 4]{MR644144}, for example). Instead we give a direct approach exploiting the inclusion $\sym_2 \wr \sym_t \subset \sym_{2t}$ and the (split) surjection $\sym_2 \wr \sym_t \cong \sym_2^{\times t} \rtimes \sym_t \twoheadrightarrow \sym_t$.

\begin{defn}
\label{defn:twisting_representations}
Define the following representations of $\sym_2 \wr \sym_t$ with underlying $\kk$-module free of rank one:
\begin{enumerate}
\item 
$\kk_{(+;+)}:= \kk$, the trivial representation; 
\item 
$\kk_{(-;+)}$, the restriction of the sign representation $\sgn_{2t}$ along $\sym_2 \wr \sym_t \subset \sym_{2t}$; 
\item 
$\kk_{(+;-)}$, the restriction of the sign representation $\sgn_t$
along $\sym_2 \wr \sym_t \rightarrow \sym_t$;
\item 
$\kk_{(-;-)}$ the tensor product $\kk _{(+;-)} \otimes \kk _{(-;+)}$ (for the diagonal action of $\sym_2 \wr \sym_t$).
\end{enumerate}
The generic notation $(\pm; \mp)$  is used  to indicate an element of $\{ (+;+), (+;-) , (-;+) , (-;-) \}$.
\end{defn} 

\begin{rem}
If the characteristic of the ring $\kk$ is not two and $t>1$, the four representations $\kk_{(\pm; \mp)}$ are pairwise non-isomorphic. Working with isomorphism classes of representations,  these form a group under $\otimes$ that is isomorphic to the Klein four group, $\zed/2 \times \zed/2$, with identity element the isomorphism class of $\kk_{(+;+)} = \kk$.
 For $t=1$, there are isomorphisms $\kk _{(+;+)} \cong \kk _{(+; -)}$ and $\kk_{(-;+)} \cong \kk _{(-;-)}$, and the corresponding group is $\zed/2$.
\end{rem}

\begin{defn}
\label{defn:kkub_twist}
Let $(\kk \ub)_{(\pm; \mp)}$ denote the $\kk$-linear category with objects finite sets and with morphisms 
\[
(\kk\ub)_{(\pm; \mp)} (X,Y) 
:= 
\left\{
\begin{array}{ll}
0 & |Y|-|X| \not \in 2\nat \\
\kk \dubord (X,Y) \otimes_{\kk\sym_2\wr\sym_t} \kk_{(\pm; \mp)} &
|Y|-|X|=2t, t\in \nat.
\end{array}
\right.
\]
Composition is induced by that of $\kk \dubord$.
\end{defn}

\begin{rem}
The fact that the composition of $\kk \dubord$ passes to a well-defined composition on $(\kk \ub)_{(\pm; \mp)}$ relies on the fact that the sign representation behaves well upon restriction to Young subgroups. Namely, for $m, n \in \nat$, the restriction $\sgn_{m+n}\downarrow^{\sym_{m+n}}_{\sym_m \times \sym_n}$ is isomorphic to the representation $\sgn_m \boxtimes \sgn_n$ of $\sym_m \times \sym_n\subset \sym_{m+n}$.
\end{rem}

\begin{exam}
The $\kk$-linear category $(\kk \ub)_{(+;+)}$ is isomorphic to $\kk \ub$. 
\end{exam}

An immediate consequence of the definition is the following:

\begin{lem}
\label{lem:bimodule_kub_twisted}
For $m , t \in \nat$, setting $n:= m + 2t$, there is an isomorphism of $\sym_n \times \sym_m\op$-modules 
\[
(\kk \ub)_{(\pm; \mp)} (\m, \n) 
\cong 
\kk \sym_n \otimes_{\kk\sym_2 \wr \sym_t} \kk _{(\pm; \mp)} ,
\] 
where the codomain is considered as a quotient of the bimodule $\kk \sym_n \downarrow _{\sym_m}^{\sym_n}$, restricting the right module structure using the restriction along $\sym_m \subset \sym_m \times \sym_2 \wr \sym_t \subset \sym_n$.
\end{lem}

\begin{prop}
\label{prop:twisted_subcat_fiordev_modules}
There is a full functor 
$ 
\kk \dubord 
\rightarrow 
(\kk\ub)_{(\pm;\mp)} 
$ 
that is the identity on objects. Hence, the category $(\kk \ub)_{(\pm;\mp)} \dash\modules$ identifies as a full subcategory of $\kk \dubord\dash\modules$.
\end{prop}

\begin{proof}
The first statement follows from the definition of $(\kk \ub)_{(\pm; \mp)}$; the second is a formal consequence. 
\end{proof}

The underlying $\kk$-module of the morphism spaces of the categories $(\kk \ub)_{(\pm;\mp)}$ can be identified by exploiting the canonical total order on $\n$, for $n \in \nat$. We `decorate'  chord diagrams  as in Remark \ref{rem:introducing_chords}.

\begin{lem}
\label{lem:basis_twisted_kub}
For $m,n \in \nat$ such that $n -m = 2t \in \nat$, there is a section of the projection $\dubord (\m, \n) \rightarrow \ub (\m , \n)$ induced by decorating the chord diagrams as follows: 
\begin{enumerate}
\item 
chords are oriented clockwise; 
\item 
chords are ordered  by the ordering of their left hand nodes inherited from the total order of $\n$. 
\end{enumerate}

On composing with the canonical projection $\kk \dubord (\m, \n) \twoheadrightarrow (\kk \ub)_{(\pm; \mp)} (\m, \n)$ this induces an isomorphism of $\kk$-modules
\begin{eqnarray}
\label{eqn:kub_twist_basis}
\kk \ub (\m, \n)
\stackrel{\cong}{\rightarrow} 
(\kk \ub)_{(\pm; \mp)} (\m, \n).
\end{eqnarray}
In particular, $(\kk \ub)_{(\pm; \mp)} (\m, \n)$ is a  finite rank free $\kk$-module with basis given by this construction. 
\end{lem}

In order to relate modules over the different categories $(\kk \ub)_{(\pm;\mp)}$, we use the following object of $\kk \dubord\dash\modules$:

\begin{lem}
\label{lem:sgn_fiord}
There is a $\fiord$-module $\sgn$ such that: 
\begin{enumerate}
\item 
the restriction $\sgn\downarrow^{\fiord}_\fb$ is given by $X \mapsto \Lambda^{|X|} (\kk X)$ (the usual sign representation);
\item 
for any $n \in \nat$, the inclusion $\n \subset \mathbf{n+1}$ acts via 
$ 
- \wedge [n+1] \colon \Lambda^n (\kk \n) \rightarrow \Lambda^{n+1} (\kk (\mathbf{n+1})).
$ 
\end{enumerate} 
By restriction, one obtains a $\kk \dubord$-module again denoted by $\sgn$.
\end{lem}

\begin{proof}
This is a direct verification, using Proposition \ref{prop:factor_fiord} to reduce to the given morphisms.
\end{proof}

Clearly one has $\sgn \otimes \sgn \cong \kk$, the constant module, in $\kk \dubord$-modules. Hence, using the tensor product of $\kk \dubord \dash\modules$, we have the involutive functor 
\[
\sgn \otimes - 
\colon 
\kk \dubord \dash\modules
\rightarrow 
\kk \dubord \dash\modules.
\]

\begin{prop}
\label{prop:sgn_tensor_modules_categories}
The functor $\sgn \otimes -$ on $\kk \dubord$-modules restricts to equivalences of categories 
\begin{eqnarray*}
\kk \ub \dash \modules & \cong & (\kk \ub)_{(-;+)}\dash\modules \\
(\kk \ub)_{(+;-) } \dash \modules & \cong & (\kk \ub)_{(-;-)}\dash \modules, 
\end{eqnarray*}
where these are viewed as full subcategories of $\dubord \dash \modules$ via Proposition \ref{prop:twisted_subcat_fiordev_modules}.
\end{prop}

\begin{rem}
Analogous results are given  in \cite[Section 4]{MR3376738}, for example.
\end{rem}
\subsection{Augmentation, grading, and more}

We have the following:

\begin{lem}
\label{lem:unit_fiordev}
The inclusion $\fb \subset \dubord$ induces a $\kk$-linear embedding 
$ 
\kk \fb \hookrightarrow (\kk \ub)_{(\pm; \mp)}
$ 
that we consider as a unit map.
\end{lem}

In fact, the $\nat$-grading of morphisms of $\kk \dubord$ passes to $(\kk \ub)_{(\pm; \mp)}$ and the unit map corresponds to the inclusion of the wide $\kk$-linear subcategory of morphisms of degree zero. Projection onto the degree zero morphisms thus induces an augmentation 
\[
(\kk \ub)_{(\pm;\mp)} 
\rightarrow 
\kk \fb 
\]
that is the identity on objects. 

The notion of a homogeneous quadratic $\kk$-linear category is the multi-object generalization of that of a homogeneous quadratic algebra over a semisimple ring, as defined in \cite{MR2177131}, for example.

The following is well-known:

\begin{prop}
\label{prop:kub_twist_homogeneous_quadratic}
The  category $(\kk \ub)_{(\pm; \mp)}$ is a homogeneous quadratic $\kk$-linear category over $\kk \fb$:
\begin{enumerate}
\item 
morphisms are generated  by the $\kk\fb$-bimodule of degree one morphisms; 
\item 
relations are generated by degree two relations.
\end{enumerate}
\end{prop}

\begin{proof}
We have already established that $(\kk \ub)_{(\pm; \mp)}$ is $\nat$-graded; that  morphisms are generated by degree zero (corresponding to $\kk \fb$) and degree one is a consequence of Proposition \ref{prop:factor_fiord} (compare Proposition \ref{prop:properties_ub}, which implies the result for the untwisted $\kk \ub$). It remains to check that the relations are homogeneous quadratic, generated by the relations in degree two. This comes down to the fact that the symmetric groups are generated by `adjacent' transpositions.
\end{proof}

We also have the following projectivity statement.

\begin{prop}
\label{prop:morphisms_right_free}
For any finite set $X$, $(\kk \ub) _{(\pm; \mp)} (-, X)$ is free as a $(\kk \fb)\op$-module. Moreover, 
\begin{enumerate}
\item 
for a finite set $V$, $(\kk \ub) _{(\pm; \mp)} (V, X)$ is a finite rank free $\kk \aut (V)\op$-module; 
\item 
there are only finitely many isomorphism classes of sets $V$ such that $(\kk \ub) _{(\pm; \mp)} (V, X)$ is non-zero.  
\end{enumerate}
\end{prop}

\begin{proof}
Consider finite sets $V, X$ and choose a total order for $X$. Then one has the isomorphism of $\kk$-modules 
\[
\kk \ub (V,X) 
\stackrel{\cong}{\rightarrow}
(\kk \ub) _{(\pm: \mp)} (V, X)
\]
provided by Lemma \ref{lem:basis_twisted_kub} (this only requires a total order on $X$). This is an isomorphism of $\kk \aut (V)\op$-modules, since precomposing by an automorphism of $V$ does not affect the complement of a map, so no signs arise. 

Thus, it suffices to prove the result for $\kk \ub (-, X)$. This follows from the  fact that $\ub (V, X)$ is a finitely-generated free $\aut(V)\op$-set and is empty if $|V|> |X|$.   
\end{proof}

\subsection{The transpose structure}
\label{subsect:transpose}

Write $(-)^\sharp$ for the duality functor on $\kk$-modules given by $\hom_\kk (-, \kk)$. This defines a functor $\kmod\op \rightarrow \kmod$; it becomes an equivalence of categories when restricted to the full subcategory of $\kmod$ with objects that are  finite rank free $\kk$-modules.

\begin{rem}
\label{rem:duality_permutation_modules}
The above can be made more explicit when a basis of the free module is given. Namely, for $S$ a set, the dual of $\kk S$, $(\kk S)^\sharp$, identifies canonically with $\kk ^S := \mathrm{Map} (S, \kk)$. If $S$ is a finite set, then the dual basis gives the  isomorphism
\begin{eqnarray}
\label{eqn:duality_permutation_modules}
 \kk S
\cong
(\kk S) ^\sharp 
\end{eqnarray}
that sends the generator $[s]$, $s \in S$, to the corresponding element $\eta_s$ of the dual basis. 

If a group $G$ acts on the left on $S$, the  induced right action  on $(\kk S)^\sharp$ is given by $(\eta_s)g = \eta_{g^{-1} s}$. Considering $(\kk S)^\sharp$ as a left $G$-module, this implies that (\ref{eqn:duality_permutation_modules}) is an isomorphism of $G$-modules. 
\end{rem}

Now, for $\cala$ a $\kk$-linear category,  $(-)^\sharp$ induces a functor 
\[
(\cala \dash \modules )\op 
\rightarrow 
\cala\op \dash \modules. 
\]
Since $(\cala \otimes \cala\op)\op$ is canonically isomorphic to $\cala \otimes \cala \op$, duality induces a functor
\[
((\cala \otimes \cala\op) \dash \modules )\op 
\rightarrow 
(\cala \otimes \cala\op) \dash \modules,
\]
i.e.,  a contravariant functor from $\cala$-bimodules to $\cala$-bimodules. Restricted to bimodules that take values that are  finite rank free $\kk$-modules, this is an equivalence of categories.

\begin{exam}
\label{exam:dual_bimodule}
The above applies to the $\cala$-bimodule $\cala (-,-)$, yielding the dual bimodule $\cala (-,-)^\sharp$. 
 Now, suppose that each $\cala (X,Y)$ is a finite rank free $\kk$-module, equipped with a given basis. Then, by Remark \ref{rem:duality_permutation_modules}, $\cala (X,Y)^\sharp$ is isomorphic as a $\kk$-module to $\cala (X,Y)$ (the isomorphism depending on the choice of basis). The corresponding $\cala$-bimodule structure can be thought of as a {\em transpose} structure on $\cala (X,Y)$, noting that $X$ is now the covariant variable and $Y$ the contravariant variable. 
 
This situation arises for example when $\cala$ is the $\kk$-linearization of an essentially small category $\calc$ with the property that $\calc (X,Y)$ is a finite set for all objects $X$, $Y$ of $\calc$ (we say that $\calc$ is {\em hom-finite}). 
\end{exam}

We first make explicit the transpose structure in the case $\cala = \kk \calc$,  by unwinding the definitions:

\begin{prop}
\label{prop:calc_trans}
Suppose that $\calc$ is a hom-finite category.
\begin{enumerate}
\item 
For $Y \in \ob \calc$, the functor  $\kk \calc (-,Y)^\sharp \in \ob \kk\calc\dash\modules$ has values $U \mapsto \kk \calc (U,Y)^\sharp$ and, for $h\in \calc (U, V)$, the map $\kk \calc  (h, Y)^\sharp$ acts on the dual basis by  
$$
\eta_f \mapsto \sum_{\hat{f}}  \eta_{\hat{f}},
$$
 where $ f\in \calc (U,Y)$ and  the sum is over the set  of $\hat{f} \in \calc (V,Y)$ such that the following diagram commutes:
\[
\xymatrix{
U \ar[r]^h 
\ar[d]_f 
& V 
\ar@{.>}[ld]^{\hat{f}}
\\
Y.
}
\]
If $h$ is an epimorphism then there exists at most one such $\hat{f}$.
\item 
For $X \in \ob \calc$, the functor $\kk \calc  (X, -)^\sharp \in \ob \kk\calc\op\dash\modules$ has values $Y \mapsto \kk \calc (X, Y)^\sharp$ and a morphism $g \in \calc ( Y, Z)$ acts on the dual basis by 
\[
\eta_f 
\mapsto 
\sum_{\tilde{f}} \eta_{\tilde{f}}, 
\]
 where $f \in \calc (X, Z)$ and  the sum is over the set of $\tilde{f} \in \calc (X,Y)$ that make the following diagram commute
 \[
 \xymatrix{
 & 
 X 
 \ar[d]^f 
 \ar@{.>}[ld]_{\tilde{f}} 
 \\
 Y 
 \ar[r]_g
 &
 Z.
 }
 \]   
 If $g$ is a monomorphism, then there exists at most one such $\tilde{f}$.
 \end{enumerate}
\end{prop}

\begin{rem}
\ 
\begin{enumerate}
\item 
The two statements in Proposition \ref{prop:calc_trans} are clearly `dual' via passage to the opposite category. However, this will usually be applied in the situation in which all morphism in $\calc$ (or, dually, $\calc\op$) are monomorphisms. In this situation the covariant and contravariant actions have somewhat different appearance.
\item 
The result extends to treat the structure of the $\kk$-linear bimodules 
$
(\kk \ub)_{(\pm; \mp)} (-,-)^\sharp$,
  using the fact that these bimodules take values in finite rank free $\kk$-modules, with basis given by Lemma \ref{lem:basis_twisted_kub}. The only difference is that the sums that occur involve signs arising from the twisting modules $\kk_{(\pm; \mp)}$. 
\end{enumerate}
\end{rem}

\subsection{Downward variants}

All of the above carries over {\em mutatis mutandis}  to the study of the downward Brauer category, its $\kk$-linearization $\kk \db$ and the twisted variants $(\kk \db)_{(\pm; \mp)}$. Indeed, these can be introduced as the respective opposite categories. 

\begin{rem}
\label{rem:fbop}
There is a subtlety to keep in mind. The canonical inclusion $\kk \fb \hookrightarrow (\kk \ub)_{(\pm; \mp)}$ yields an inclusion $ \kk \fb\op  \hookrightarrow (\kk \db)_{(\pm; \mp)}$ on passing to the opposite categories. In order to identify the image with $\kk \fb$, we use the isomorphism of categories $\kk \fb \cong \kk \fb\op$ that is given by $[\alpha] \mapsto [\alpha^{-1}]$. In particular, for finite sets $X$, $Y$ with $|Y| - |X | \in 2 \nat$:
\begin{enumerate}
\item 
the canonical left action of $\alpha \in \aut (X)$ on $\kk \db (Y,X)$ is given by $\alpha [g] = [g \circ \alpha^{-1}]$, where $g \in \ub (X, Y)$; 
\item 
the canonical right action of $ \beta \in \aut (Y) $  on $\kk \db (Y,X)$ is given by $[g] \beta  = [\beta^{-1} \circ g]$.
\end{enumerate}
\end{rem}

\section{Torsion modules}
\label{sect:torsion}

The purpose of this section is to review some generalities on torsion modules, working over the $\kk$-linear category $(\kk \ub)_{(\pm;\mp)}$ for a fixed choice of $(\pm; \mp)$. (The reader may prefer to  restrict to the case $(+;+)$, which corresponds to $\kk \ub$, since this case exemplifies all the salient features, but is slightly simpler, due to the absence of signs.)

The discussion involves the functor $\kubg^\sharp \otimes_{\kubg} -$, considered as an endofunctor of the category of $\kubg$-modules. This arises naturally when considering Koszul complexes in Section \ref{sect:koszul} and their applications.

Throughout this section, $\kk$ is taken to be a field, for simplicity.

\subsection{Torsion}
\label{subsect:torsion}

Recall that, for $n \in \nat$, the $\kubg$-module $\kubg( \n, -)$ is projective; it corepresents evaluation on $\n$. Dually, the $\kubg$-module $\kubg (-, \n)^\sharp$ is injective; it represents the functor $M\mapsto M(\n)^\sharp$. A $\kubg$-module $M$ has finite support if $M(\mathbf{t})=0$ for all $t \gg 0$.

\begin{defn}
\label{defn:torsion}
For $M$ a $\kubg$-module,
\begin{enumerate}
\item  
an element $x \in M (\n)$ (we say that $x$ is a section of $M$) is torsion if the image of the morphism $\kubg (\n, -) \stackrel{x}{\rightarrow} M$ corresponding to $x$ by Yoneda's lemma has finite support; 
\item 
the module $M$ is torsion if every section is torsion; 
\item 
the module $M$ is torsion-free if it contains no non-zero torsion submodule.
\end{enumerate}
Write $\torskubg$ for the full subcategory of torsion $\kubg$-modules in $\kubg\dash\modules$.
\end{defn}

There is a more explicit characterization of torsion elements: 

\begin{lem}
\label{lem:alternative_criterion_torsion_element}
For $M$ a $\kubg$-module, a section $x$ in $M (\n)$ is torsion if there exists $t \geq n$ with $t \equiv n \mod 2$ such that the morphism $[i_{n,t}] \in \kubg (\n, \mathbf{t})$, corresponding to the inclusion $\n \subset \mathbf{t}$ and the canonical order on $\mathbf{t} \backslash\n$  by Lemma \ref{lem:basis_twisted_kub}, sends $x$ to zero.
\end{lem}

There is thus a clear relationship between torsion modules and modules with finite support:

\begin{prop}
\label{prop:torsion_vs_finite_support}
Let $M$ be a $\kubg$-module. Then $M$ is torsion if and only if it is the filtered colimit of its submodules that have finite support.
\end{prop}

\begin{proof}
It is immediate that a $\kubg$-module with finite support is torsion and  the submodule generated by a torsion element has finite support, by definition. 
It is also clear that a module is torsion if and only if it is the filtered colimit of its torsion submodules. Putting these facts together gives the result.
\end{proof}

The torsion $\kubg$-modules satisfy the following:

\begin{thm}
\label{thm:localizing_serre}
The subcategory $\torskubg$ is a localizing Serre subcategory of $\kubg\dash\modules$:
\begin{enumerate}
\item 
for $0 \rightarrow M_1\rightarrow M_2 \rightarrow M_3 \rightarrow 0$ a short exact sequence of $\kubg$-modules, $M_2$ is torsion if and only if both $M_1$ and $M_3$ are torsion;
\item 
the subcategory $\torskubg$ is closed under arbitrary coproducts. 
\end{enumerate}
\end{thm}

In particular, one has the localization {\em à la Gabriel} \cite{MR232821}: 
\[
\loc \colon \kubg\dash\modules
\rightarrow 
\kubg\dash \modules / \torskubg
\]
and the localization functor has a right adjoint $\sat \colon \kubg\dash \modules / \torskubg \rightarrow \kubg\dash\modules$. Moreover, the quotient category is a Grothendieck abelian category.

\begin{cor}
\label{cor:torsion_submodule}
For $M$ a $\kubg$-module, there is a natural exact sequence 
\[
0
\rightarrow 
\tors M 
\rightarrow 
M
\rightarrow 
\sat \loc M,
\]
where $\tors M$ is the largest torsion submodule of $M$. Thus $\tors$ is right adjoint to the inclusion of $\torskubg$ in $\kubg\dash\modules$.
\end{cor}

\begin{rem}
\label{rem:torsion-free_quotient}
For $M$ a $\kubg$-module, the quotient $M / \tors M$ is a torsion-free module (termed the torsion-free quotient); this embeds in $\sat\loc M$. 
\end{rem}

\begin{exam}
\label{exam:torsion_torsion-free}
Fix $n \in \nat$.
\begin{enumerate}
\item 
The projective $\kubg (\n, -)$ is torsion-free (as is easily checked by using the criterion of Lemma \ref{lem:alternative_criterion_torsion_element}).
\item 
The injective $\kubg (-,\n)^\sharp$ has finite support (since $\kubg (\mathbf{t},\n)^\sharp=0$ for $t >n$), hence is torsion by Proposition \ref{prop:torsion_vs_finite_support}. 
\end{enumerate}
\end{exam}

We have the following stronger property of $\kubg (-,\n)^\sharp$, that is established by standard arguments:

\begin{prop}
\label{prop:finite-presentation}
For $n\in \nat$, the $\kubg$-module $\kubg (-,\n)^\sharp$ has a finite presentation 
\[
P_1 \rightarrow P_0 \rightarrow \kubg (-,\n)^\sharp \rightarrow 0
\]
where $P_0$ (respectively $P_1$) is a finite direct sum of projectives of the form $\kubg (\mathbf{m}, -)$ with $m \leq n$ (respectively $m \leq n+1$). 
\end{prop}

This implies the following:

\begin{cor}
\label{cor:compact}
For $n \in \nat$, the $\kubg$-module $\kubg (-,\n)^\sharp$ is compact, i.e., the functor
 $$\hom_{\kubg} (\kubg (-,\n)^\sharp, -)$$
  commutes with filtered colimits.
\end{cor}

\subsection{An adjunction}
\label{subsect:adjunction}

We consider $\kubg^\sharp$ as a $\kubg$-bimodule (cf. the duality considerations in Section \ref{subsect:transpose}). This yields the functor 
\[
\kubg^\sharp \otimes_{\kubg} - 
\ : \  
\kubg \dash\modules \rightarrow \kubg \dash\modules.
\]
 Explicitly, for $M$ a $\kubg$-module and $X$ a finite set (considered as an object of $\kubg$), 
\[
\big( 
\kubg^\sharp \otimes_{\kubg} M
\big)
(X) 
=
\kubg(X, -)^\sharp \otimes_{\kubg} M.
\]

The following is clear, using  that, for each $n\in \nat$, $\kubg(-,\n)^\sharp$ is torsion, as observed in Example \ref{exam:torsion_torsion-free}, together with Theorem \ref{thm:localizing_serre}.

\begin{prop}
\label{prop:functor_to_torsion}
For any $\kubg$-module $M$, the $\kubg$-module $\kubg^\sharp \otimes_{\kubg} M$ is torsion. Hence $\kubg^\sharp \otimes_{\kubg} -$ induces a functor 
\[
\kubg^\sharp \otimes_{\kubg} -
\colon 
\kubg\dash\modules 
\rightarrow 
\torskubg.
\]
\end{prop}

We also have:

\begin{prop}
\label{prop:vanish-on-torsion}
For $M$ a torsion $\kubg$-module, $\kubg^\sharp \otimes_{\kubg} M=0$.
\end{prop}

\begin{proof}
It suffices to show that, for any $X$,  $\hom_\kk (\kubg(X,-)^\sharp \otimes_{\kubg} M, \kk)$ is zero. By the universal property of the tensor product, this is isomorphic to $\hom_{\kubg} (M, \kubg(X,-))$, using that the bimodule $\kubg$ takes finite-dimensional values, so that $(\kubg^\sharp)^\sharp$ is isomorphic to $\kubg$. Now, as in Example \ref{exam:torsion_torsion-free}, $\kubg (X, -)$ is torsion-free; thus, there is no non-zero morphism from the torsion module $M$  to $\kubg (X, -)$.
\end{proof}

This has the immediate consequence:

\begin{cor}
\label{cor:pass_torsion-free_quotient}
For $M$ a $\kubg$-module, the canonical quotient $M \twoheadrightarrow M/\tors M$ induces an isomorphism
$$
\kubg^\sharp \otimes_{\kubg} M
\stackrel{\cong}{\rightarrow}
\kubg^\sharp \otimes_{\kubg} (M/\tors M).
$$
\end{cor}

The functor $\kubg^\sharp \otimes_{\kubg} - $ has right adjoint $\hom_{\kubg} (\kubg^\sharp , -) $. Now, using the fact that $\tors$ is right adjoint 
to the inclusion of $\torskubg$ in $\kubg$-modules (see Corollary \ref{cor:torsion_submodule}), for a $\kubg$-module $M$ one has the natural isomorphism:
$$
\hom_{\kubg} (\kubg^\sharp , M) \cong 
\hom_{\kubg} (\kubg^\sharp , \tors M).
$$
In particular, we may as well restrict  $\hom_{\kubg} (\kubg^\sharp , -) $ to $\torskubg$.

The following statement summarizes this situation:

\begin{prop}
\label{prop:tors_adjunction}
There is an adjunction 
\begin{eqnarray*}
 \resizebox{\hsize}{!}{
$
\xymatrix{
\kubg^\sharp \otimes_{\kubg} - 
\ar@{} [r]|(.55):
&
\kubg\dash\modules
\ar@<1ex>[r]
\ar@{}[r]|\perp
&
\torskubg 
\ar@<1ex>[l]
&
\hom_{\kubg} (\kubg^\sharp , -).
\ar@{}[l]|(.6) :
}
$
}
\end{eqnarray*} 
\end{prop}

\begin{exam}
\label{exam:proj<->inj}
For $n \in \nat$, there are natural isomorphisms
\begin{eqnarray*}
\kubg^\sharp \otimes_{\kubg} \kubg (\n, -) &\cong & \kubg (-,\n)^\sharp ;
\\
\hom_{\kubg} (\kubg^\sharp , \kubg (-,\n)^\sharp) &\cong &\kubg (\n, -). 
\end{eqnarray*}
Thus the standard projectives and the standard injectives are in one-one correspondence across the adjunction of Proposition \ref{prop:tors_adjunction}.
\end{exam}

\subsection{A refinement}
We have the induction functor 
\begin{eqnarray}
\label{eqn:induction}
\kubg \otimes_{\kk \fb} - \  \colon \kk\fb\dash\modules \rightarrow \kubg\dash\modules
\end{eqnarray}
and the following extension of the behaviour on projectives given in Example \ref{exam:proj<->inj}:

\begin{lem}
\label{lem:composite_with_free}
The composite of the functor (\ref{eqn:induction})  with $\kubg^\sharp \otimes _{\kubg} -$ is naturally isomorphic to 
\[
\kubg^\sharp \otimes_{\kk \fb} - \ \colon  
\kk\fb\dash\modules \rightarrow \torskubg.
\]
\end{lem}

It follows that, for $N$ a $\kk \fb$-module, the adjunction unit (for the induced module $M:= \kubg \otimes_{\kk \fb} N$)  provides the natural morphism of $\kubg$-modules:
\begin{eqnarray}
\label{eqn:adjunction_unit_kubg}
\kubg \otimes_{\kk \fb} N 
\rightarrow 
\hom_{\kubg} (\kubg^\sharp, \kubg^\sharp \otimes _{\kk \fb} N).
\end{eqnarray}

\begin{prop}
\label{prop:iso_in_char_0}
Suppose that $\kk$ is a field of characteristic zero, then the natural transformation (\ref{eqn:adjunction_unit_kubg}) is an isomorphism.
\end{prop}

\begin{proof}
By Corollary \ref{cor:compact}, for each $X$, $\kubg(-,X)^\sharp$ is a compact $\kubg$-module. Using this, one reduces to the case that $N$ is a finite-dimensional $\kk \sym_t$-module, considered as a $\kk \fb$-module supported on $\mathbf{t}$. 

If $N = \kk \sym_t$,  the morphism (\ref{eqn:adjunction_unit_kubg})  is an isomorphism, realizing the identification of $\hom_{\kubg} (\kubg^\sharp, -)$ applied to standard injectives  given in Example \ref{exam:proj<->inj}.  It follows easily that this is also the case for a finite direct sum of $\kk \sym_t$'s. 

In general, for $N$ supported on $\mathbf{t}$ with finite dimension, since we are working in characteristic zero, $N$ is a direct summand of a finite direct sum of $\kk \sym_t$'s. By the above, (\ref{eqn:adjunction_unit_kubg})  in this case is a retract of an isomorphism, hence is an isomorphism.
\end{proof}

This implies the following:

\begin{cor}
\label{cor:equivalence_full_subcategories}
For $\kk$ of characteristic zero, the adjunction of Proposition \ref{prop:tors_adjunction} induces an equivalence between the following :
\begin{enumerate}
\item 
the full subcategory of $\kubg\dash\modules$ on the essential image of  $\kubg \otimes_{\kk \fb} -$;
\item 
the full subcategory of $\torskubg$ on the essential image of $\kubg^\sharp \otimes_{\kk \fb} -$.
\end{enumerate}
\end{cor}

\section{The Koszul property}
\label{sect:koszul}

The purpose of this section is to introduce and make explicit the Koszul complexes that are associated to the $\kk$-linear categories $\kk \ub_{(\pm; \mp)}$. Then we explain, using the fact that these $\kk$-linear categories are Koszul over $\kk \fb$, that the (co)homology of these complexes calculate certain $\ext$ and $\tor$ groups. We also consider briefly the relationship between the latter.

This material is not original;  it is spelled out in relative detail, since the author knows of no convenient reference where this is presented in this form.

\subsection{Quadratic duality}

We review quadratic duality for homogeneous quadratic $\kk$-linear categories. For a general reference (working over a base ring) see \cite[Chapter 1]{MR4398644}, for example.

We first review duality for bimodules for $R$ a unital, associative ring. We use the following two duality functors on $R$-bimodules:
\begin{eqnarray*}
&&
\hom_R (-, R) 
\\
&& \hom_{R\op} (-, R).
\end{eqnarray*}
For instance, if $M$ is an $R$-bimodule, the left $R$-module structure on $\hom_R (M, R)$ is induced by the right $R$-action on $M$ whereas the right $R$-action is induced by the right regular action on the codomain. (Explicitly, for $\phi$ a morphism of left $R$-modules from $M$ to $R$, one has $(r_1 \phi r_2) (m) = \phi (m r_1) r_2$.) The functor $\hom_R (-, R) $ behaves particularly well when restricted to bimodules $M$ that are finitely-generated projective as left $R$-modules. For example, this hypothesis implies that the bimodule $\hom_R (M, R)$ is finitely-generated projective as a {\em right} $R$-module. 

\begin{rem}
\label{rem:right_duality_eval_coeval}
For $M$ an $R$-bimodule, evaluation induces the morphism of $R$-bimodules: 
\begin{eqnarray*}
\hom_{R\op} (M, R) \otimes_R M &\rightarrow & R
\\
\phi \otimes m &\mapsto & \phi (m).
\end{eqnarray*}

Suppose that $M$, considered as a right $R$-module,  is free of finite rank  with basis $\{x_i\}$ and let $\{ \phi_i\}$ be the dual basis of $\hom_{R\op} (M, R)$. Then there is a morphism of $R$-bimodules: 
\begin{eqnarray*}
R & \rightarrow & M \otimes_R \hom_{R\op} (M, R)\\
1 & \mapsto & \sum_i x_i \otimes \phi_i.
\end{eqnarray*}
That this is a morphism of $R$-bimodules is equivalent to the fact that, for all $r \in R$, one has $r (\sum_i x_i \otimes \phi_i) = (\sum_i x_i \otimes \phi_i)r$, which can be checked directly.
\end{rem}

Now consider  a $\kk$-linear category $\cala$ that is homogeneous quadratic; in particular, the morphisms of $\cala$ are $\nat$-graded (denoted by $\cala^n$, for $n \in \nat$), $\cala^0$ is a sub $\kk$-linear category, and each $\cala^n$ has the structure of an $\cala^0$-bimodule. The composition induces morphisms of $\cala^0$-bimodules 
\[
\cala^s \otimes_{\cala^0} \cala^t \stackrel{\mu_\cala}{\longrightarrow} \cala^{s+t}.
\] 
The homogeneous quadratic hypothesis implies that morphisms are generated under composition by the bimodule $\cala^1$ and the relations are generated as an $\cala$-bimodule  by the kernel $\mathscr{R}$ in the short exact sequence of $\cala^0$-bimodules
\[
0
\rightarrow 
\mathscr{R}
\rightarrow 
\cala^1 \otimes _{\cala ^0} \cala ^1 
\stackrel{\mu_\cala}{\longrightarrow}
\cala^2
\rightarrow 
0.
\]

The left quadratic dual is constructed using the functor $\hom_{\cala^0} (-, \cala^0)$,  requiring that both $\cala^1$ and $\cala^2$ are finitely-generated projective as left $\cala^0$-modules. We can then form the homogenous quadratic category over $\cala^0$ with 
 $\hom_{\cala^0} (\cala ^1 , \cala^0)$ the bimodule of generators in degree one  and relations in degree two given by $\hom_{\cala^0} (\cala ^2,\cala^0)$.

Likewise, the right quadratic dual is constructed  using $\hom_{(\cala^0)\op} (-, \cala^0)$, requiring the  projectivity of the right $\cala^0$-module structures of $\cala^1$ and $\cala^2$.

\begin{exam}
\label{exam:projectivity_properties}
Recall from Proposition \ref{prop:kub_twist_homogeneous_quadratic} that $(\kk \ub)_{(\pm; \mp)}$ is homogeneous quadratic over $\kk \fb$; this implies that $(\kk \db)_{(\pm; \mp)} $ is so also (using the identification $(\kk \db)^0_{(\pm;-)} \cong \kk \fb$ given in Remark \ref{rem:fbop}).  Then, by Proposition \ref{prop:morphisms_right_free},
\begin{enumerate}
\item 
the right $\kk \fb$-module projectivity hypothesis is satisfied for $(\kk \ub)_{(\pm; \mp)}$; 
\item 
the left $\kk \fb$-module projectivity hypothesis is satisfied for $(\kk \db)_{(\pm; \mp)}$.
\end{enumerate}
Hence, we can consider the right quadratic dual of $(\kk \ub)_{(\pm; \mp)}$ and the left quadratic dual of $(\kk \db)_{(\pm; \mp)}$. These are related by the passage to the opposite category.
\end{exam}

\begin{prop}
\label{prop:quadratic_duals_kub}
Considered as $\nat$-graded $\kk$-linear categories over $\kk \fb$, the following hold:
\begin{enumerate}
\item
the right quadratic dual of $(\kk \ub)_{(\pm; +)}$ is isomorphic to $(\kk \db)_{(\pm;-)}$;
\item 
the right quadratic dual of $(\kk \ub)_{(\pm; -)}$ is isomorphic to $(\kk \db)_{(\pm;+)}$. 
\end{enumerate}
\end{prop}

\begin{proof}
First we note that $(\kk \ub)_{(\pm; +)}^1$ and $(\kk \ub)_{(\pm; -)}^1$ are isomorphic as $\kk \fb$-bimodules, since in degree one, there is only one `chord', so the orientation data for the order of chords does not intervene. Likewise for  $(\kk \db)_{(\pm; +)}^1$ and $(\kk \db)_{(\pm; -)}^1$. 

The key point is to establish  the duality isomorphism $\hom_{\kk \fb\op} ((\kk \ub)_{(\pm; +)}^1, \kk \fb) \cong (\kk \db)_{(\pm;-)}^1$
of $\kk \fb$-bimodules (this also yields the case with $(\pm;+)$ and $(\pm;-)$ swapped, by the above remark). This isomorphism is induced by the pairing 
\[
(\kk \db)_{(\pm; -)}^1
\otimes_{\kk \fb}
(\kk \ub)_{(\pm; +)}^1
\rightarrow 
\kk \fb
\]
of $\kk\fb$-bimodules described below. 

For clarity, we treat the case $\pm=+$, so that $(\kk \ub)_{(\pm;+)}$ identifies with $\kk \ub$; the case $\pm = -$ is treated by a similar argument (as indicated below). Using the standard basis, a $\kk$-module generator of $\kk \ub^1 (X, Y) $ is of the form $[f]$ where $f : X\hookrightarrow Y$ and $|Y|-|X|=2$. Likewise, a $\kk$-module generator of $(\kk \db)_{(+; -)}^1 (Y, X)$ is given by the class $[g]$ of $g : X \hookrightarrow Y$. If $g(X)=f(X)$, then there exists a unique $\alpha \in \aut (X)$ (depending on $f$ and $g$) such that $f = g \circ \alpha$. Using this, the pairing 
\[
(\kk \db)_{(+; -)}^1(Y,X)
\otimes_{\kk\fb} 
\kk \ub^1 (X,Y)
\rightarrow 
\kk \aut (X)
\]
is given by 
\[
[g] \otimes [f]  \mapsto 
\left\{ 
\begin{array}{ll}
\ 0 & f(X) \neq g(X) \\
\ [\alpha] & f(X) =g(X).
\end{array}
\right.
\]
First one checks that this is well-defined, i.e. that it factors across $\otimes_{\kk \fb}$. This follows from the fact that, for $\beta \in \aut (Y)$, the right action of $\aut (Y)$ on $(\kk \db)_{(+; -)}^1(Y,X)$ corresponds to $[g] \mapsto [\beta^{-1} \circ g]$ (see Remark \ref{rem:fbop}), and the left action on $\kk \ub^1 (X,Y)$ corresponds to $[f]\mapsto [\beta \circ f]$. Clearly, the image of $\beta^{-1} \circ g$ equals that of $f$ if and only the image of $g$ equals that of $\beta \circ f$. Moreover, $f= (\beta^{-1} \circ g) \circ \alpha$ if and only if $\beta \circ f=  g\circ \alpha$. This yields the factorization over $\otimes_{\kk \fb}$.

That this is a morphism of bimodules is verified similarly: the left action of $\gamma \in \aut (X)$ on $(\kk \db)_{(+; -)}^1(Y,X)$ corresponds to $[g]\mapsto [g \circ \gamma^{-1}]$.  If $f = g \circ \alpha$, then $f = (g\circ \gamma^{-1}) \circ (\gamma \circ \alpha)$, which implies the morphism is compatible with the left module structure. The right module structure is treated similarly.

In the case $\pm = -$, the only modification that is required is due to the orientation sign associated to the complements $Y \backslash f(X)$ and $Y \backslash g(X)$. In the case that $f(X) =g(X)$, one arranges that the order on these complements is the same, thereby possibly introducing a sign.

It remains to check the quadratic relations. This follows by the usual argument: there is a Koszul-type sign which is introduced. This accounts for the $(\pm ; +) \leftrightarrow (\pm ; -)$ correspondence.
\end{proof}

\begin{rem}
\label{rem:left/right_quadratic_duality}
By general properties of quadratic duality (see \cite[Proposition 1.6]{MR4398644}, for example), there are counterparts for left quadratic duality. For example, the  left quadratic dual of $(\kk \db)_{(\pm;-)}$ is isomorphic to $(\kk \ub)_{(\pm; +)}$.
\end{rem}

For use in the following section, we note the following consequence of (the proof of) Proposition \ref{prop:quadratic_duals_kub}.

\begin{cor}
\label{cor:coevaluation}
There are morphisms of $\kk \fb$-bimodules
\begin{eqnarray*}
\kk \fb & \rightarrow & (\kk \ub)_{(\pm; +)}^1 \otimes _{\kk \fb} (\kk \db)_{(\pm;-)}^1
\\
\kk \fb & \rightarrow & (\kk \ub)_{(\pm; -)}^1 \otimes _{\kk \fb} (\kk \db)_{(\pm;+)}^1.
\end{eqnarray*}
Evaluated on $\n$, for $n \in \nat$, these are  given by 
\[
[1]  \mapsto \sum_{g} [g'] \otimes [g']
\]
where $g \in \finj (\n, \mathbf{n+2})$ runs over the set of order preserving inclusions and $[g']$ is the associated element of  $(\kk \ub)_{(\pm; +)} (\n, \mathbf{n+2})$ (respectively $(\kk \db)_{(\pm; -)} (\n, \mathbf{n+2})$) represented by $g' \in \dubord (\n ,\mathbf{n+2})$ using the canonical order on the complement of $g(\n)$. 
\end{cor}

\begin{proof}
This follows from the general result recalled in Remark \ref{rem:right_duality_eval_coeval} together with the identification of the dual basis across the isomorphism given in the proof of Proposition \ref{prop:quadratic_duals_kub}.
\end{proof}

\subsection{The Koszul dualizing complex}

Using the fact that $(\kk \ub)_{(\pm; \mp)}$ is a homogeneous quadratic $\kk$-linear category over $\kk \fb$, the identification of its quadratic dual given by Proposition \ref{prop:quadratic_duals_kub}, together with the coevaluation map of Corollary \ref{cor:coevaluation}, one has the Koszul dualizing complex
\[
\mathscr{K}_{(\pm; \mp)}
:= 
(\kk \ub)_{(\pm; \mp)} \otimes_{\kk \fb} (\kk \db)_{(\pm; -\mp)},
\]
equipped with the differential induced by coevaluation. (Compare the complexes introduced in \cite[Chapter 2]{MR4398644}.) This is a complex of $(\kk \ub)_{(\pm; \mp)} \otimes (\kk \db)_{(\pm; - \mp)}\op$-modules (equivalently $(\kk \ub)_{(\pm; \mp)} \otimes (\kk \ub)_{(\pm; - \mp)}$-modules). 

\begin{rem}
\label{rem:choice_grading}
There is a choice for the grading: either using that arising from  $(\kk \ub)_{(\pm; \mp)}$ or that arising from  $ (\kk \db)_{(\pm; -\mp)}$;  this affects whether the differential has degree $+1$ or $-1$.
\end{rem}

Evaluated on $(U, X)\in \ob \ (\kk \ub)_{(\pm; \mp)} \otimes (\kk \db)_{(\pm; - \mp)}$, $\mathscr{K}_{(\pm; \mp)}$ has underlying graded $\kk$-module:
\[
\bigoplus_s 
(\kk \ub)_{(\pm; \mp)} (\mathbf{s}, U)  \otimes_{\kk \sym_s} (\kk \db)_{(\pm; -\mp)} (X, \mathbf{s}).
\]
This is zero if $|U|$ and $|X|$ have different parities; otherwise it ranges over $s \leq \min \{ |U|, |X| \}$ such that $s$ has the same parity as $|U|$ and $|X|$. In particular the sum is finite.

\begin{rem}
\label{rem:graphical_K}
It is illuminating to interpret this complex using the graphical description explained below, extending the approach outlined in Section \ref{subsect:ub_db}.
We fix $s$ and, for concreteness, set $(\pm; \mp) = (+;+)$ (the analysis carries over {\em mutatis mutandis} to the other cases of $(\pm; \mp)$).

Consider $\kk \ub (\mathbf{s}, U) \otimes_{\sym_s} (\kk \db)_{(+;-)} (X,\mathbf{s})$ (where we may assume that $|U|- s \in 2 \nat$ and $|X|-s \in 2 \nat$). There is a natural graphical description of a basis given by  diagrams of the form:
 
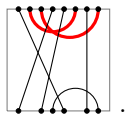
\begin{center}
\begin{tikzpicture}[scale = .15]
\begin{scope}
    \clip (1.5,9) rectangle (6.5,6);
    \draw [red, very thick] (4,9) circle(2);
\end{scope}

\begin{scope}
    \clip (2.5,9) rectangle (8.5,6);
    \draw [red, very thick] (5.5,9) circle(2.5);
\end{scope}

\begin{scope}
    \clip (3.5,0) rectangle (8.5,3);
    \draw (6,0) circle(2);
\end{scope}

\draw [gray] (0,0) -- (0,9) -- (9,9) -- (9,0) -- (0,0); 
\draw (1,9) -- (5,0); 
\draw (4,9) -- (1,0); 
\draw (5,9) -- (3,0); 
\draw (7,9)-- (7,0);

\draw [fill=black] (1,9) circle (0.2);
\draw [fill=black] (2,9) circle (0.2);
\draw [fill=black] (3,9) circle (0.2);
\draw [fill=black] (4,9) circle (0.2);
\draw [fill=black] (1,0) circle (0.2);
\draw [fill=black] (3,0) circle (0.2);
\draw [fill=black]  (6,9) circle (0.2);
\draw [fill=black]  (7,9) circle (0.2);
\draw [fill=black]  (8,9) circle (0.2);
\draw [fill=black]  (7,0) circle (0.2);
\draw [fill=black] (5,9) circle (0.2);
\draw [fill=black] (5,0) circle (0.2);
\draw [fill=black] (4,0) circle (0.2);
\draw [fill=black] (8,0) circle (0.2);
\node [right] at (9,0) {.};
\end{tikzpicture}
\end{center}
Here $s$ is the number of vertical strands; 
$X$ labels the nodes on the top line;  
$U$ labels the nodes on the bottom line; 
the orientation data is encoded via a total order of the upper pairs (indicated by using thick red arcs).

This makes the $\kk \ub \otimes (\kk \db)_{(+;-)}\op$-module structure of $\mathscr{K}$ transparent: for the $\kk\ub$-module structure, one stacks below using diagrams representing morphisms of $\ub$; for the $(\kk \db)_{(+;-)}$-module structure, one stacks above with diagrams representing morphisms of $\db$ (taking into account the orientation data given by the total ordering of the pairs). In particular, these operations leave $s$ (the number of vertical strands) unchanged. 

It is clear that $\kk \ub (\mathbf{s}, -) \otimes _{\sym_s} (\kk \db)_{(+;-)} (-, \mathbf{s})$ is generated as a $\kk \ub \otimes (\kk \db)_{(+;-)}\op$-module by the element of 
$\kk \ub (\mathbf{s}, \mathbf{s}) \otimes _{\sym_s} (\kk \db)_{(+;-)} (\mathbf{s}, \mathbf{s})$ corresponding to the `identity' diagram, in which all strands are vertical (and there
are no chords).  We denote this element by $[\mathbf{1}_\mathbf{s}]$.

For instance, the above diagram can be considered as the following composite:

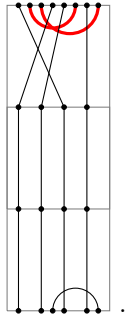
\begin{center}
\begin{tikzpicture}[scale = .15]
\begin{scope}
    \clip (1.5,9) rectangle (6.5,6);
    \draw [red, very thick] (4,9) circle(2);
\end{scope}

\begin{scope}
    \clip (2.5,9) rectangle (8.5,6);
    \draw [red, very thick] (5.5,9) circle(2.5);
\end{scope}

\begin{scope}
    \clip (3.5,-18) rectangle (8.5,-12);
    \draw (6,-18) circle(2);
\end{scope}

\draw [gray] (0,-18) -- (0,9) -- (9,9) -- (9,-18) -- (0,-18); 
\draw [gray] (0,-9) -- (0,0) -- (9,0) -- (9,-9) -- (0,-9);

\draw (1,9) -- (5,0); 
\draw (4,9) -- (1,0); 
\draw (5,9) -- (3,0); 
\draw (7,9)-- (7,0);

\draw (1,0) -- (1,-18);
\draw (3,0) -- (3,-18);
\draw (5,0) -- (5,-18);
\draw (7,0) -- (7,-18);

\draw [fill=black] (1,9) circle (0.2);
\draw [fill=black] (2,9) circle (0.2);
\draw [fill=black] (3,9) circle (0.2);
\draw [fill=black] (4,9) circle (0.2);

\draw [fill=black]  (6,9) circle (0.2);
\draw [fill=black]  (7,9) circle (0.2);
\draw [fill=black]  (8,9) circle (0.2);
\draw [fill=black] (5,9) circle (0.2);

\draw [fill=black] (1,0) circle (0.2);
\draw [fill=black] (3,0) circle (0.2);
\draw [fill=black] (5,0) circle (0.2);
\draw [fill=black]  (7,0) circle (0.2);

\draw [fill=black] (1,-9) circle (0.2);
\draw [fill=black] (3,-9) circle (0.2);
\draw [fill=black] (5,-9) circle (0.2);
\draw [fill=black]  (7,-9) circle (0.2);

\draw [fill=black] (1,-18) circle (0.2);
\draw [fill=black] (3,-18) circle (0.2);
\draw [fill=black] (5,-18) circle (0.2);
\draw [fill=black] (4,-18) circle (0.2);
\draw [fill=black]  (7,-18) circle (0.2);
\draw [fill=black] (8,-18) circle (0.2);
\node [right] at (9, -18) {.};
\end{tikzpicture}
\end{center}
(Such a representation is not unique, due to the $\otimes_{\sym_s}$.)

The differential also has an elegant interpretation. It has components of the form:
\[
\kk \ub (\mathbf{s}, U) \otimes_{\sym_s} (\kk \db)\tw (X,\mathbf{s})
\rightarrow 
\kk \ub (\mathbf{s-2}, U) \otimes_{\sym_{s-2}} (\kk \db)\tw (X,\mathbf{s-2})
\]
and is determined by its values on $[\mathbf{1}_\mathbf{s}]$ (corresponding to the vertical strand diagram for $X= \mathbf{s} = U$, as above). 

For $i < j \in \mathbf{s}$, let $[d_{i,j}]$ be the element of $\kk\ub (\mathbf{s-2}, \mathbf{s}) \otimes_{\sym_{s-2}} (\kk \db)_{(+;-)} (\mathbf{s}, \mathbf{s-2})$ corresponding to the basis element $d_{i,j}$ given by  the  diagrams in which chords link $i$ and $j$ and the strands are vertical. For example, for  $s=4$ and $i=2$, $j=4$ (the `crossing' is only removed  for visual clarity):

\begin{center}
\begin{tikzpicture}[scale = .15]
\begin{scope}
    \clip (2.5,9) rectangle (7.5,6);
    \draw  [red, very thick] (5,9) circle(2);
\end{scope}

\begin{scope}
    \clip (2.5,0) rectangle (7.5,3);
    \draw (5,0) circle(2);
\end{scope}

\draw [white,fill= white] (5,2) circle (0.2);

\draw [gray] (0,0) -- (0,9) -- (9,9) -- (9,0) -- (0,0); 
\draw (1,9) -- (1,0); 
\draw (5,9) -- (5,0); 

\draw [fill=black] (1,9) circle (0.2);
\draw [fill=black] (3,9) circle (0.2);
\draw [fill=black] (1,0) circle (0.2);
\draw [fill=black] (3,0) circle (0.2);
\draw [fill=black]  (7,9) circle (0.2);
\draw [fill=black]  (7,0) circle (0.2);
\draw [fill=black] (5,9) circle (0.2);
\draw [fill=black] (5,0) circle (0.2);
\node [right] at (9, 0) {.};
\end{tikzpicture}
\end{center}

\noindent

The differential is determined by 
\[
\ 
[\mathbf{1}_\mathbf{s}] 
\mapsto
\sum_{1 \leq i < j \leq s} [d_{i,j}]
\]
(see Corollary \ref{cor:coevaluation}). This is well-defined since the coevaluation map is a morphism of bimodules. 
 The fact that $d^2=0$ (i.e., that this gives a differential) is a  consequence of the orientation sign coming from the order of the red chords. 
\end{rem}

There are various relations between the different complexes $\mathscr{K}_{(\pm; \mp)}$. For example:

\begin{prop}
\label{prop:isomorphism_K_+-}
Choosing compatible gradings, there is an isomorphism of complexes of $(\kk \ub)_{(\pm; \mp)}\otimes (\kk \ub)_{(\pm; -\mp)}$-modules
\[
\mathscr{K}_{(\pm; \mp)} \cong \mathscr{K}_{(\pm; -\mp)}.
\]
\end{prop} 

\begin{proof}
By definition, $(\kk \db) _{(\pm; \mp)}$ is $(\kk \ub) _{(\pm; \mp)} \op$. It follows that there is an isomorphism
\[
(\kk \ub)_{(\pm; \mp)} (\mathbf{s}, U)  \otimes_{\kk \sym_s} (\kk \db)_{(\pm; -\mp)} (X, \mathbf{s})
\cong 
(\kk \ub)_{(\pm; -\mp)} (\mathbf{s}, X)  \otimes_{\kk \sym_s} (\kk \db)_{(\pm; -\mp)} (U, \mathbf{s}).
\] 
Considering the module structures with respect to $(U,X)$, this is an isomorphism of  $(\kk \ub)_{(\pm; \mp)}\otimes (\kk \db)_{(\pm; -\mp)}\op = (\kk \ub)_{(\pm; \mp)}\otimes (\kk \ub)_{(\pm; -\mp)}$-modules. 

This gives the result at the level of the underlying graded $(\kk \ub)_{(\pm; \mp)}\otimes (\kk \ub)_{(\pm; -\mp)}$-modules. It remains to check that this is compatible with the differentials. This follows from the symmetric nature of the construction of the differential which is apparent from the diagram representing $[d_{i,j}]$ in Remark \ref{rem:graphical_K}.
\end{proof}

\begin{rem}
Proposition \ref{prop:isomorphism_K_+-} means that we can reduce to considering the complexes $\mathscr{K}_{(+; +)}$,  a complex of $\kk \ub \otimes (\kk \ub)_{(+;-)}$-modules, and $\mathscr{K}_{(-;+)}$, a complex of $(\kk \ub)_{(-;+)} \otimes (\kk \ub)_{(-;-)}$-modules. In the first complex, the chords are unoriented, whereas in the second the chords carry orientation signs. Both complexes involve orientation signs associated to the order of one of the collection of chords.

For notational simplicity, these will be denoted $\mathscr{K}_+$ and $\mathscr{K}_-$ respectively, so that 
\begin{eqnarray*}
\mathscr{K}_+ &=& \kk \ub \otimes_{\kk \fb} (\kk \db)_{(+;-)}  \ \cong \ (\kk \ub)_{(+;-)}  \otimes_{\kk \fb} \kk \db\\
\mathscr{K}_- &=& (\kk \ub)_{(-;+)} \otimes_{\kk \fb} (\kk \db)_{(-;-)}
\ \cong \ (\kk \ub)_{(-;-)} \otimes_{\kk \fb} (\kk \db)_{(-;+)}
.
\end{eqnarray*}
\end{rem}

\begin{rem}
\label{rem:relate_K+_K-}
The complexes $\mathscr{K}_+$ and $\mathscr{K}_-$ can be related by exploiting $\sgn \otimes -$, as in Proposition \ref{prop:sgn_tensor_modules_categories}.
\end{rem}

\subsection{Associated Koszul complexes}

We  focus on the following Koszul complexes  defined using $\mathscr{K}_+$:
\begin{enumerate}
\item 
for $M$ a $\kk \db$-module, 
$
\mathscr{K}_+ \otimes_{\kk \db} M$,
a complex of $(\kk \ub)_{(+;-)}$-modules;
\item 
for $N$ a $(\kk \db)_{(+;-)}$-module,
$
\mathscr{K}_+ \otimes_{(\kk \db)_{(+;-)}}N$,
 a complex of $\kk \ub$-modules; 
\item 
for $M'$ a $\kk \ub$-module, 
$
\hom _{\kk \ub} (\mathscr{K}_+, M')$,
 a complex of $(\kk\db)_{(+;-)}$-modules;
\item 
for $N'$ a $(\kk\ub)_{(+;-)}$-module, 
$ 
\hom_{(\kk \ub)_{(+;-)}}(\mathscr{K}_+, N')$,
 a complex of $\kk \db$-modules. 
\end{enumerate}

\begin{rem}
These  extend to complexes of modules; the  associated functors are then related by the adjunctions 
\begin{eqnarray}
\label{eqn:adjunction1}
\mathscr{K}_+ \otimes_{\kk \db} - & \dashv & \hom_{(\kk \ub)_{(+;-)}} (\mathscr{K}_+, -) 
\\
\label{eqn:adjunction2}
\mathscr{K}_+ \otimes_{(\kk \db)_{(+;-)}} - & \dashv & \hom_{\kk \ub} (\mathscr{K}_+, -), 
\end{eqnarray}
the first relating complexes of $\kk \db$-modules and complexes of $(\kk\ub)_{(+;-)}$-modules, the second relating complexes of  $(\kk \db)_{(+;-)}$-modules and complexes of $\kk\ub$-modules.
\end{rem}

The terms of the Koszul complexes identify as follows:

\begin{lem}
\label{lem:kz-complexes_identify_underlying}
 There are isomorphisms of the underlying $\kk$-modules (forgetting the grading):
\begin{enumerate}
\item 
for $M$ a $\kk \db$-module and $X$ an object of $(\kk \ub)_{(+;-)}$, 
$$
(\mathscr{K}_+ \otimes_{\kk \db} M) (X) \cong 
(\kk \ub)_{(+;-)} (-, X) \otimes _{\kk \fb} M
\cong 
\bigoplus_{\substack{Y\subset X \\
\ub (\emptyset, X\backslash Y)} }
M (Y);
$$
\item 
for $N$ a $(\kk \db)_{(+;-)}$-module and $X$ an object of $\kk \ub$,
$$
(\mathscr{K}_+ \otimes_{(\kk \db)_{(+;-)}} N) (X) \cong 
\kk \ub(-, X) \otimes _{\kk \fb}N
\cong 
\bigoplus_{\substack{Y\subset X \\
\ub (\emptyset, X\backslash Y)}  }
N (Y);
$$
 \item 
for $M'$ a $\kk \ub$-module and $X$ an object of $(\kk \db)_{(+;-)}$, 
$$
 \hom _{\kk \ub} (\mathscr{K}_+, M')(X)
\cong 
\hom_{\kk \fb} ((\kk \db)_{(+;-)}(X,-) , M' )
\cong 
\bigoplus_{\substack{Y\subset X \\
\ub (\emptyset, X\backslash Y)}  }
M' (Y);
$$
\item 
for $N'$ a $(\kk\ub)_{(+;-)}$-module and $X$ an object of $\kk \db$,
$$
\hom_{(\kk \ub)_{(+;-)}}(\mathscr{K}_+, N')(X)
\cong 
\hom_{\kk \fb} (\kk \db (X, -), N') 
\cong 
\bigoplus_{\substack{Y\subset X \\
\ub (\emptyset, X\backslash Y)}  }
N' (Y).
$$
\end{enumerate}
\end{lem}

\begin{proof}
In each case, the first isomorphism is immediate from standard properties of $\otimes$ and $\hom$. The second isomorphism then follows from Lemma \ref{lem:bimodule_kub_twisted}; in the cases $\hom_\cala (\mathscr{K}_+, -)$, with $\cala \in \{ \kk \ub, (\kk \ub)_{(+;-)} \}$, 
one also uses that the relevant $(\kk \db) _{(\pm; \mp)} (-, X)$ has finite support, so that one obtains a finite direct sum in the explicit identification.  
\end{proof}

\begin{rem}
\label{rem:dualisability}
For fixed $X$, $(\kk \db )_{(\pm; \mp)} (X, -)$ is finitely-generated projective as a $\kk \fb$-module (cf. Example \ref{exam:projectivity_properties}); more precisely, for any $s$, $(\kk \db )_{(\pm; \mp)} (X, \mathbf{s})$ is a finite rank free $\kk \sym_s$-module that is zero for $s>|X|$. The corresponding statement holds for $\kk \db (X, -)$. 
This implies that these objects are strongly dualizable, so that  there are isomorphisms
\begin{eqnarray*}
\hom_{\kk \fb} ((\kk \db)_{(+;-)} , M' )
& \cong & 
(\kk \db)_{(+;-)}^\sharp \otimes_{\kk \fb} M' 
\\
 \hom_{(\kk \fb)}(\kk \db, N')
& \cong & 
\kk \db ^\sharp \otimes_{\kk \fb} N'.
\end{eqnarray*} 
The structure of these complexes can be described explicitly in terms of the expression on the right hand side, using the analysis of the transpose structure (i.e., taking into account the duality $(-)^\sharp$) presented in Section \ref{subsect:transpose}.
\end{rem}

\begin{exam}
\label{exam:koszul_complexes_kfb}
\ 
\begin{enumerate}
\item 
By restriction along the augmentation $\kk \db \rightarrow \kk \fb$, $\kk \fb$ can be considered as a $\kk \db$-module (forgetting the right $\kk \fb$-module structure).  Hence we can form the complex $$\mathscr{K}_ + \otimes_{\kk \db} \kk \fb.$$
This has underlying object $(\kk \ub)_{(+;-)} $, considered as a graded $(\kk \ub)_{(+;-)}$-module (using the length grading), for the canonical left module structure.  The differential is zero.  
\item 
Similarly, $\kk \fb$ can be considered as a $\kk \ub$-module, so that we can form the Koszul complex 
$$\hom_{\kk \ub} (\mathscr{K}_+, \kk \fb).$$
 This has underlying object  $(\kk \db)_{(+;-)}^\sharp$, considered as a graded $(\kk \db)_{(+;-)}$-module using the canonical left module structure. The differential is zero.
\end{enumerate}
There is more structure, given by also taking into account the right $\kk \fb$-module structure of $\kk \fb$.
\end{exam}

\begin{rem}
\label{rem:koszul_homological/cohomological}
\ 
\begin{enumerate}
\item 
The complexes $\mathscr{K}_+ \otimes_{\kk \db} M$ and $\mathscr{K}_+ \otimes_{(\kk \db)_{(+;-)}} N$ should be considered as cohomological complexes. For example, the first evaluated on $X$ has the following form (with $M(X)$ in cohomological degree zero):
\[
M (X) 
\rightarrow 
\bigoplus_{\substack{Y_1\subset X \\
|Y_1|=|X|- 2 \\
\ub (\emptyset, X \backslash Y_1) }}
M (Y_1)
\rightarrow 
\bigoplus_{\substack{Y_2\subset X \\
|Y_2|=|X|- 4\\
\ub (\emptyset, X \backslash Y_2)
 }}
M (Y_2)
\rightarrow 
\ldots 
\]
where the terms in cohomological degree greater than $\frac{|X|}{2}$ are zero. The differential is induced by the structure of $M$ together with the appropriate Koszul signs; namely for $Y_t \subset Y_{t-1} \subset X$, one uses the restriction map $M(Y_{t-1}) \rightarrow M (Y_t)$ (up to the signs arising from the construction of $\mathscr{K}_+$).
\item
The complexes $\hom _{\kk \ub} (\mathscr{K}_+, M')$ and $\hom_{(\kk \ub)_{(+;-)}}(\mathscr{K}_+, N')$ should be considered as homological complexes. For example, the first complex evaluated on $X$ has the following form (with $M'(X)$ in homological degree zero):
\[
\ldots 
\rightarrow 
\bigoplus_{\substack{Y_2\subset X \\
|Y_2|=|X|- 4
\\
\ub (\emptyset, X\backslash Y_2)
} }
M' (Y_2)
\rightarrow 
\bigoplus_{\substack{Y_1\subset X \\
|Y_1|=|X|- 2
\\
\ub (\emptyset, X \backslash Y_1)
} }
M' (Y_1)
\rightarrow 
M'(X),
\]
where the terms in homological degree greater than $\frac{|X|}{2}$ are zero. The differential is induced by the module structure of $M'$ (up to the appropriate signs), as above.
\end{enumerate}
\end{rem}

\begin{rem}
There are analogous Koszul complexes that are obtained by using $\mathscr{K}_-$ in place of $\mathscr{K}_+$. These can also be obtained from the complexes considered above by using the equivalences of categories given by Proposition \ref{prop:sgn_tensor_modules_categories} (see Remark \ref{rem:relate_K+_K-}).
\end{rem}

\subsection{Adjunction units and counits}

Consider the adjunction (\ref{eqn:adjunction1}); then, for $M$ a $\kk \db$-module, one has the adjunction unit 
\[
M \rightarrow \hom_{(\kk \ub)_{(+;-)}}(\mathscr{K}_+, \mathscr{K}_+ \otimes_{\kk \db} M),
\]
 a morphism of complexes of $\kk \db$-modules. 

Taking $M$ to be $\kk \fb$ as in Example \ref{exam:koszul_complexes_kfb}, this gives 
$$
\kk \fb \rightarrow \hom_{(\kk \ub)_{(+;-)}}(\mathscr{K}_+, \mathscr{K}_+ \otimes_{\kk \db} \kk \fb)
\cong 
 \hom_{(\kk \ub)_{(+;-)}}(\mathscr{K}_+, (\kk \ub)_{(+;-)}),
 $$
so that the codomain can be considered as the left dual of the complex $\mathscr{K}_+$ with respect to its $(\kk \ub)_{(+;-)}$-module structure.
 This can be rewritten as 
\begin{eqnarray}
\label{eqn:unit_fb_kkdb}
\kk \fb \rightarrow 
\kk \db ^\sharp \otimes_{\kk\fb} (\kk \ub)_{(+;-)}.
\end{eqnarray}
This is a morphism of $\kk \db \otimes \kk \fb\op$-modules, for the restricted structure on the codomain (cf.  Example \ref{exam:koszul_complexes_kfb} above).

\begin{rem}
\label{rem:understand_unit}
The underlying object of $\kk \db ^\sharp \otimes_{\kk\fb} (\kk \ub)_{(+;-)}$ is 
$$
\bigoplus_s
\kk \db ( -,\mathbf{s}) ^\sharp \otimes _{\kk\sym_s} \kk \ub_{(+;-)} (-, \mathbf{s}).
$$
Evaluating on $(U, X) $, the term indexed by $s$ is 
$\kk \db ( U,\mathbf{s}) ^\sharp \otimes _{\kk\sym_s} \kk \ub_{(+;-)} (X, \mathbf{s})$. This is zero unless $|U| \equiv s \equiv |X| \mod 2$ and $|X| \leq s \leq |U|$. Taking $X=U$, there is only one term that contributes, namely for $s = |X|$, and the adjunction unit is an isomorphism evaluated on $(X,X)$.
\end{rem}

 We have the following:

\begin{lem}
\label{lem:finiteness_projectivity}
\ 
\begin{enumerate}
\item 
The underlying  $\kk \db \otimes \kk \db_{(+;-)}$-module (forgetting the grading) of $\kk \db^\sharp \otimes_{\kk \fb} \kk \ub_{(+;-)}$ takes values in finite rank free $\kk$-modules. 
\item 
For $U$ an object of $\kk \db$,  the complex 
$\big(\kk \db ^\sharp \otimes_{\kk\fb} (\kk \ub)_{(+;-)}\big) (U, -)$ is a finite length complex in which each term is a  finitely-generated projective $(\kk \db)_{(+;-)}$-module. 
\end{enumerate}
\end{lem}

\begin{proof}
For fixed $U$, $\kk \db ( U,\mathbf{s}) ^\sharp$ is zero unless $|U|-s \in 2 \nat$, so there are only finitely many $s$ for which this is non-zero. When it is non-zero, it is a  finite rank  free $\kk \sym_s\op$-module, by Proposition \ref{prop:morphisms_right_free}. Using this, together with the fact that $\ub$ has finite hom sets, the first statement follows from the support condition given in  Remark \ref{rem:understand_unit}. The second then follows from the fact that $\kk \db ( U,\mathbf{s}) ^\sharp \otimes _{\kk\sym_s} \kk \ub_{(+;-)} (-, \mathbf{s})$ is a finitely-generated projective.
\end{proof}

Similarly, for $M'$ a $\kk\ub$-module, one has the adjunction counit for the adjunction (\ref{eqn:adjunction2}). This has the form 
\[
\mathscr{K}_+ \otimes_{(\kk \db)_{(+;-)}} \hom_{\kk \ub} (\mathscr{K}_+, M')
\rightarrow 
M',
\]
 a morphism of complexes of $\kk \ub$-modules. 

Taking $M' = \kk \fb$ as in Example \ref{exam:koszul_complexes_kfb}, this gives 
$$
\mathscr{K}_+ \otimes_{(\kk \db)_{(+;-)}} \hom_{\kk \ub} (\mathscr{K}_+, \kk \fb) 
\cong 
\mathscr{K}_+ \otimes_{(\kk \db)_{(+;-)}} (\kk \db)_{(+; -)}^\sharp
\rightarrow 
\kk \fb.
$$
This can be written as 
\begin{eqnarray}
\label{eqn:counit_kkfb}
\kk \ub \otimes_{\kk \fb} (\kk \db)_{(+; -)}^\sharp
\rightarrow 
\kk \fb.
\end{eqnarray}
The counit is a morphism of $\kk \ub \otimes \kk\fb\op$-modules, using the restricted structure  (cf. Example \ref{exam:koszul_complexes_kfb}, again). 

\begin{rem}
\label{rem:understand_counit}
The underlying object of $\kk \ub \otimes_{\kk \fb} (\kk \db)_{(+; -)}^\sharp$ is 
$$
\bigoplus_s 
\kk \ub (\mathbf{s}, -) \otimes_{\kk \fb} (\kk \db)_{(+; -)}( \mathbf{s}, -)^\sharp.
$$
Evaluating on $(V,Y) $, the term indexed by $s$ identifies as
$
\kk \ub (\mathbf{s}, V) \otimes_{\kk \fb} (\kk \db)_{(+; -)}( \mathbf{s}, Y)^\sharp.
$ 
This is zero unless $|V|\equiv s \equiv |Y| \mod 2$ and $|Y| \leq s \leq |V|$. Thus, for $V=Y$, there is only one term that contributes, namely $s = |Y|$ and the counit is an isomorphism evaluated on $(Y,Y)$.
\end{rem}

\begin{rem}
\label{rem:unit_counit_duality}
Suppose that $\kk$ is a field of characteristic zero (this allows us to identify coinvariants and invariants over the symmetric groups). 
Then the unit (\ref{eqn:unit_fb_kkdb}) and the counit (\ref{eqn:counit_kkfb}) identify under vector space duality. This is clear at the level of the underlying graded objects (using the finiteness properties given by Lemma \ref{lem:finiteness_projectivity}); that this identification respects the differential is a straightforward verification. 
\end{rem}

\begin{rem}
\label{rem:other_cases_unit_counit}
One also has the unit and counit maps
\begin{eqnarray}
N & \rightarrow & \hom_{\kk\ub} (\mathscr{K}_+, \mathscr{K}_+ \otimes_{(\kk \db)_{(+; -)}} N )
\\
\mathscr{K}_+ \otimes_{\kk \db} \hom_{(\kk\ub)_{(+;-)}} (\mathscr{K}_+, N') 
& \rightarrow & N'
\end{eqnarray}
for $N$ a $(\kk \db)_{(+;-)}$-module and $N'$ a $(\kk \ub)_{(+;-)}$-module, given by the second adjunction.
\end{rem}

We can consider the adjunctions defined using $\mathscr{K}_-$. In particular, we have the adjunction 
\[
\mathscr{K}_- \otimes_{(\kk \db)_{(-;-)}} - \quad \dashv \quad \hom_{(\kk\ub)_{(-,+)}} (\mathscr{K}_-, -) 
\]
between functors relating  complexes of $(\kk \db)_{(-;-)}$-modules and complexes of $(\kk \ub)_{(-;+)}$-modules.

For $L$ a $(\kk \db)_{(-;-)}$-module and $X$ an object of $(\kk\ub)_{(-;+)}$, as in Lemma \ref{lem:kz-complexes_identify_underlying} one has 
$$
\big(\mathscr{K}_- \otimes_{(\kk \db)_{(-;-)}} L\big) (X) 
\cong 
(\kk \ub)_{(-;+)} (-, X) \otimes _{\kk \fb} L. 
$$
For $J$ a $(\kk \ub)_{(-;+)}$-module and $Y$ an object of $(\kk \db)_{(-;-)}$, one has 
\[
\hom_{(\kk\ub)_{(-,+)}} (\mathscr{K}_-, J) (Y) 
\cong 
\hom_{\kk \fb} ((\kk \db)_{(-;-)} (Y, -) , J) 
\cong 
(\kk \db)_{(-;-)} (Y, -)^\sharp \otimes_{\kk \fb} J.
\]

Proceeding as for the $\mathscr{K}_+$-adjunctions, we obtain
\begin{eqnarray}
\label{eqn:unit_K-}
\kk \fb 
& 
\rightarrow & 
(\kk \db)_{(-;-)}^\sharp \otimes_{\kk \fb} (\kk \ub)_{(-;+)} 
\\
\label{eqn:counit_K-}
(\kk \ub)_{(-;+)} \otimes_{\kk \fb} (\kk \db)_{(-;-)}^\sharp 
&
\rightarrow &
\kk \fb. 
\end{eqnarray}

\subsection{The Koszul property}
\label{subsect:koszul_property}

Throughout this section, $\kk$ is taken to be a field of characteristic zero.
 We  review the fact that the   categories $\kk \db$ and $(\kk \db)_{(-;+)}$ are  Koszul $\kk$-linear categories over $\kk \fb$.

\begin{rem}
\ 
\begin{enumerate}
\item
This is well-known to the experts (but not necessarily stated in this form). For instance, for $\kk \db$ this is stated by Sam and Snowden \cite{MR3376738} and can be expressed in terms  of their category   $\mathrm{Rep} (\mathbf{O})$. The Koszul property is essentially equivalent to the statement of \cite[Proposition 4.3.5]{MR3376738}. The latter is proved by appealing to the general \cite[Proposition 2.3.4]{MR3376738}; this essentially encodes the Koszul complexes that we consider.
\item
Sam and Snowden observe that the Koszul property is established in \cite{MR3439686}.
\end{enumerate}
\end{rem}

\begin{thm}
\label{thm:koszul_categories}
For $\kk$ a field of characteristic zero, the $\kk$-linear categories $\kk \db$ and $\kk \db_{(-;-)}$ are Koszul $\kk$-linear categories over $\kk \fb$. 

Equivalently, the units (\ref{eqn:unit_fb_kkdb}) and (\ref{eqn:unit_K-}): 
\begin{eqnarray*}
\kk \fb 
& 
\rightarrow & 
\kk \db ^\sharp \otimes_{\kk \fb} (\kk \ub)_{(+;-)} 
\\
\kk \fb 
& 
\rightarrow & 
(\kk \db)_{(-;-)}^\sharp \otimes_{\kk \fb} (\kk \ub)_{(-;+)} 
\end{eqnarray*}
 are weak equivalences. 
\end{thm}

\begin{proof}
We give a sketch proof based on a graphical analysis of the complex, analogous to that given in Remark \ref{rem:graphical_K}.

For the first case, consider the analysis of $\kk \db ^\sharp \otimes_{\kk \fb} (\kk \ub)_{(+;-)} $ given in Remark \ref{rem:understand_unit}. Evaluating on $(U, X) \in \ob \ \kk \db \otimes \kk \db_{(+;-)}$, the term indexed by $s$ is  $\kk \db ( U,\mathbf{s}) ^\sharp \otimes _{\kk\sym_s} \kk \ub_{(+;-)} (X, \mathbf{s})$. We use the identification $\kk \db (U, \mathbf{s}) = \kk \ub ( \mathbf{s},U)$ (bearing in mind Remark \ref{rem:fbop}). 
Then, as a $\kk \sym_s\op$-module,  one has that $\kk \ub ( \mathbf{s}, U) ^\sharp$ is isomorphic to  $\kk \ub ( \mathbf{s}, U) $ (see Section \ref{subsect:transpose} for such duality considerations). 

We also have, by Lemma \ref{lem:basis_twisted_kub}, that $(\kk \ub)_{(+;-)} (X, \mathbf{s})$ is isomorphic to $\kk \ub (X, \mathbf{s})$ as a $\kk \sym_s$-module. This provides the isomorphism of $\kk$-vector spaces:
\[
\kk \db( U,\mathbf{s}) ^\sharp \otimes _{\kk\sym_s} (\kk \ub)_{(+;-)}  (X, \mathbf{s})
\cong 
\kk \ub (\mathbf{s}, U) \otimes_{\kk \sym_s} \kk \ub (X, \mathbf{s})
\cong 
\kk \big( \ub (\mathbf{s}, U) \times_{\sym_s} \ub (X , \mathbf{s})\big).
\]
This gives the basis $ \ub (\mathbf{s}, U) \times_{\sym_s} \ub (X , \mathbf{s})$. Composition in $\ub$ induces a map $\ub (\mathbf{s}, U) \times_{\sym_s} \ub (X , \mathbf{s}) \rightarrow \ub (X,U)$; using this, the set $ \ub (\mathbf{s}, U) \times_{\sym_s} \ub (X , \mathbf{s})$ can be described as elements of $\ub (X, U)$ together with a partition of the chords describing the morphisms, so as to retain the information of how they arose. Namely,  if $|U|-s = 2u $ and $s- |X| = 2x$, then $|U|-|X|= 2 (u+x)$; the chords are partitioned into subsets of cardinality $u$ and $x$ respectively. 

For example, taking $X = \mathbf{2}$, $U = \mathbf{8}$ and $s=2$, the following represents a basis element, 
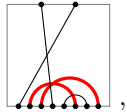
\begin{center}
\begin{tikzpicture}[scale = .15]

\begin{scope}
    \clip (0,-9) rectangle (9,-6);
    \draw [very thick, red]  (4,-9) circle(2);
    \draw [very thick, red] (5.5,-9) circle(2.5);
    \draw (6,-9) circle(1);
\end{scope}
\draw [gray] (0,0) -- (0,-9) -- (9,-9) -- (9,0) -- (0,0); 
\draw (6,0) -- (1,-9); 
\draw (3,0) -- (4,-9);
\draw [fill=black] (3,0) circle (0.2);
\draw [fill=black] (6,0) circle (0.2);
\draw [fill=black] (1,-9) circle (0.2);
\draw [fill=black] (2,-9) circle (0.2);
\draw [fill=black] (3,-9) circle (0.2);
\draw [fill=black] (4,-9) circle (0.2);
\draw [fill=black]  (6,-9) circle (0.2);
\draw [fill=black]  (7,-9) circle (0.2);
\draw [fill=black]  (8,-9) circle (0.2);
\draw [fill=black] (5,-9) circle (0.2);
\node [right] at (9,-9) {,};
\end{tikzpicture}
\end{center} 
where the thick red arcs record those that arose from $\ub (\mathbf{s}, X)$. 

Tracing through the isomorphisms used to obtain the basis, we observe that the distinguished  red chords  give rise to orientation signs associated to their order.  

The differential corresponds to taking the (signed) sum of the diagrams obtained by choosing precisely one of the red chords and forgetting that it is red; the sign arises from the orientation data.  Now, if $|U|\neq |X|$ (so that we may assume that $|U|-|X|= 2t$ for $t>0$), the acyclicity of the complex follows directly from that of the classical Koszul complex. The result follows. 

The proof in the second case is similar. There are  two differences: the chords have orientation signs associated to the direction of the chords; the orientation signs for ordering of the chords arise from the black chords rather than the red chords. The conclusion follows by using the de Rham complex in this case, rather than the Koszul complex.
\end{proof}

\begin{cor}
\label{cor:projective_resolution}
\ 
\begin{enumerate}
\item 
The complex $\big( \kk \db ^\sharp \otimes_{\kk \fb} (\kk \ub)_{(+;-)}\big) (U, -)  $ gives a projective resolution of $\kk \aut (U)$ in $(\kk \db)_{(+,-)}$-modules. 
\item 
The complex $\big( (\kk \db)_{(-;-)}^\sharp \otimes_{\kk \fb} (\kk \ub)_{(-;+)}\big) (V,-)$ (where  $V \in \ob  (\kk \db)_{(-;-)}$) gives a projective resolution of $\kk \aut (V)$ in $(\kk \db)_{(-;+)}$-modules.
\end{enumerate}
\end{cor}

\begin{proof}
The first statement follows from the Theorem in conjunction with Lemma \ref{lem:finiteness_projectivity}.  The second statement follows similarly, {\em mutatis mutandis}. 
\end{proof}

There are versions of these results for other choices of $(\pm; \mp)$.  For instance:

\begin{prop}
\label{prop:projective_resolution_kdb}
The complex $\big( (\kk \db)_{(+;-)}^\sharp \otimes_{\kk \fb} \kk \ub\big) (W,-)$ (where  $W \in \ob  (\kk \db)_{(+;-)}$) gives a projective resolution of $\kk \aut (W)$ in $\kk \db$-modules.
\end{prop}

\subsection{(Co)homological consequences} 
\label{subsect:cohom_consequences}
We continue to assume  that $\kk$ is a field of characteristic zero. 

The Koszul property of a homogeneous quadratic $\kk$-linear category has  important (co)homological consequences (indeed, the Koszul property can be {\em expressed} in terms of $\ext$ - compare \cite[Theorem 2.35]{MR4398644}). Here we focus upon the consequences for $\kk \db$-modules and for $(\kk \db)_{(-;-)}$-modules that will be used in the applications. 

For a $\kk \db$-module $M$, we have the complexes 
\begin{eqnarray*}
&&\mathscr{K}_+ \otimes_{\kk \db} M \\
&& (\kk \db)_{(+;-)}^\sharp \otimes_{(\kk \ub)_{(+;-)} } \mathscr{K}_+ \otimes_{\kk \db} M 
\ = \  (\kk \ub)_{(+;-)}^\sharp \otimes_{(\kk \ub)_{(+;-)} } \mathscr{K}_+ \otimes_{\kk \db} M 
\end{eqnarray*}
of $(\kk \ub)_{(+;-)}$-modules, where the equality uses the identity $(\kk \ub)_{(+;-)}= (\kk \db)_{(+;-)}\op$.

For a $(\kk \db)_{(-;-)}$-module $N$, we have the complexes 
\begin{eqnarray*}
&&\mathscr{K}_- \otimes_{(\kk \db)_{(-;-)} }N \\
&& (\kk \db)_{(-;+)}^\sharp \otimes_{(\kk \ub)_{(-;+)} } \mathscr{K}_- \otimes_{(\kk \db)_{(-;-)} }N 
\ = \ 
(\kk \ub)_{(-;+)}^\sharp \otimes_{(\kk \ub)_{(-;+)} } \mathscr{K}_- \otimes_{(\kk \db)_{(-;-)} }N 
\end{eqnarray*}
of $(\kk \ub)_{(-;+)}$-modules, using  $(\kk \ub)_{(-;+)}= (\kk \db)_{(-;+)}\op$ for the equality. 

\begin{rem}
\label{rem:discuss_sharp_otimes_K}
The complex $(\kk \db)_{(+;-)}^\sharp \otimes_{(\kk \ub)_{(+;-)} } \mathscr{K}_+ $ appearing above can be rewritten as
\[
\mathscr{K}_+ \otimes_{(\kk \db)_{(+;-)}}  (\kk \db)_{(+;-)}^\sharp ,
\]
which is a complex of  $\kk \ub \otimes (\kk \ub)_{(+;-)}$-modules. This is the complex appearing in the adjunction {\em counit} (\ref{eqn:counit_kkfb}). It is dual (by Remark \ref{rem:unit_counit_duality}) to the complex appearing in Proposition \ref{prop:projective_resolution_kdb}. Hence the projectivity properties of the complexes considered (Lemma \ref{lem:finiteness_projectivity}) yield {\em injective} resolutions with respect to the $(\kk \ub)_{(+;-)}$-module structure.

A similar remark applies to $(\kk \db)_{(-;+)}^\sharp \otimes_{(\kk \ub)_{(-;+)} } \mathscr{K}_-$.
\end{rem}

Since we are working over a field of characteristic zero, we also have projectivity properties:

\begin{prop}
\label{prop:second_projective_resolution}
The complex $\big( \mathscr{K}_+ \otimes_{(\kk \db)_{(+;-)}}  (\kk \db)_{(+;-)}^\sharp \big) (-,Y)$, where $Y$ is an object of $(\kk \ub)_{(+;-)}$, is a projective resolution of $\kk \aut (Y)$ in $\kk \ub$-modules with term in homological degree $\frac{1}{2} (s - |Y|)$ (assuming $s \equiv |Y| \mod 2$)
\[
\kk \ub (\mathbf{s}, -) \otimes_{\kk \sym_s} (\kk \db)_{(+;-)} (\mathbf{s}, Y)^\sharp, 
\]
which is a finitely-generated projective $\kk \ub$-module.
\end{prop}

\begin{proof}
The Koszul property implies that the counit (\ref{eqn:counit_kkfb}) is a weak equivalence. This yields the exactness of the complex. 

The terms of the complex are analysed in Remark \ref{rem:understand_counit}, which gives the stated expression.  It is not in general true that $(\kk \db)_{(+;-)} (\mathbf{s}, Y)$ is free as a $\kk \sym_s\op$-module. However,  working over a field of characteristic zero, since the category of $\kk \sym_s$-modules is semisimple, the terms of the complex are finitely-generated projectives. 
\end{proof}

\begin{rem}
The resolution given by Proposition \ref{prop:second_projective_resolution} does not have finite length, since  $(\kk \db)_{(+;-)} (\mathbf{s}, Y)$ is  non-zero for infinitely many $s \in \nat$.
\end{rem}

For clarity, we first treat the $\kk \db$-module case (see Theorem \ref{thm:ext_tor_variant} for a counterpart of the following). 

\begin{thm}
\label{thm:ext_tor}
For $M$ a $\kk \db$-module, there are natural isomorphisms:
\begin{eqnarray*}
H^* (\mathscr{K}_+ \otimes_{\kk \db} M) & \cong & \ext^* _{\kk \db} (\kk \fb, M) \\
H_* ((\kk \db)_{(+;-)}^\sharp \otimes_{(\kk \ub)_{(+;-)} } \mathscr{K}_+ \otimes_{\kk \db} M )
&\cong & 
\tor_*^{\kk \db} (\kk \fb, M).
\end{eqnarray*}
\end{thm}

\begin{proof}
For the first statement, we require to prove that, for any finite set $W$, considering $\kk \aut (W)$ as a $\kk \db$-module, there is a natural isomorphism
\begin{eqnarray}
\label{eqn:ext_W}
H^* \big((\mathscr{K}_+ \otimes_{\kk \db} M)(W)\big) 
\cong 
\ext^* _{\kk \db} (\kk \aut (W), M).
\end{eqnarray}
Now, by Proposition \ref{prop:projective_resolution_kdb}, the complex $\big( (\kk \db)_{(+;-)}^\sharp \otimes_{\kk \fb} \kk \ub\big) (W,-)$ gives a projective resolution of $\kk \aut (W)$ in $\kk \db$-modules. Thus the $\ext$ group is  the cohomology of the complex 
\[
\hom_{\kk \db} \big( ((\kk \db)_{(+;-)}^\sharp \otimes_{\kk \fb} \kk \ub) (W,-), M\big).
\]
The underlying object of this complex is naturally isomorphic to $\hom_{\kk \fb} ( (\kk \db)_{(+;-)}(W, -) ^\sharp , M)$ (using the identification given in Remark \ref{rem:understand_unit}) and hence to $(\kk\db)_{(+;-)} (W,-) \otimes_{\kk \fb} M$ by the strong dualizability of $(\kk \db)_{(+;-)}(W, \mathbf{s})$ as a $\kk \sym_s$-module and the fact that there are only finitely many $s$ for which this is non-zero. 

Since we have $(\kk\db)_{(+;-)} (W,-)= (\kk\ub)_{(+;-)} (-,W)$ by Lemma \ref{lem:kz-complexes_identify_underlying}, $(\kk\db)_{(+;-)} (W,-) \otimes_{\kk \fb} M$ is isomorphic to the underlying object of $(\mathscr{K}_+ \otimes_{\kk \db} M)(W)$. That the differentials correspond under this isomorphism is proved by a direct verification from the definitions. 
Thus we have an isomorphism of complexes and this implies the required isomorphism (\ref{eqn:ext_W}), completing the proof of the first statement. 

The proof of the second statement is more direct, using the projective resolution of $\kk \fb$ given by Proposition \ref{prop:second_projective_resolution}.
\end{proof}

\begin{exam}
\label{exam:case_M=kkfb}
Taking $M= \kk \fb$ considered as a $\kk \fb$-module, one has (by Theorem \ref{thm:ext_tor}) the isomorphism
\[
\ext^*_{\kk \db}(\kk \fb, \kk \fb) \cong \mathscr{K}_+ \otimes_{\kk \db} \kk \fb \cong (\kk \ub)_{(+;-)}, 
\]
since the complex has zero differential.
 This  encodes the Koszul duality between $\kk \db$ and $\kk \ub_{(+;-)}$. 

The $\tor$ statement is dual: 
\[
\tor_*^{\kk \db} (\kk \fb, \kk \fb) \cong (\kk \db)_{(+;-)}^\sharp .
\]
\end{exam}

Recall that, for $M$ a $\kk \db$-module, $\mathscr{K}_+ \otimes_{\kk \db} M$  and $ (\kk \db)_{(+;-)}^\sharp \otimes_{(\kk \ub)_{(+;-)} } \mathscr{K}_+ \otimes_{\kk \db} M $ are both complexes of $(\kk \ub)_{(+;-)}$-modules, where the module structure shifts the (co)homological degree according to the degree of $(\kk\ub)_{(+;-)}$ (corresponding to the length grading). Hence their (co)homology inherits such a module structure. 

The identification of $\ext^*_{\kk \db}(\kk \fb, \kk \fb)$ given in Example \ref{exam:case_M=kkfb} allows the following addendum to Theorem \ref{thm:ext_tor} to be stated: 

\begin{prop}
\label{prop:yoneda_and_cap_product}
For $M$ a $\kk \db$-module,
\begin{enumerate}
\item 
the graded  $(\kk \ub)_{(+;-)}$-module structure on the cohomology of $\mathscr{K}_+ \otimes_{\kk \db} M$ identifies as the module structure of $\ext^* _{\kk \db} (\kk \fb, M)$ over $\ext^*_{\kk \db}(\kk \fb, \kk \fb)$ with respect to the Yoneda product;
\item 
the graded $(\kk \ub)_{(+;-)}$-module structure on the homology of $ (\kk \db)_{(+;-)}^\sharp \otimes_{(\kk \ub)_{(+;-)} } \mathscr{K}_+ \otimes_{\kk \db} M $ identifies as the 
module structure of $\tor_*^{\kk \db} (\kk \fb, M)$ over $\ext^*_{\kk \db}(\kk \fb, \kk \fb)$ given by the cap product.
\end{enumerate}
\end{prop}

\begin{proof}
This is a standard identification in the Koszul context, proved using the explicit nature of the resolutions.
\end{proof}

There are counterparts of these results for $(\kk \db)_{(-;-)}$-modules:

\begin{thm}
\label{thm:ext_tor_variant}
For $N$ a $(\kk \db)_{(-;-)}$-module, there are natural isomorphisms
\begin{eqnarray*}
H^* (\mathscr{K}_- \otimes_{(\kk \db)_{(-;-)}} N) & \cong & \ext^* _{(\kk \db)_{(-;-)}} (\kk \fb, N) \\
H_* ((\kk \db)_{(-;+)}^\sharp \otimes_{(\kk \ub)_{(-;+)} } \mathscr{K}_- \otimes_{(\kk \db)_{(-;-)}} N )
&\cong & 
\tor_*^{(\kk \db)_{(-;-)}} (\kk \fb, N).
\end{eqnarray*}
In particular, one has the identification $\ext^* _{(\kk \db)_{(-;-)}} (\kk \fb , \kk \fb) \cong (\kk \fb)_{(-;+)}$ and 
\begin{enumerate}
\item 
the natural $(\kk \ub)_{(-;+)}$-module structure on $H^* (\mathscr{K}_- \otimes_{(\kk \db)_{(-;-)}} N)$ identifies with the $\ext^* _{(\kk \db)_{(-;-)}} (\kk \fb , \kk \fb)$-module structure on $\ext^* _{(\kk \db)_{(-;-)}} (\kk \fb, N)$ given by the Yoneda product; 
\item 
the natural $(\kk \ub)_{(-;+)}$-module structure on  $H_* ((\kk \db)_{(-;+)}^\sharp \otimes_{(\kk \ub)_{(-;+)} } \mathscr{K}_- \otimes_{(\kk \db)_{(-;-)}} N )$ identifies with the 
$\ext^* _{(\kk \db)_{(-;-)}} (\kk \fb , \kk \fb)$-module structure on $\tor_*^{(\kk \db)_{(-;-)}} (\kk \fb, N)$ given by the cap product.
\end{enumerate}
\end{thm}

\subsection{Relating $\ext$ and $\tor$}
\label{subsect:relate_ext_tor}

The purpose of this section is to sketch the relationship between the $\ext$ and $\tor$ groups that interest us, working over a field $\kk$ of characteristic zero.

We focus on the complexes for $(\kk \db)_{(-;-)}$-modules:
\begin{eqnarray}
\label{eqn:N1}
&&\mathscr{K}_- \otimes_{(\kk \db)_{(-;-)} }N \\
\label{eqn:N2}
&& (\kk \ub)_{(-;+)}^\sharp \otimes_{(\kk \ub)_{(-;+)} } \mathscr{K}_- \otimes_{(\kk \db)_{(-;-)} }N .
\end{eqnarray}

Here, the complex (\ref{eqn:N2}) is obtained from (\ref{eqn:N1}) by applying the functor $(\kk \ub)_{(-;+)}^\sharp \otimes_{(\kk \ub)_{(-;+)} } -$.  Moreover, the complex (\ref{eqn:N1}) has underlying graded $(\kk \ub)_{(-;+)} $-module that is projective. It follows that there is a universal coefficients spectral sequence that relates the two: 
$$
\tor^{(\kk \ub)_{(-;+)} } _* \big(
(\kk \ub)^\sharp _{(-;+)}, \ext^*_{(\kk \db)_{(-;-)} } (\kk \fb , N) 
\big) 
\Rightarrow 
\tor_*^{(\kk \db)_{(-;-)} } (\kk \fb , N).
$$

In particular, we have the  edge homomorphism:
\begin{eqnarray}
\label{eqn:edge1}
(\kk \ub)^\sharp _{(-;+)}
\otimes_{
(\kk \ub)_{(-;+)}
}
 \ext^*_{(\kk \db)_{(-;-)} } (\kk \fb , N) 
\rightarrow 
\tor_*^{(\kk \db)_{(-;-)} } (\kk \fb , N);
\end{eqnarray}
this is compatible with the degrees (working with cohomological degree, for example). 

\begin{rem}
Note that, with respect to cohomological degrees, the $\ext^*$ term above is concentrated in non-negative degrees, whereas $(\kk \ub)^\sharp _{(-;+)}$ and the $\tor_*$ term are concentrated in non-positive degrees. The only {\em direct} relation between $\ext^*$ and $\tor_*$ (i.e., without the intervention of negative degree from $(\kk \ub)^\sharp _{(-;+)}$) is the natural transformation:
\[
\ext^0_{(\kk \db)_{(-;-)} } (\kk \fb , N) 
\rightarrow 
\tor_0^{(\kk \db)_{(-;-)} } (\kk \fb , N).
\]
This identifies as the usual natural transformation from `invariants' (or `primitives') to `indecomposables'. 
\end{rem}

\begin{rem}
\label{rem:ext_tor_degree_0}
In considering the edge homomorphism (\ref{eqn:edge1}), observe that:
\begin{enumerate}
\item 
$\tor_*^{(\kk \db)_{(-;-)} } (\kk \fb , N)$ is a $(\kk \ub)_{(-;+)}$-torsion module; 
\item 
the module 
$
(\kk \db)^\sharp _{(-;+)}
\otimes_{
(\kk \ub)_{(-;+)}
}
 \ext^*_{(\kk \db)_{(-;-)} } (\kk \fb , N) $ only depends on the torsion-free quotient of $\ext^*_{(\kk \db)_{(-;-)} } (\kk \fb , N) $, by
 Corollary \ref{cor:pass_torsion-free_quotient}.
\end{enumerate}
\end{rem}

One can also exploit the functor $\hom_{(\kk \ub)_{(-;+)}} ((\kk \ub)_{(-;+)}^\sharp, -)$ applied to torsion $(\kk \ub)_{(-;+)}$-modules (see Section \ref{subsect:adjunction}). By Corollary \ref{cor:equivalence_full_subcategories}, one can recover (\ref{eqn:N1}) from (\ref{eqn:N2}) by applying this functor. One has an associated universal coefficients spectral sequence. In particular there is an edge morphism
\begin{eqnarray} 
\label{eqn:edge2}
 \ext^*_{(\kk \db)_{(-;-)} } (\kk \fb , N) 
\rightarrow 
\hom_{(\kk \ub)_{(-;+)}} ((\kk \ub)_{(-;+)}^\sharp, 
\tor_*^{(\kk \db)_{(-;-)} } (\kk \fb , N)
).
\end{eqnarray}
This should identify as  the mate of (\ref{eqn:edge1}).

\begin{rem}
The two edge homomorphisms (\ref{eqn:edge1}) and (\ref{eqn:edge2}) give a natural way of relating $\ext$ and $\tor$. They underline the importance of understanding the 
respective $(\kk \ub)_{(-;+)}$-module structures on these. 
\end{rem}

\section{Using symplectic and orthogonal vector spaces}
\label{sect:forms}

The purpose of this section is to review aspects of Brauer-Schur-Weyl duality for both the orthogonal and the symplectic cases. This goes back to fundamental results due to Weyl (see \cite{MR1488158}, for example).

In the symplectic case,  this  provides a bridge between $(\kk \ub)_{(-;+)}$-modules  and functors on the category of symplectic vector spaces, inspired by work of Sam and Snowden such as  \cite{MR3376738} and \cite{MR3876732}. In particular, this gives the precise notion of {\em stabilization} of (suitable) functors on the category of symplectic vector spaces.

Throughout this section $\kk$ is a field of characteristic zero.

\subsection{Modules from symplectic and orthogonal vector spaces}

\begin{defn}
\label{defn:forms}
\ 
\begin{enumerate}
\item 
Let $\vo$, the category of (split) orthogonal vector spaces,  be the category with objects $(V,\sbf )$, where $V$ is a  finite dimensional $\kk$-vector space and  $\sbf $ is a non-degenerate symmetric bilinear form $\sbf  \colon S^2 (V) \rightarrow \kk$ such that $V$ admits an orthonormal basis with respect to $\sbf$;  morphisms are $\kk$-linear maps between such spaces that preserve the bilinear form. 
\item 
Let $\vsp$, the category of symplectic vector spaces,  be the category with objects $(V, \omega)$, where $V$ is a finite dimensional $\kk$-vector space and $\omega$ is a non-degenerate symplectic form $\omega : \Lambda^2 (V) \rightarrow \kk$;  morphisms are $\kk$-linear maps between such spaces that preserve the symplectic form.
\end{enumerate}
\end{defn}

\begin{rem}
\ 
\begin{enumerate}
\item 
A symmetric bilinear form $(V, \sbf )$ yields a linear map $V^{\otimes 2} \rightarrow \kk$ that factors across  the canonical projection $V^{\otimes 2} \twoheadrightarrow S^2 (V)$; non-degeneracy is equivalent to the property that this induces an isomorphism $V \cong V^\sharp$. Likewise for a symplectic form, this time using the canonical projection $V^{\otimes 2} \twoheadrightarrow \Lambda^2 (V)$. 
\item 
If $(V, \sbf _V) \rightarrow (W, \sbf _W)$ is a morphism in $\vo$, since the forms are non-degenerate, the underlying $\kk$-linear map $V \rightarrow W$ is injective. Moreover, this admits a canonical retract, by using the orthogonal complement to the image of $V$ in $W$. An analogous statement holds for a morphism in $\vsp$.
\end{enumerate}
\end{rem}

Recall that $\finj$ is the category of finite sets and bijections. In the following we implicitly use the skeleton with objects $\{ \n \mid n \in \nat\}$.

\begin{prop}
\label{prop:finj_vo_vsp}
There are faithful functors 
\begin{enumerate}
\item 
$\finj  \rightarrow  \vo$ given by $\n  \mapsto  (\kk \n, \sbf)$,  
where $\n$ is an orthonormal basis for $\sbf$; 
\item 
$\finj  \rightarrow \vsp$ given by $\n  \mapsto  (\kk (\n \times \{p, q\}), \omega)$,  
where $\{ (i, p) , (i, q) \mid i \in \n\}$ is a symplectic basis (so that $\omega ((i,p) , (j,q))=1$ if $i=j$ and is  zero otherwise).
\end{enumerate}
Moreover, the respective images of the skeleton of $\finj$ define skeleta of $\vo $ and $\vsp$ respectively.
\end{prop}

\begin{proof}
It is clear that, forgetting the respective forms, the statement defines two (non-isomorphic) faithful functors $\finj \rightarrow \kmod$. To establish that we have functors as claimed, it remains to show that these factor across the respective forgetful functors $\vo \rightarrow \kmod$ and $\vsp \rightarrow \kmod$. 
This is clear from the definition of the respective forms. Finally, it is immediate that this provides skeleta as stated.
\end{proof}

For $(V, \sbf)$ an object of $\vo$, we have the automorphism group $\ogp (V, \sbf)$, which identifies with the endomorphisms of $(V, \sbf)$ in $\vo$. Likewise, for $(V, \omega)$ an object of $\vsp$, we have the automorphism group $\spgp (V, \omega)$. Using the  inclusions $(\kk ^n, \sbf ) \hookrightarrow (\kk^{n+1}, \sbf)$ (respectively $(\kk^{2n}, \omega) \hookrightarrow (\kk ^{2(n+1)}, \omega)$) induced by the canonical inclusion $\mathbf{n} \subset \mathbf{n+1}$ and Proposition \ref{prop:finj_vo_vsp}, 
 one can define the groups:
\begin{eqnarray*}
\ogp _\infty &:=& \bigcup_{n\in \nat} \ogp (\kk^n, \sbf)
\\
\spgp _\infty &:=& \bigcup_{n \in \nat} \spgp (\kk^{2n}, \omega) 
\end{eqnarray*}
and the respective categories of representations of these,  $\rep (\ogp_\infty)$ and $\rep (\spgp_\infty)$.

We also consider the related functor categories, adopting the following notation for concision.

\begin{nota}
\label{nota:F(C)}
For $\calc$ an essentially small category, write $\f (\calc)$ for the category of functors from $\calc$ to $\kmod$. 
\end{nota}

The functor categories $ \f (\vo)$ and $\f (\vsp)$ are related to the categories of representations $\rep (\ogp_\infty) $ and $\rep (\spgp_\infty)$ by the following exact functors, which can be thought of as stabilization functors; they are fundamental ingredients in Sam and Snowden's study of stability patterns \cite{MR3376738,MR3876732}

\begin{defn}
\label{defn:stabilization_functors}
Let  
\begin{eqnarray*}
\label{eqn:fvo_rep}
\stabvo &: &
\f (\vo)  \rightarrow  \rep (\ogp_\infty) 
\\
\label{eqn:fvsp_rep}
\stabvsp &: &
\f (\vsp) \rightarrow  \rep (\spgp_\infty)
\end{eqnarray*}
be the functors induced by passage to  the colimit under the restriction (via Proposition \ref{prop:finj_vo_vsp}) to the sequence of inclusions $ \ldots \subset \n \subset \mathbf{n+1} \subset \ldots $. 

For instance, for $F$ in $\f (\vo)$, 
$$
\stabvo F := 
\lim_{\substack{\longrightarrow \\ n \rightarrow \infty}} F (\kk\n , \sbf),
$$
equipped with the natural $\ogp_\infty$-module structure.
\end{defn}

Clearly we have the following: 

\begin{prop}
\label{prop:stab_exact}
The functors $\stabvo$ and $\stabvsp$ are exact. 
\end{prop}

\subsection{The tensor powers}

Fundamental examples of objects of $\f(\vo)$ (respectively $\f (\vsp)$) are provided by the tensor power functors, as follows. 

\begin{exam}
\label{exam:tensor_functors_vo_vsp}
We have the forgetful functor $(V, \sbf ) \mapsto V$, which can be considered as an object of $\f (\vo)$. Then, for $d \in \nat$, this can be postcomposed with the $d$th tensor power functor to give  $\ot^d$ in $\f (\vo)$ given by 
$$
\ot^d : (V, \sbf) \mapsto V^{\otimes d}.
$$
Similarly,  we have the functor $\spt^d$ in $\f (\vsp)$ given by 
$$
\spt^d : (V, \omega) \mapsto V^{\otimes d}.
$$

In both cases, the group $\sym_d$ acts (by natural transformations)  by place permutations of the tensor factors.
\end{exam}

The functors $\ot^d$, for $d \in \nat$, have further naturality properties. Namely, for $(V, \sbf)$ an object of $\vo$, the linear map $V^{\otimes 2} \rightarrow \kk$ given by the form is `natural' in the sense that this corresponds to  a natural transformation 
$$
\ot^2 \rightarrow \ot^0 \cong \kk,
$$
where $\ot^0$ identifies as the constant functor $\kk$. This is $\sym_2$-equivariant, where $\sym_2$ acts trivially on the codomain. 
 Similarly, in the symplectic case, we have the  map 
$$
\spt^2 \rightarrow \spt^0 \cong \kk. 
$$ 
Letting $\sym_2$ act via the sign representation $\sgn_2$ on the codomain, this is $\sym_2$-equivariant.

This extends to give the following well-known result:

\begin{prop}
\label{prop:tensor_functors_sp_o}
\ 
\begin{enumerate}
\item 
The functors $\ot^d$ (for $d \in \nat$) assemble to a functor
\begin{eqnarray*}
\ot^\bullet  \colon \db \times \vo &\rightarrow &\kmod 
\\
(\mathbf{d}, (V, \sbf )) &\mapsto &V^{\otimes d};
\end{eqnarray*}
automorphisms of $\mathbf{d}$ act via place permutations on $V^{\otimes d}$ and the morphism $\db (\mathbf{d+2}, \mathbf{d})$ corresponding to the canonical inclusion $\mathbf{d} \subset \mathbf{d+2}$ acts via the linear map $V^{\otimes d+2} \rightarrow V^{\otimes d}$ given by applying the form $\sbf $ to the last two tensor factors. 
\item 
The functors $\spt^d$ (for $d \in \nat$) assemble to a $\kk$-linear functor
\begin{eqnarray*}
\spt^\bullet \colon (\kk \db)_{(-;+)} \otimes \kk \vsp &\rightarrow &\kmod 
\\
(\mathbf{d}, (V, \omega)) \mapsto V^{\otimes d};
\end{eqnarray*}
$\kk \sym_d$  acts  via  place permutations on $V^{\otimes d}$ and the generator of $(\kk \db)_{(-;+)} (\mathbf{d+2}, \mathbf{d})$ corresponding to the canonical inclusion $\mathbf{d} \subset \mathbf{d+2}$ and the canonical order on $\{d+1, d+2\}$  acts via the linear map $V^{\otimes d+2} \rightarrow V^{\otimes d}$ given by applying the form $\omega$ to the last two tensor factors. 
\end{enumerate}
\end{prop}

\begin{rem}
We have the following equivalent formulations: 
\begin{enumerate}
\item 
$\ot^\bullet$ is  a $\kk$-linear functor $\kk\db \rightarrow \f (\vo)$, given by $\mathbf{d} \mapsto \ot^d$; 
\item 
$\spt^\bullet$ is  a $\kk$-linear functor $(\kk \db)_{(-;+)} \rightarrow \f (\vsp)$, given by $\mathbf{d} \mapsto \spt^d$.
\end{enumerate}
\end{rem}

One has  the following  consequence of the  fundamental theorems of invariant theory. (Sam and Snowden mostly work over $\kk = \mathbb{C}$; however, they indicate that their results hold over any field of characteristic zero, if one works with split forms of the orthogonal group \cite[Section 1.7]{MR3376738}.)

\begin{thm} 
\label{thm:vker_fully-faithful}
\cite[Sections 4.1 and 4.2]{MR3376738}
The $\kk$-linear functors 
\begin{eqnarray*}
 \ot^\bullet  &\colon &\kk\db \rightarrow \f (\vo)
 \\ 
\spt^\bullet  &\colon &(\kk\db)_{(-;+)} \rightarrow \f (\vsp) 
\end{eqnarray*}
are fully faithful. 
\end{thm}

\subsection{Torsion}

The notion of torsion for $\kubg$-modules considered in Section \ref{subsect:torsion} has a counterpart here. 

\begin{defn}
\label{defn:torsion_fvo_fvsp}
For $F$ an object of $\f (\vo)$, 
\begin{enumerate}
\item 
a section $x \in F(V, \sbf)$ is torsion if the subfunctor generated by $x$ has finite support; 
\item 
the functor $F$ is torsion if all of its sections are torsion; 
\item 
the functor $F$ is torsion-free if it has no non-zero torsion subobject.
\end{enumerate}
The full subcategory of torsion functors is written $\f_\tors (\vo)$.

Torsion for objects of $\f (\vsp)$ is defined similarly, leading to the full subcategory $\f _\tors (\vsp)$. 
\end{defn}

\begin{exam}
\label{exam:torsion-examples}
For $d \in \nat$, the functor $\ot^d$ in $\f (\vo)$ is torsion-free. Similarly, $\spt^d$ is torsion-free in $\f (\vsp)$.
\end{exam}

We have the following counterpart of Theorem \ref{thm:localizing_serre}:

\begin{prop}
\label{prop:localizing_serre_vo_vsp}
The category $\f_\tors (\vo)$  (respectively $\f_\tors (\vsp)$) is a localizing Serre subcategory of $\f (\vo)$ (resp. $\f (\vsp)$. 
\end{prop}

We thus have the respective quotient categories, localizing away from the torsion: 
\begin{eqnarray*}
\f (\vo) & \rightarrow & \f (\vo)/ \f_\tors (\vo) 
\\
\f (\vsp) & \rightarrow & \f (\vsp)/ \f_\tors (\vsp). 
\end{eqnarray*}

Moreover, since the restriction of $\stabvo$ to $\f_\tors (\vo)$ is zero (respectively that of $\stabvsp$ to
$\f _\tors (\vsp)$), we have the following, by the universal  property of the quotient category construction: 

\begin{prop}
\label{prop:factor_stab}
The stabilization functors factor respectively across 
\begin{eqnarray*}
\stabvo 
&:& 
\f (\vo)/ \f_\tors (\vo) 
\rightarrow 
\rep (\ogp_\infty)
\\
\stabvsp
&:& 
\f (\vsp)/ \f_\tors (\vsp) 
\rightarrow 
\rep (\spgp_\infty).
\end{eqnarray*}
\end{prop}

\subsection{Generalized Schur functors}
\label{subsect:gen_Schur}

Recall that the classical Schur functor construction  relates $\kk \fb$-modules with functors on the category of finite-dimensional vector spaces. 
For $M$ a $\kk \fb$-module, the associated Schur functor is given by 
\[
V \mapsto M(V) := V^{\otimes \bullet} \otimes_{\kk \fb} M = \bigoplus_{n} V^{\otimes n}\otimes_{\kk \sym_n} M(\n),  
\]
where $V ^{\otimes \bullet}$ denotes the $\kk \fb\op$-module $\n \mapsto V^{\otimes n}$, with (right) place permutation action of $\sym_n$.

\begin{rem}
\label{rem:ambiguity}
There is some ambiguity in the notation for $\kk\fb$-modules and their associated Schur functors.  However, for a $\kk \fb$-module $M$, the type of an object on which $M(-)$ is evaluated determines the meaning.
\end{rem}

For usage later, we introduce the following

\begin{nota}
\label{nota:schur_components}
For $M$ a $\kk \fb$-module and $n \in \nat$, 
 write $M_n (V)$ for the direct summand of the Schur functor $M(V)$ given by $V^{\otimes n}\otimes_{\kk \sym_n} M(\n)$. (This is the homogeneous component of $M(V)$ of polynomial degree $n$.)
\end{nota}

In the context of $\kk \ub$-(respectively $(\kk\ub)_{(-;+)}$-)modules,  the functors $\ot^\bullet$  (respectively $\spt^\bullet$) give rise to generalized Schur functors (this  terminology is inspired by \cite{MR3876732}): 

\begin{lem}
\label{lem:tensor_with_vker}
There are  right exact functors
\begin{eqnarray*}
\ot^\bullet  \otimes_{\kk \ub} -  &\colon & \kk\ub\dash\modules \rightarrow \f (\vo)
\\
\spt^\bullet \otimes_{(\kk\ub)_{(-;+)}} - & \colon & (\kk \ub)_{(-;+)}\dash\modules \rightarrow \f (\vsp).
\end{eqnarray*}
\end{lem}

\begin{rem}
One has the induced $\kk \ub$-module functor $\kk \ub \otimes _{\kk \fb} - \colon \kk\fb \dash \modules \rightarrow \kk \ub \dash \modules$.  This can be composed with the functor $\ot^\bullet  \otimes_{\kk \ub} -$ to give the composite 
\[
\ot ^\bullet \otimes_{\kk \ub} \kk \ub \otimes _{\kk \fb} - \colon \kk\fb \dash \modules \rightarrow  \f (\vo). 
\]
This identifies as 
$$
\ot^\bullet  \otimes_{\kk \fb} - \colon \kk \fb\dash \modules \rightarrow \f (\vo),
$$
 using the restriction of the $\kk \db$-module structure of $\ot^\bullet$ to $\kk \fb$.  This is the restriction  to $\vo$ (via the forgetful functor $(V, \sbf )\mapsto V$)  of the classical Schur functor.

The corresponding statement holds in the symplectic case.
\end{rem}

\subsection{The harmonic subfunctors}

Proposition \ref{prop:tensor_functors_sp_o} allows us to consider the respective `harmonic' subfunctors of the tensor functors, following \cite{MR1488158} and using the terminology of \cite[Chapter 10]{MR2522486}, for example.

\begin{defn}
\label{defn:harmonic}
For $d \in \nat$, let  
\begin{enumerate}
\item 
$\ot^{[d]}$ be the subfunctor of $\ot^d$ in $\f (\vo)$ defined as the kernel of the morphism
$$
\ot^d \rightarrow \kk\db (\mathbf{d} , \mathbf{d-2}) ^\sharp \otimes \ot^{d-2}
$$
adjoint to the structure  map $\kk\db (\mathbf{d} , \mathbf{d-2})\otimes \ot^d \rightarrow  \ot^{d-2}$
 for the $\kk \db$-module structure of $\ot^\bullet$ (the codomain is understood to be zero if $d <2$);
\item 
$\spt^{[d]}$ be the subfunctor of $\spt^d$  in $\f (\vsp)$ defined as the kernel of the morphism
$$
\spt^d \rightarrow (\kk\db)_{(-;+)} (\mathbf{d} , \mathbf{d-2}) ^\sharp \otimes \spt^{d-2}
$$
adjoint to the structure map for the $(\kk\db)_{(-;+)}$-module structure of $\spt^\bullet$ (where the codomain is zero if $d <2$).
\end{enumerate}
\end{defn}

\begin{rem}
The functor $\ot^{[d]}$ is torsion-free in $\f (\vo)$, as a subfunctor of a torsion-free functor. Moreover the $\sym_d$-action on $\ot^d$ restricts to a $\sym_d$-action on $\ot^{[d]}$. Hence, if $M$ is a (right) $\rat \sym_d$-module, we may form the functor $M \otimes_{\rat \sym_d} \ot^{[d]}$ in $\f (\vo)$, which is again torsion-free. 
The corresponding statements hold for the symplectic case.
\end{rem}

\subsection{The algebraic subcategories}
Following Sam and Snowden \cite{MR3876732}, we introduce the categories of algebraic objects. Restricting to algebraic objects facilitates the analysis of the localization away from torsion.

\begin{defn}
\label{defn:falg}
An object of $\f (\vo)$ (respectively  $\f (\vsp)$) is algebraic if it is a subquotient of a direct sum (possibly infinite) of objects of the form $\ot^{d}$ (resp. $\spt^{d}$) (for varying $d \in \nat$).

The full subcategories of algebraic objects are denoted respectively by $\falg (\vo)$ and $\falg (\vsp)$.
\end{defn}

\begin{rem}
Sam and Snowden show that the categories  $\falg (\vo)$ and $\falg (\vsp)$ have good properties: they are Grothendieck abelian categories; moreover, they are locally noetherian (by \cite[Theorem 2.11]{MR3876732}).
\end{rem}

\begin{exam}
For a $\kk \ub$-module $N$, the generalized Schur functor $\ot^\bullet \otimes_{\kk \ub} N$ is algebraic in $\f(\vo)$. 
(There is an analogous statement for $\f (\vsp)$.)
\end{exam}

One has the respective full subcategories of torsion objects, $\falg_\tors (\vo)$ and $\falg_\tors (\vsp)$. These are localizing Serre subcategories of $\falg (\vo)$ and $\falg (\vsp)$ respectively. In particular, this allows us to form the quotient categories $  \falg (\vo)/ \falg_\tors (\vo) $ and 
$  \falg (\vsp)/ \falg_\tors (\vsp).$

There is a counterpart of Definition \ref{defn:falg} for the categories $\rep (\ogp_\infty)$ and $\rep (\spgp_\infty)$ (see \cite[Section 2.2]{MR3876732}).

\begin{nota}
\label{nota:rep_algebraic}
Write $\rep (\ogp)$ (respectively $\rep (\spgp)$) for the full subcategory of algebraic representations in $\rep (\ogp_\infty)$ (resp. $\rep (\spgp_\infty)$).
\end{nota}

\begin{rem}
\label{rem:orthogonal-symplectic_duality}
Sam and Snowden establish the orthogonal-symplectic duality result \cite[Theorem 4.3.4]{MR3376738}, giving the equivalence of categories $\rep (\ogp) \simeq \rep (\spgp)$ (note that, in {\em loc. cit.}, the authors restrict to finite length objects). This is the counterpart of the equivalence between $\kk \ub$-modules and $(\kk \ub)_{(-;+)}$-modules that is given by Proposition \ref{prop:sgn_tensor_modules_categories}. 
\end{rem}

These categories of algebraic objects are related by using the stabilization functors of Definition \ref{defn:stabilization_functors}: 

\begin{thm}
\cite[Theorem 2.5]{MR3876732}
\begin{enumerate}
\item 
The stabilization functor $\stabvo$ restricts to  $\stabvo : \falg (\vo) \rightarrow \rep (\ogp)$ and this induces an equivalence of categories 
$$
\falg (\vo) / \falg_\tors (\vo) 
\stackrel{\simeq}{\rightarrow} 
\rep (\ogp).
$$
\item 
The stabilization functor $\stabvsp$ restricts to $\stabvsp : \falg (\vsp) \rightarrow \rep (\spgp)$ and this induces an equivalence of categories 
$$
\falg (\vsp) / \falg_\tors (\vsp) 
\stackrel{\simeq}{\rightarrow} 
\rep (\spgp).
$$
\end{enumerate}
\end{thm}

Moreover, we have the important property:

\begin{thm}
\label{thm:injectives_falgvo_falgvsp}
\cite[Corollary 2.6]{MR3876732}
For any $d \in \nat$:
\begin{enumerate}
\item
the functor $\ot^d$ is injective in $\falg (\vo)$ and identifies as the injective envelope of $\ot^{[d]}$; 
\item 
the functor $\spt^d$ is injective in $\falg (\vsp)$ and identifies as the injective envelope of $\spt^{[d]}$. 
\end{enumerate}
\end{thm}

Sam and Snowden in \cite[Proposition 4.1.4]{MR3376738} give a complete (and explicit) classification of  the simple objects of the quotient categories modulo torsion. Combining this with  Theorem \ref{thm:injectives_falgvo_falgvsp}, one deduces:

\begin{cor}
\label{cor:injective_cogenerators}
\ 
\begin{enumerate}
\item 
The set $\{ \ot^d \mid d \in \nat \}$  is a set of injective cogenerators of $\falg (\vo) / \falg_\tors (\vo)$. 
\item 
The set  $\{ \spt^d \mid d \in \nat \}$  is a set of injective cogenerators of $\falg (\vsp) / \falg_\tors (\vsp)$. 
\end{enumerate}
\end{cor}

\subsection{Analysing algebraic functors}
\label{subsect:analyse_algebraic}

The object $\ot^\bullet$ yields the functor 
$$
\hom_{\f (\vo)} (-, \ot^\bullet ) \colon 
\f (\vo) \op 
\rightarrow 
\kk \db \dash \modules.
$$
Moreover, as a case of Sam and Snowden's transforms by kernels \cite[Section 2]{MR3376738}, this has right adjoint 
$$
\hom_{\kk \db} (-, \ot^\bullet ) \colon 
\kk \db \dash \modules  \op 
\rightarrow 
\f (\vo).
$$

\begin{rem}
The adjunction corresponds to the fact that, for $F$ in $\f (\vo)$ and a $\kk \db$-module $M$, there is a natural isomorphism
$$
\hom_{\kk \db} (M, \hom_{\f (\vo)} (F, \ot^\bullet ))
\cong 
\hom_{\f (\vo)} (F, \hom_{\kk \db} (M, \ot^\bullet)).
$$
Both expressions are isomorphic to 
$
\hom_{\kk \db \otimes \kk \vo} (M \boxtimes F, \ot^{\bullet}),
$
where we consider $\ot^{\bullet}$ as a $\kk \db \otimes \kk \vo$-module, in the obvious way.
\end{rem}

The functor $\hom_{\f (\vo)} (-, \ot^\bullet ) $ restricts to an exact functor:
$$
\hom_{\falg (\vo)} (-, \ot^\bullet ) \colon 
\falg (\vo) \op 
\rightarrow 
\kk \db \dash \modules
$$
that factors across the opposite of the quotient category $ \falg (\vo) / \falg_\tors (\vo)$. 

Analogous statements hold for $\f (\vsp)$ using $\spt^\bullet$.

\begin{thm}
\label{thm:stabilization}
\cite[Theorem 4.2.6]{MR3376738}
\begin{enumerate}
\item 
The functor 
$$
\hom_{\falg (\vo)} (-, \ot^\bullet ) \colon 
\big( \falg (\vo)  / \falg_\tors (\vo) \big)\op 
\rightarrow 
\kk \db \dash \modules
$$
 is exact 
and this restricts to an equivalence of categories between the respective full subcategories of finite length objects.
\item 
The functor 
$$
\hom_{\falg (\vsp)} (-, \spt^\bullet ) \colon 
\big( \falg (\vsp)  / \falg_\tors (\vsp) \big)\op 
\rightarrow 
(\kk \db)_{(-;+)} \dash \modules
$$
is exact and this restricts to an equivalence of categories between the respective full subcategories of finite length objects.
\end{enumerate}
\end{thm}

\begin{rem}
\label{rem:approximating_stabilization}
In this remark, we focus upon the orthogonal case;  there are analogous statements for the symplectic case.
\begin{enumerate}
\item 
For $F$ an object of $\falg (\vo)$, we consider the $\kk \db$-module $\hom_{\f (\vo)} (F, \ot^\bullet)$ as a substitute for the stabilization of $F$ (considered as an object of $\falg (\vo)/ \falg_\tors (\vo)$ or, equivalently, $\rep (\ogp)$). If $F$ has finite length, then this is equivalent to the stabilization of $F$, by Theorem \ref{thm:stabilization}.
\item 
One reason that the restriction to finite length objects arises is due to the fact that vector space duality does not induce an equivalence between the category of $\rat$-vector spaces and its opposite. (Recall that, to resolve this issue, one should equip the dual space with its natural profinite topology and use the continuous dual to induce an equivalence.)
\item 
The finite length hypothesis can be relaxed by using the natural filtration of an object $F$ of $\falg (\vo)$ given by Proposition \ref{prop:decreasing_filtration_fvo} below. In this context (using the notation introduced in that Proposition), one requires that, for all $n \in \nat$, the subquotient $F^{\geq n}/ F^{\geq n+1}$ have finite length in the quotient category $\falg (\vo) / \falg_\tors (\vo)$.  There is a caveat: in general, the associated $\kk \db$-module allows one to recover the tower of quotients $(F/ F^{\geq n} \mid n \in\nat)$, but not $F$ itself. 
\end{enumerate}
\end{rem}

Motivated by the above results, we introduce the following natural filtration of objects of $\falg (\vo)$ (respectively $\falg (\vsp)$):

\begin{defn}
\label{defn:falg_geqn}
For $n \in \nat$, let  
\begin{enumerate}
\item 
$\falg_{\geq n} (\vo)$ be the full subcategory of $\falg (\vo)$ of objects $F$ such that $\hom_{\f (\vo)} (F, \ot^d) =0$ for all $d < n$;
\item 
$\falg_{\geq n} (\vsp)$  be the full subcategory of $\falg (\vsp)$ of objects $F$ such that $\hom_{\f (\vsp)} (F, \spt^d) =0$ for all $d < n$.
\end{enumerate}
\end{defn}

\begin{prop}
\label{prop:falg_geqn}
For $n \in \nat$, the subcategory $\falg_{\geq n} (\vo) \subset \falg (\vo)$ (respectively $\falg_{\geq n} (\vsp) \subset \falg (\vsp)$) is a localizing Serre subcategory. 
\end{prop}

\begin{proof}
The fact that these subcategories are closed under quotients, extensions, and arbitrary coproducts is immediate. Closure under subobjects is a consequence of the injectivity statement of Theorem \ref{thm:injectives_falgvo_falgvsp}.
\end{proof}

We state the following for the orthogonal case; there is an analogous result for the symplectic case.

\begin{prop}
\label{prop:decreasing_filtration_fvo}
For $F$ in $\falg (\vo)$, there is a natural decreasing filtration 
$$
\ldots \subseteq F^{\geq n+1} \subseteq F^{\geq n} \subseteq \ldots \subseteq F^{\geq 0} = F
$$
by subobjects such that 
\begin{enumerate}
\item 
$F^{\geq n}$ is the largest subobject of $F$ in $\falg_{\geq n} (\vo)$; 
\item 
$\bigcap_{n \in \nat} F^{\geq n}  = F_\tors$, where $F_\tors$ is the largest torsion subobject of $F$;
\item 
the functor $F^{\geq n}/ F^{\geq n+1}$ is torsion-free and embeds in a coproduct of copies of $\ot^{[n]}$. 
\end{enumerate}
\end{prop}

\begin{proof}
The functor $F \mapsto F^{\geq n}$ corresponds to the right adjoint to the inclusion $\falg_{\geq n} (\vo) \subseteq \falg (\vo)$, the existence of which is readily checked.  In particular, essentially by definition,  $F^{\geq n}$ is the largest subobject of $F$ in $\falg_{\geq n} (\vo)$. 

It is clear that the functors $F^{\geq n}$, for $n \in \nat$, yield a natural filtration of $F$ as stated. Moreover, since $\ot^k$ is torsion-free for each $k \in \nat$, one has that $F_\tors$ is contained in $F^{\geq n}$ for each $n \in \nat$. Now consider the inclusion $F_\tors \subseteq \bigcap_{n \in \nat} F^{\geq n} $. To show that this is an equality, we require to prove that $\bigcap_{n \in \nat} F^{\geq n} $ is isomorphic to $0$ in the quotient category $\falg (\vo) /\falg_{\tors} (\vo)$. Since the $\ot^k$, for $k\in \nat$, form a set of injective cogenerators of the quotient category by Corollary \ref{cor:injective_cogenerators}, this follows from  the definition of the $F^{\geq n}$. 

For the final point, similarly to the above argument, one has that $F^{\geq n}/ F^{\geq n+1}$ is torsion-free. Now, $F^{\geq n+1}$ is the subobject of $F^{\geq n}$ that is constructed by taking the intersection of the kernels of all maps from $F^{\geq n}$ to $\ot^{n}$. Since the category $\falg (\vo)$ is locally noetherian, this implies that there is an exact sequence of the form 
$$
0
\rightarrow 
F^{\geq n+1} 
\rightarrow
F^{\geq n}
\rightarrow 
\bigoplus \ot^n.
$$  
This gives an embedding of $F^{\geq n}/ F^{\geq n+1}$ in $\bigoplus \ot^n$.

Now, by Definition \ref{defn:harmonic},  there is an exact sequence of the form
$$
0
\rightarrow 
\ot^{[n]} 
\rightarrow 
\ot^n 
\rightarrow 
\bigoplus \ot^{n-2},
$$
where the direct sum is finite and $\ot^{n-2}$ is understood to be zero for $n<2$. By construction $\hom_{\falg(\vo)} (F^{\geq n} , \ot^{n-2})$ is zero. Using this, one shows that the inclusion of $F^{\geq n}/ F^{\geq n+1}$ in $\bigoplus \ot^n$ factors across $\bigoplus \ot^{[n]} \subseteq \bigoplus \ot^{n}$.
\end{proof}

\begin{cor}
\label{cor:properties_filtration^geq}
For $F$ in $\falg (\vo)$ and $n \in \nat$, the following properties hold:
\begin{enumerate}
\item 
$\hom_{\falg (\vo)} (F/ F^{\geq n}, \ot^k)$ is zero for $k \geq n$; 
\item
the quotient map $F \twoheadrightarrow F/ F^{\geq n}$ induces an isomorphism 
$$
\hom_{\falg (\vo)} (F/ F^{\geq n}, \ot^k) \stackrel{\cong}{\rightarrow} \hom_{\falg (\vo)} (F, \ot^k) 
$$
for $k <n$.
\end{enumerate}
\end{cor}

\begin{proof}
The functor $F/ F^{\geq n}$ admits a finite filtration with filtration quotients that are subobjects of coproducts of $\ot^{[l]}$, for $0 \leq l < n$. Hence, to prove the first statement, it suffices to show that $\hom_{\falg (\vo)} (\ot^{[l]}, \ot^k)$ is zero for $0 \leq l <n $ and $k\geq n$. This is well-known; an argument is given for completeness.
 
Since $\ot^{[l]}$ is a subobject of $\ot^l$, by the injectivity of $\ot^k$ in $\falg (\vo)$, the inclusion induces a surjection 
$$
\hom_{\falg (\vo)} (\ot^{l}, \ot^k ) \twoheadrightarrow \hom_{\falg (\vo)} (\ot^{[l]}, \ot^k).
$$
To conclude, one uses that the domain is isomorphic to $\kk \db (\mathbf{l}, \mathbf{k})$, by Theorem \ref{thm:vker_fully-faithful}; this vanishes under the hypotheses on $l $ and $k$.

The second statement is proved by considering the short exact sequence obtained by applying $\hom_{\falg (\vo)} (-, \ot^k )$ (for $k<n$) to the short exact sequence 
$$
0
\rightarrow 
F^{\geq n}
\rightarrow 
F 
\rightarrow 
F/ F^{\geq n}
\rightarrow 
0.
$$
By construction, $\hom_{\falg (\vo)} (F^{\geq n}, \ot^k) =0$, whence the result. 
\end{proof}

We also have the following result (stated for the symplectic case, there is a counterpart for the orthogonal case): 

\begin{thm}
\label{thm:stabilization_trivial_action}
For $F \in \falg (\vsp)$, the following conditions are equivalent: 
\begin{enumerate}
\item 
The stabilization $\stabvsp F$ in $\rep (\spgp)$ has trivial $\spgp_\infty$-action. 
\item 
The subfunctor $F^{\geq 1}$ is torsion. 
\item 
$\hom_{\falg (\vsp) } (F, \spt^k)$ is zero for all $k>0$.
\end{enumerate}
\end{thm}

\begin{proof}
The equivalence of the last two conditions follows from Proposition \ref{prop:decreasing_filtration_fvo}. 
Moreover, this result implies that, if these conditions hold, then $F/ F_\tors$ embeds in a constant functor. 
This clearly implies that the stabilization of $F$ has constant $\spgp_\infty$-action.

For the converse, suppose that $F^{\geq 1}$ is not torsion. This implies that there exists a positive integer $d$ such that $F^{\geq d}/ F^{\geq d+1}$ is non-zero (and we may take $d$ minimal such that this holds). This embeds in a coproduct of copies of $\ot^{[d]}$, by Proposition \ref{prop:decreasing_filtration_fvo}. Hence, (up to torsion), there exists $\lambda \vdash d$ such that $\spt^{[\lambda]}$ occurs as a composition factor of $ F^{\geq d}/ F^{\geq d+1}$. Since $\spgp_\infty$ acts non-trivially on $\stabvsp (\spt^{[\lambda]})$, one deduces that $\stabvsp (F)$ has non-trivial $\spgp_\infty$-action, as required.  
\end{proof}

\subsection{Precomposing with generalized Schur functors}

It is of interest to precompose the functor $$\hom_{\falg (\vo)} (-, \ot^\bullet)$$ with the generalized Schur functor $\ot^\bullet \otimes_{\kk \ub} -$ of Section \ref{subsect:gen_Schur}. Here, due to the occurrence of $\ot^\bullet$ twice in the resulting expression, we denote one of the wild-cards $\bullet$ by $\ast$. Since $\falg (\vo)$ is a full subcategory of $\f (\vo)$, we can work in $\f (\vo)$. The symplectic case is treated similarly.

\begin{prop}
\label{prop:natural_iso_hom_vker_vker}
\ 
\begin{enumerate}
\item 
For $M$ a $\kk \ub$-module, there is a natural isomorphism of $\kk \db$-modules:
\[
\hom _{\f (\vo)} ( \ot^\bullet \otimes_{\kk \ub} M, \ot^\ast) 
\cong 
\hom_\kk ( \kk \ub^\sharp \otimes_{\kk \ub} M, \kk),
\]
where the $\kk \db$-module structure on the left hand side arises from the naturality with respect to $\ast$; on the right hand side, this is given by the $\kk \ub$-module structure of $\kk \ub^\sharp \otimes_{\kk \ub} M$.
\item
For $N$ a $(\kk \ub)_{(-;+)}$-module, there is a natural isomorphism of $(\kk \db)_{(-;+)}$-modules:
\[
\hom_{\f (\vsp)} (\spt^\bullet  \otimes_{(\kk \ub)_{(-;+)}} N, \spt^\ast )  
\cong 
\hom_\kk ((\kk \ub)_{(-;+)} ^\sharp \otimes _{(\kk \ub)_{(-;+)}} N, \kk),
\]
where the $(\kk \db)_{(-;+)}$-module structure on the left hand side arises from the naturality with respect to $\ast$; on the right hand side, this is given by the $(\kk \ub)_{(-;+)}$-module structure of $(\kk \ub)_{(-;+)} ^\sharp \otimes _{(\kk \ub)_{(-;+)}} N$.
\end{enumerate}
\end{prop}

\begin{proof}
We prove the first case; the second is proved similarly.

One has the natural adjunction isomorphism
\[
\hom _{\f (\vo)} ( \ot^\bullet \otimes_{\kk \ub} M, \ot^\ast) 
\cong 
\hom_{\kk \ub} (M, \hom_{\f (\vo)} (\ot^\bullet, \ot^\ast))
\cong 
\hom_{\kk \ub} (M, \kk \ub (\ast, \bullet) ),
\]
where the second isomorphism is given by Theorem \ref{thm:vker_fully-faithful}. Explicitly, considering $X$ as an object of $\kk \ub$, one has 
\[
\hom_{\kk \ub} (M, \kk \ub )(X) 
= 
\hom_{\kk \ub} (M, \kk \ub (X, -))
.
\]
Since $\kk\ub(X,-)$ takes finite dimensional values, by the universal property of the tensor product the latter is isomorphic to $
\hom_\kk ( \kk\ub(X,-)^\sharp \otimes _{\kk \ub} M, \kk)$, as required.
\end{proof}

Proposition \ref{prop:natural_iso_hom_vker_vker} has the following consequence (stated for the symplectic case): 

\begin{cor}
\label{cor:property_vker_sp_otimes_torsion}
For $N$ a torsion $(\kk \ub)_{(-;+)}$-module, one has 
\[
\hom_{\f (\vsp)} (\spt^\bullet \otimes_{(\kk \ub)_{(-;+)}} N, \spt^\ast) 
=0.
\]

Equivalently, $\spt^\bullet \otimes_{(\kk \ub)_{(-;+)}} N$ is torsion in $\falg (\vsp)$.
\end{cor}

\begin{proof}
Proposition \ref{prop:natural_iso_hom_vker_vker} gives the identification 
$$\hom_{\f (\vsp)} (\spt^\bullet \otimes_{(\kk \ub)_{(-;+)}} N, \spt^\ast) 
\cong 
\hom_\kk ((\kk \ub)_{(-;+)} ^\sharp \otimes _{(\kk \ub)_{(-;+)}} N, \kk).
$$
 The first statement follows from Proposition \ref{prop:vanish-on-torsion}. The second equivalent statement then follows from Corollary \ref{cor:injective_cogenerators}.
\end{proof}

Proposition \ref{prop:natural_iso_hom_vker_vker} will be applied to objects that are induced up from $\kk \fb$-modules. For convenience we restate the conclusion as a Corollary: 

\begin{cor}
\label{cor:natural_iso_hom_vker_vker_induced_case}
For $L$ a $\kk \fb$-module, there are natural isomorphisms:
\begin{eqnarray*}
\hom _{\f (\vo)} ( \ot^\bullet  \otimes_{\kk \fb} L, \ot^\ast) 
&\cong &
\hom_\kk ( \kk \ub^\sharp \otimes_{\kk \fb} L, \kk)
\\
\hom_{\f (\vsp)} (\spt^\bullet \otimes_{\kk \fb} L, \spt^\ast) 
& \cong & 
\hom_\kk ((\kk \ub)_{(-;+)} ^\sharp \otimes _{\kk \fb} L, \kk).
\end{eqnarray*}
\end{cor}

\subsection{Relating to symplectic invariants}

Weyl's results on (stable) symplectic invariants have been subsumed in Theorem \ref{thm:vker_fully-faithful}. 
To make this explicit we recall  the following well-known result: 

\begin{prop}
\label{prop:symplectic_invariants}
For $n \in \nat$ and $(V, \omega)$ a symplectic vector space with $\dim V \geq n$, there is an isomorphism of $\kk \sym_n$-modules 
\begin{eqnarray}
\label{eqn:V_otimes_n_invariants}
(V^{\otimes n} )^{\mathrm{Sp}(V)}
\cong 
(\kk\db)_{(-;+)} (\n, \mathbf{0})^\sharp.
\end{eqnarray}
In particular, this is zero if $n$ is odd.
\end{prop}

\begin{rem}
Over a field of characteristic zero, finite-dimensional $\rat \sym_n$-modules are self-dual. The result has been stated in the above form, since it makes the relationship with the following isomorphism (given by Theorem \ref{thm:vker_fully-faithful}) more transparent: 
\[
\hom_{\f (\vsp)} (\spt^n, \spt^0) \cong (\kk \db) _{(-;+)} (\n, \mathbf{0}).
\]
\end{rem}

\section{Koszul complexes and generalized Schur functors}
\label{sect:forms_koszul}

The purpose of this section is to apply the theory reviewed in Section \ref{sect:forms} to the Koszul complexes of Section \ref{sect:koszul}. For instance, in the symplectic case, we first replace a complex of $(\kk \ub)_{(-;+)}$-modules by a complex of algebraic functors on $\vsp$ using the generalized Schur functor - and describe the structure of this complex. Then we consider the stabilization of the homology of this complex, exhibiting a complex of $(\kk \ub)_{(-;+)}$-modules such that the dual of its homology calculates our approximation to the stabilization.

Throughout this section, $\kk$ is a field of characteristic zero.

\subsection{Applying generalized Schur functors to Koszul complexes}
\label{subsect:compose_Koszul_functors}

Recall from Section \ref{sect:koszul} that
\begin{enumerate}
\item 
 $\mathscr{K}_+ \otimes _{(\kk \db)_{(+;-)}} -$ is a functor from $(\kk\db)_{(+;-)}$-modules to complexes of $\kk \ub$-modules; 
 \item 
 $\mathscr{K}_- \otimes _{(\kk \db)_{(-;-)}}-$ is a functor from $(\kk\db)_{(-; -)}$-modules to complexes of $(\kk \ub)_{(-;+)}$-modules.
\end{enumerate}
We may thus form the following composite functors by using the generalized Schur functors of Lemma  \ref{lem:tensor_with_vker}:
\begin{enumerate}
\item 
$\ot^\bullet  \otimes_{\kk \ub} \mathscr{K}_+ \otimes _{(\kk \db)_{(+;-)}} -$ from $(\kk\db)_{(+;-)}$-modules to complexes in $\falg (\vo)$; 
\item 
$\spt^\bullet  \otimes_{(\kk\ub)_{(-;+)} }\mathscr{K}_- \otimes _{(\kk \db)_{(-;-)}}-$ from $(\kk\db)_{(-; -)}$-modules to complexes in $\falg (\vsp)$.
\end{enumerate}
These are complexes of algebraic functors, in the sense of Definition \ref{defn:falg}.

\begin{rem}
\label{rem:grading_vker_otimes_Koszul}
One has to pay attention to the grading in forming these complexes. For example, in the second case, one can use the length grading of $(\kk\ub)_{(-;+)}$ to grade $\mathscr{K}_- \otimes _{(\kk \db)_{(-;-)}}M$, for $M$ a $(\kk \db)_{(-;-)}$-module; to apply $\spt^\bullet \otimes_{(\kk\ub)_{(-;+)} }-$, we require a compatible grading on $\spt^\bullet$. The following is a natural choice: 
\[
\mathrm{degree} (V^{\otimes n}):= \Bigl\lfloor 
\frac{n}{2}
\Bigr\rfloor.
\]
In some cases the $(\kk \db)_{(-;-)}$-module $M$ comes equipped with a grading compatible with the module structure; it may then be more natural to use this to define the grading of the complex.

Below, we will leave the grading unspecified; in applications, the grading will be made explicit.
\end{rem}

For concision, we introduce the following notation:

\begin{nota}
\label{nota:complexes}
\ 
\begin{enumerate}
\item 
For  $N$ a $(\kk\db)_{(+;-)}$-module, denote by $\ocx (N)$ the complex $\ot^\bullet \otimes_{\kk \ub} \mathscr{K}_+ \otimes _{(\kk \db)_{(+;-)}} N$ in $\falg (\vo)$.
\item 
For $M$ a $(\kk\db)_{(-;-)}$-module, denoted by $\spcx (M)$ the complex $\spt^\bullet  \otimes_{(\kk\ub)_{(-;+)} }\mathscr{K}_- \otimes _{(\kk \db)_{(-;-)}}M$ in $\falg (\vsp)$.
\end{enumerate}
\end{nota}

The underlying objects (without their differentials) of these complexes are easily identified:

\begin{lem}
\label{lem:underlying_complexes}
\ 
\begin{enumerate}
\item 
For $N$ a $(\kk\db)_{(+;-)}$-module, the complex $\ocx(N)$ evaluated on an orthogonal vector space $(V, \sbf )$ has underlying object isomorphic to $N(V)$, given by the (classical) Schur functor of $N$.
\item 
For $M$ a $(\kk\db)_{(-;-)}$-module, the complex $\spcx(M)$ evaluated on a symplectic vector space $(V,  \omega)$ has underlying object isomorphic to $M (V)$, given by the (classical) Schur functor of $M$.
\end{enumerate}
In  both cases, the underlying object is independent of the form and depends only upon the underlying $\kk \fb$-module.
\end{lem}

\begin{proof}
This follows directly from the description of the underlying bimodule of $\mathscr{K}_+$ (respectively $\mathscr{K}_-$), together with the definition of the classical Schur functor, as recalled in Section \ref{subsect:gen_Schur}.
\end{proof}

It remains to identify the differentials of these complexes. For this, we first recall from Proposition \ref{prop:kub_twist_homogeneous_quadratic} that the morphisms of $(\kk \ub)_{(\pm; \mp)}$ are generated (over $\kk \fb$) by the $\kk \fb$-bimodule of degree one morphisms. For $(\kk\ub)_{(+;-)}$ and $(\kk \ub)_{(-;-)}$ the degree one morphisms are identified by the following special case of Lemma \ref{lem:bimodule_kub_twisted}:

\begin{lem}
\label{lem:kub_bimodule_special_case}
For $n \in \nat$, there are isomorphisms of $\kk \sym_n \otimes \kk \sym_{n+2}\op$-modules:
\begin{eqnarray*}
(\kk \ub)_{(+; -)}(\n , \mathbf{n+2}) & \cong & \kk \sym_{n+2} \otimes_{\kk \sym_2} \triv_2 \\
(\kk \ub)_{(-; -)}(\n , \mathbf{n+2}) & \cong & \kk \sym_{n+2} \otimes_{\kk \sym_2} \sgn_2,
\end{eqnarray*}
where the bimodule structures are as in Lemma \ref{lem:bimodule_kub_twisted}.
\end{lem}

Thus, if $M$ is a $(\kk\db)_{(-;-)}$-module, for each $n \in \nat$, the action of the degree one morphisms corresponds to a morphism of $\kk \sym_{n+2}$-modules
\begin{eqnarray}
\label{eqn:deg1_(--)module}
M (\mathbf{n+2})
\rightarrow 
(M(\n) \boxtimes \sgn_2) \uparrow^{\sym_{n+2}} _{\sym_n \times \sym_2}.
\end{eqnarray}
Applying the functor $V^{\otimes n+2} \otimes_{\kk \sym_{n+2}}-$, for $V$ a finite-dimensional $\kk$-vector space,  this yields 
\begin{eqnarray}
\label{eqn:deg1_(--)module_schur}
M_{n+2} (V) \rightarrow M_n (V) \otimes \Lambda^2 (V),
\end{eqnarray}
where $\Lambda^2 (-)$ is the second exterior power functor. (Here we are using the notation introduced in  Notation \ref{nota:schur_components}.)

Similarly, if $N$ is a $(\kk\db)_{(+;-)}$-module, for each $n \in \nat$, the action of the  degree one morphisms corresponds to a  morphism of $\kk \sym_{n+2}$-modules
\begin{eqnarray}
\label{eqn:deg1_(+-)module}
N (\mathbf{n+2})
\rightarrow 
(N(\n) \boxtimes \triv_2) \uparrow^{\sym_{n+2}} _{\sym_n \times \sym_2}.
\end{eqnarray}
Applying the functor $V^{\otimes n+2} \otimes_{\kk \sym_{n+2}}-$, this yields 
\begin{eqnarray}
\label{eqn:deg1_(+-)module_schur}
N_{n+2} (V) \rightarrow N_n (V) \otimes S^2 (V),
\end{eqnarray}
where $S^2(-)$ is the second symmetric power functor.

\begin{rem}
Using the Day convolution product for $\kk \fb$-modules (recalled in Section \ref{subsect:cyclic_brauer_recollections} below), the morphisms (\ref{eqn:deg1_(--)module}) and (\ref{eqn:deg1_(+-)module}) can be written in the following global forms respectively:
\begin{eqnarray*}
M & \rightarrow & M \odot \sgn_2 
\\
N & \rightarrow & N \odot \triv_2.
\end{eqnarray*} 
\end{rem}

\begin{defn}
\label{defn:differentials}
\ 
\begin{enumerate}
\item 
For $M$  a $(\kk\db)_{(-;-)}$-module, $n \in \nat$, and $(V, \omega)$ a symplectic vector space, let 
\begin{eqnarray*}
\label{eqn:domega}
\domega_{n}^M(V,\omega) \colon M_{n+2} (V) \rightarrow M_n(V)
\end{eqnarray*}
be the natural morphism obtained by composing (\ref{eqn:deg1_(--)module_schur}) with the morphism induced by the form $\omega : \Lambda^2 (V) \rightarrow \kk$. 
\item 
For $N$  a $(\kk\db)_{(+;-)}$-module, $n \in \nat$, and $(V, \sbf )$ an orthogonal vector space, let 
\begin{eqnarray*}
\label{eqn:domega_orthog}
\domega_{n}^N(V,\sbf ) \colon N_{n+2} (V) \rightarrow N_n(V)
\end{eqnarray*}
be the natural morphism obtained by composing (\ref{eqn:deg1_(+-)module_schur}) with the morphism induced by the form $\sbf  : S^2 (V) \rightarrow \kk$.
\end{enumerate}
\end{defn}

This allows us to identify the complexes:

\begin{prop}
\label{prop:V_omega_complex}
\ 
\begin{enumerate}
\item 
For $M$ a $(\kk\db)_{(-;-)}$-module, the complex 
$
\spcx( M)
$ 
evaluated on $(V, \omega)$ has underlying object $M(V)$ and differential  given by the  natural morphisms $\domega_n^M(V, \omega)$.
\item 
For $N$  a $(\kk\db)_{(+;-)}$-module, the complex 
$
\ocx (N) 
$ 
evaluated on $(V, \sbf )$ has underlying object $N (V)$ and differential given by the natural morphisms $\domega_n^N(V, \sbf )$.
\end{enumerate}
\end{prop}

\begin{proof}
The underlying objects were identified in Lemma \ref{lem:underlying_complexes}. The identification of the differential follows from the definition of the differential of the respective complexes $\mathscr{K}_-$ and $\mathscr{K}_+$.
\end{proof}

\subsection{Homology and stabilizing}

Recall from Section \ref{sect:forms} that the stabilization of an algebraic functor on $\vsp$ is given by applying the functor $\stabvsp : \falg (\vsp) \rightarrow \rep (\spgp)$. Likewise, for an algebraic functor on $\vo$, it is given by applying $\stabvo : \falg (\vo) \rightarrow \rep (\ogp)$.

The following is immediate:

\begin{lem}
\label{lem:stabilization_homology}
\ 
\begin{enumerate}
\item 
For $M$ a $(\kk \db)_{(-;-)}$-module, the homology $H_* \spcx (M)$ is a graded object in $\falg (\vsp)$. 
 Hence the stabilization $\stabvsp (H_* \spcx (M))$ is a graded object in $\falg (\vsp)/\falg_\tors (\vsp) \simeq \rep (\spgp)$.
\item 
For $N$ a $\kk \db$-module, the homology $H_* \ocx (N)$ is a graded object in $\falg (\vo)$. Hence the stabilization $\stabvo (H_* \ocx (N))$ is a graded object in $\falg (\vo)/ \falg_\tors (\vo) \simeq \rep (\ogp)$.
\end{enumerate}
\end{lem}

We are interested both in the functors $H_* \spcx (M)$ (respectively $H_* \ocx (N)$) and in their respective stabilizations. Now, as in Section \ref{subsect:analyse_algebraic}, we can approximate the stabilizations (see Remark \ref{rem:approximating_stabilization}) by applying the respective functors 
\begin{eqnarray*}
\hom_{\falg (\vsp)} (-, \spt^\ast) &:& \falg (\vsp) \rightarrow (\kk \db)_{(-;+)}\dash\modules \\
\hom_{\falg (\vo)} (-, \ot ^\ast ) &:& \falg (\vo) \rightarrow \kk \db \dash\modules.
\end{eqnarray*}

These functors are exact, by Theorem \ref{thm:injectives_falgvo_falgvsp}. This implies the following:

\begin{lem}
\label{lem:weak_stabilization_homology}
\ 
\begin{enumerate}
\item 
For $M$ a $(\kk \db)_{(-;-)}$-module, the graded $(\kk\db)_{(-;+)}$-module   $\hom_{\falg (\vsp)} ( H_* \spcx (M), \spt^\ast)$ is naturally isomorphic to the homology of the complex of $(\kk\db)_{(-;+)}$-modules
$$
\hom_{\falg (\vsp)} ( \spcx (M), \spt^\ast).
$$ 
\item 
For $N$ a $(\kk \db)_{(+;-)}$-module, the graded $\kk\db$-module   $\hom_{\falg (\vo)} ( H_* \ocx (N), \ot^\ast)$ is naturally isomorphic to the homology of the complex of $\kk\db$-modules
$$
\hom_{\falg (\vo)} ( \ocx (N), \ot^\ast).
$$ 
\end{enumerate}
\end{lem}

We can identify these complexes by  Proposition \ref{prop:natural_iso_hom_vker_vker}:

\begin{prop}
\label{prop:identify_complexes}
\ 
\begin{enumerate}
\item 
For $M$ a $(\kk \db)_{(-;-)}$-module, there is a natural  isomorphism of complexes of $(\kk\db) _{(-;+)}$-modules 
\begin{eqnarray}
\label{eqn:2nd_Kz_M_dual}
\hom_{\falg (\vsp)} ( \spcx (M), \spt^\ast) &\cong & \Big( 
(\kk\ub)_{(-;+)}^\sharp  \otimes_{(\kk\ub)_{(-;+)}} \mathscr{K}_- \otimes_{(\kk\db)_{(-;-)}}M \Big) ^\sharp. 
\end{eqnarray}
\item 
For $N$ a $(\kk \db)_{(+;-)}$-module,  there is a natural  isomorphism of complexes of $\kk\db$-modules 
\begin{eqnarray}
\label{eqn:2nd_Kz_N_dual}
\hom_{\falg (\vo)} ( \ocx (N), \ot^\ast)
&\cong &
\Big( \kk\ub^\sharp  \otimes_{\kk\ub} \mathscr{K}_+ \otimes_{(\kk\db)_{(+;-)}}N \Big)^\sharp .
\end{eqnarray}
\end{enumerate}
\end{prop}

Proposition \ref{prop:identify_complexes} has the following Corollary:

\begin{cor}
\label{cor:indentify_homology}
\ 
\begin{enumerate}
\item 
For $M$ a $(\kk \db)_{(-;-)}$-module, there is a natural  isomorphism of graded  $(\kk\db) _{(-;+)}$-modules 
\begin{eqnarray*}
\hom_{\falg (\vsp)} ( H_* \spcx (M), \spt^\ast) &\cong & 
H_* \big((\kk\ub)_{(-;+)}^\sharp  \otimes_{(\kk\ub)_{(-;+)}} \mathscr{K}_- \otimes_{(\kk\db)_{(-;-)}}M\big)  ^\sharp. 
\end{eqnarray*}
\item 
For $N$ a $(\kk \db)_{(+;-)}$-module,  there is a natural  isomorphism of graded $\kk\db$-modules 
\begin{eqnarray*}
\hom_{\falg (\vo)} ( H_* \ocx (N), \ot^\ast)
&\cong &
H_* \big( \kk\ub^\sharp  \otimes_{\kk\ub} \mathscr{K}_+ \otimes_{(\kk\db)_{(+;-)}}N \big)^\sharp
\end{eqnarray*}
\end{enumerate}
\end{cor}

\begin{rem}
This result isolates the vector space duality in this approach to approximating $\stabvsp H_* \spcx (M)$ and $\stabvo H_* \ocx (N)$. This shows that the graded $(\kk \ub)_{(-;+)}$-module
$$
H_* \big(\kk\ub)_{(-;+)}^\sharp  \otimes_{(\kk\ub)_{(-;+)}} \mathscr{K}_- \otimes_{(\kk\db)_{(-;-)}}M\big) 
$$ 
is very close to determining $\stabvsp H_* \spcx (M)$.

Likewise, the graded $\kk \ub$-module
$$
H_* \big( \kk\ub^\sharp  \otimes_{\kk\ub}  \mathscr{K}_+ \otimes_{(\kk\db)_{(+;-)}}N \big)
$$
is very close to determining $\stabvo H_* \ocx (N)$.
\end{rem}

\section{Cyclic operads and modules over twisted Brauer categories}
\label{sect:cyclic_brauer}

The purpose of this section is to show how to associate  a  $(\kk \db)_{(-;-)}$-module to a non-unital cyclic operad $\cpd$; this is a signed version of a more direct construction which yields a $\kk \db$-module. This result, stated as Theorem \ref{thm:cpd_db_kdbm}, is related to Stoll's characterization of modular operads in terms of the Brauer properads \cite{MR4541945}.

As will be seen in Section \ref{sect:kontsevich_lie}, one significance of the $(\kk \db)_{(-;-)}$-module structure is that the associated generalized Schur functor construction yields the Lie algebra associated to a cyclic operad $\cpd$ and a symplectic vector space $(V, \omega)$, as constructed by Conant and Vogtmann \cite{MR2026331} generalizing a construction of Kontsevich \cite{MR1247289,MR1341841}. 

In this section, we take $\kk$ to be a field; for subsequent applications it will have characteristic zero.

\subsection{Recollections}
\label{subsect:cyclic_brauer_recollections}

Recall that the category of $\kk \fb$-modules is equipped with the Day convolution product, denoted $\odot$ here. For $F$, $G$ two $\kk \fb$-modules, this is defined explicitly by 
\[
(F\odot G) (X) := \bigoplus_{X = S_1 \amalg S_2} F(S_1) \otimes G(S_1),
\]
where the sum is over ordered decompositions of $X$ into two subsets. 
This defines a symmetric monoidal structure on $\kk \fb$-modules with unit $\kk_\mathbf{0}$, the $\kk\fb$-module supported on $\emptyset$ with value $\kk$. Thus, for a $\kk \fb$-module $F$ and $n \in \nat$, one can form the iterated Day convolution product $F^{\odot n}$; the symmetric group $\sym_n$ acts by permuting the $\odot$ factors. By convention $F^{\odot 0}$ is $\kk_\mathbf{0}$.

\begin{nota}
\label{nota:sfb_lfb}
Write $\sfb^*$ (respectively $\lfb^*$) for the symmetric (resp. exterior) algebra functor in the category of $\kk \fb$-modules, so that, for $n \in \nat$ and an $\kk\fb$-module $F$,
\begin{eqnarray*}
\sfb ^n (F) &=& F^{\odot n}/ \sym_n\\
\lfb ^n (F) &=& (\sgn_n \otimes F^{\odot n}) / \sym_n,
\end{eqnarray*}
where $\sym_n$ acts diagonally on $\sgn_n \otimes F^{\odot n}$. 
\end{nota}

\begin{rem}
\label{rem:schur_sfb_lfb}
The classical Schur functor construction (as reviewed at the beginning of Section \ref{subsect:gen_Schur}) is symmetric monoidal from $\kk \fb$-modules (equipped with the Day convolution product) to functors on $\kk$-vector spaces (equipped with the pointwise tensor product). Namely, for $\kk\fb$-modules $M$ and $N$, and $V $ a $\kk$-vector space, there is a natural isomorphism 
$$
(M\odot N) (V) \cong M(V) \otimes N(V)
$$
that is compatible with the symmetry and the associativity constraints. 

This implies that, for $F$ a $\kk\fb$-module and $n\in \nat$, there are natural isomorphisms of functors (with respect to the vector space $V$):
\begin{eqnarray*}
\sfb ^n (F) (V) & \cong & S^n (F (V)) \\
\lfb ^n (F)(V) & \cong & \Lambda^n (F(V)),
\end{eqnarray*}
where, on the right, $S^n$ and $\Lambda^n$ are the usual symmetric and exterior power functors respectively. 
\end{rem}

\begin{lem}
\label{lem:sfb_lfb}
For  $0<n \in \nat$ and $F$ a $\kk \fb$-module such that $F (\mathbf{0})=0$, there are isomorphisms of $\kk$-vector spaces 
\[
\sfb ^n (F) (X) 
\cong 
\lfb^n (F) (X)
\cong 
\bigoplus_{\substack{(U_i \mid i\in \n) \\ \amalg_i U_i =X}} 
\bigotimes_i F(U_i), 
\]
where the sum is over unordered decompositions of $X$ into $n$ non-empty subsets.
\end{lem}

\begin{proof}
This follows  from the explicit description of the Day convolution product $\odot$, together with the fact that, since $F(\mathbf{0})=0$, the $\sym_n$-action on $F^{\odot n} (X)$ is free.
\end{proof}

\begin{exam}
\label{exam:sfb_lfb_differ}
We give an example exhibiting that, for $n>1$, the $\kk\fb$-modules $\sfb^n (F)$ and $\lfb^n (F)$ differ if $F\neq 0$.
 Consider the $\kk \fb$-module $I$ that is supported on $\mathbf{1}$ with value $\kk$. Then, as $\kk$-vector spaces, 
\[
\sfb^n  (I)(\mathbf{t})  \cong \lfb^n (I) (\mathbf{t}) \cong
\left\{
 \begin{array}{ll} 
 \kk & n=t \\
 0 &\mbox{otherwise.}
 \end{array}
 \right.
\]
For $t=n$,  if $n>1$ these $\kk \sym_t$-modules are not isomorphic, since $\sfb^t  (I)(\mathbf{t}) \cong \triv_t$, whereas $\lfb^t (I) (\mathbf{t}) \cong \sgn_t$.
\end{exam}

For later use, we record the following:

\begin{lem}
\label{lem:split_F_odot_F}
Suppose that $\kk$ is a field of characteristic $\neq 2$. Then, for $F$ a $\kk \fb$-module, there is a natural splitting
\[
F\odot F \cong \sfb^2 (F) \oplus \lfb^2 (F).
\]
\end{lem}

\begin{proof}
This follows from the fact that, if $\kk$ has characteristic other than two,  $\kk \sym_2$  considered as a $\kk \sym_2$-module splits  as $\triv_2 \oplus \sgn_2$.
\end{proof}

The objects $\sfb^* (F)$ and $\lfb^* (F)$ come equipped with product and coproduct maps that are  analogous to those for the classical symmetric (respectively exterior) algebras:
\begin{eqnarray*}
\sfb^m (F) \odot \sfb^n (F) 
&\rightarrow & 
\sfb^{m+n} (F) 
\\
\lfb^m (F) \odot \lfb^n (F) 
&\rightarrow & 
\lfb^{m+n} (F)
\\
\sfb^{m+n} (F)
&\rightarrow & 
\sfb^m (F) \odot \sfb^n (F) 
\\
\lfb^{m+n} (F)
&\rightarrow & 
\lfb^m (F) \odot \lfb^n (F) ,
\end{eqnarray*}
for $m, n \in \nat$. (Over a field of characteristic zero, the analogy evoked above can be made precise using the Schur correspondence and the identifications in Remark \ref{rem:schur_sfb_lfb}.)

Using these natural structure morphisms, we have the following:

\begin{lem}
\label{lem:induced structure morphisms}
For $F, G$ two $\kk \fb$-modules and $2 \leq n \in \nat$,
\begin{enumerate}
\item 
a morphism $\sfb^{ 2} (F) \rightarrow F \odot G$  induces 
$
\sfb^{ n} (F) \rightarrow \sfb^{n-1} (F)  \odot G
$
 via the composite 
$$
\sfb^{n} (F) 
\rightarrow 
\sfb^{ n-2 } (F) \odot \sfb^{2} (F) 
\rightarrow 
\sfb^{ n-2}(F) \odot F \odot G
\rightarrow 
\sfb^{ n-1}(F) \odot G,
$$
where the first morphism is the coproduct of $\sfb^* (F)$, the second the given map applied to $\sfb^{ 2} (F)$, and the last is induced by the product of $\sfb^* (F)$;
\item 
a morphism $\lfb^{ 2} (F) \rightarrow F \odot G$  induces $
\lfb^{n} (F) \rightarrow \lfb^{n-1}(F)  \odot G
$ 
via the composite 
$$
\lfb^{n} (F) 
\rightarrow 
\lfb^{ n-2 } (F) \odot \lfb^{ 2} (F) 
\rightarrow 
\lfb^{n-2}(F) \odot F \odot G
\rightarrow 
\lfb^{ n-1}(F) \odot G,
$$
where the first morphism is the coproduct of $\lfb^* (F)$, the second the given map applied to $\lfb^{2} (F)$, and the last is induced by the product of $\lfb^* (F)$.
\end{enumerate}
For $n<2$, these morphisms are taken to be zero.
\end{lem}

\begin{rem}
In the applications of the above Lemma, the $\kk \fb$-module $G$ will either be $\triv_2$ or $\sgn_2$, $F$ will be the underlying $\kk \fb$-module of a non-unital cyclic operad $\cpd$, and the given morphisms will arise from the structure morphism of $\cpd$ (see Lemmas \ref{lem:cpd_composite_map} and \ref{lem:cpd_composite_map_lfb} below).
\end{rem}

To make the construction more concrete, we offer the following:

\begin{exam}
\label{exam:extended_morphism_n=3}
Take $n=3$ and consider $F$ a $\kk \fb$-module such that $F(\mathbf{0})=0$. Consider
 the morphism $\sfb^3 (F) \rightarrow \sfb^2(F) \odot \triv_2$ obtained as in Lemma \ref{lem:induced structure morphisms} from a morphism $\sfb^2 (F) \rightarrow F \odot \triv_2$.

Evaluated on $X$, this gives the $\kk$-linear map
\begin{eqnarray}
\label{eqn:sfb_eval_X}
\bigoplus_{(U_1, U_2, U_3)} F (U_1 ) \otimes F(U_2) \otimes F(U_3) \rightarrow \bigoplus_{\substack{((V_1, V_2),V_3 ) \\|V_3|=2}} F (V_1) \otimes F(V_2),
\end{eqnarray}
where the sums are over decompositions of $X$ into three non-empty subsets, with the $(U_i)$ unordered, $(V_1, V_2)$ unordered and $|V_3|=2$, so that $\triv_2(V_3)= \kk$.

The component indexed by $(U_1, U_2, U_3)$ and $((V_1, V_2), V_3)$ is zero unless there exists $i\in \{1, 2,3\}$ and $j \in \{1, 2\}$ such that $U_i=V_j$. In the latter case, without loss of generality, we may assume that $U_1 = V_1$ and thus $U_2 \amalg U_3 = V_2 \amalg V_3 $. 

The morphism  $\sfb^2 (F) \rightarrow F \odot \triv_2$ provides the component (indexed by $(U_2, U_3)$ and $(V_2, V_3)$ and using that $\triv_2 (V_3) = \kk$): 
\[
F (U_2 ) \otimes F(U_3) \rightarrow F(V_2).
\]
Tensoring with $F(V_1)$ gives the component of (\ref{eqn:sfb_eval_X}) indexed by $(U_1, U_2, U_3)$ and $((V_1, V_2), V_3)$.

A similar analysis applies with $\lfb^*$ in place of $\sfb^*$. The only difference is that one must take into account that transposing tensor factors introduces a sign.
\end{exam}

\subsection{Cyclic operads}

We work with non-unital  cyclic operads in $\kk$-vector spaces; these form a category $\nuco$. (See \cite[Definitions 6.1 and 6.2]{MR4199072} for example  for definitions of cyclic operads and morphisms between such; for $\nuco$ we simply omit the axioms involving the unit.)

\begin{hyp}
\label{hyp:nuco_0} 
For $\cpd$ in $\nuco$, we will always assume that the underlying $\kk\fb$-module satisfies $\cpd (\mathbf{0})=0$. 
\end{hyp}

\begin{lem}
\label{lem:cpd_composite_map}
The composition structure map for $\cpd$ in $\nuco$ defines a morphism of $\kk \fb$-modules
\begin{eqnarray}
\label{eqn:cpd-structure}
\sfb^2 (\cpd) \rightarrow \cpd \odot \triv_2.
\end{eqnarray}
\end{lem}

\begin{proof}
We require to prove that, for each finite set $X$, the  structure map  of $\cpd$ yields  a $\kk \aut (X)$-equivariant map
\[
\bigoplus_{\substack{(U_1, U_2)\\ U_1 \amalg U_2 = X}} \cpd (U_1) \otimes \cpd (U_2) 
\rightarrow 
\bigoplus _{\substack{Y \subset X \\ |X|= |Y|+2}} \cpd (Y),
\]
where the sum in the domain is over unordered decompositions $(U_1, U_2)$.

For given $(U_1, U_2)$ and $Y$ as above, we take the component 
\[
\cpd (U_1) \otimes \cpd (U_2) 
\rightarrow 
 \cpd (Y)
\]
to be zero unless $|Y\cap U_1| = |U_1|-1$ (equivalently $|Y\cap U_2| = |U_2|-1$), when it is given by the structure map of $\cpd$. 
The fact that this defines a morphism of the form (\ref{eqn:cpd-structure}) follows from the equivariance properties of the structure morphism of a non-unital cyclic operad. 
\end{proof}

\begin{rem}
\label{rem:cpd-structure-condition}
The proof identifies an additional condition on (\ref{eqn:cpd-structure}) that is satisfied when the morphism corresponds to the composition map of a non-unital cyclic operad. Namely, the component of (\ref{eqn:cpd-structure}) 
\[
\cpd (U_1) \otimes \cpd (U_2) 
\rightarrow 
 \cpd (Y)
\]
is zero if either $U_1 \subset Y$ or $U_2 \subset Y$. Of course, this condition does not hold in general for modular operads as in \cite{MR4199072} (cf. also \cite{MR1601666}).
\end{rem}

There is a counterpart of Lemma \ref{lem:cpd_composite_map} obtained by replacing $\sfb^2 (\cpd) $ by $\lfb^2 (\cpd)$: 

\begin{lem}
\label{lem:cpd_composite_map_lfb}
The structure map of  $\cpd$ in $\nuco$ defines a morphism of $\kk \fb$-modules
\begin{eqnarray}
\label{eqn:cpd-structure_lfb}
\lfb^2 (\cpd) \rightarrow \cpd \odot \sgn_2.
\end{eqnarray}
\end{lem}

\begin{proof}
We claim that the passage to $\lfb^2 (\cpd)$ means that an orientation sign is required corresponding to an order of the pair of elements $X \backslash Y$ (using the notation of the proof of Lemma \ref{lem:cpd_composite_map}). 

This is seen as follows. Given $(U_1, U_2)$ and $Y \subset X$ such that $|U_1 \cap Y|= |U_1 |-1 $ and $|U_2 \cap Y| = |U_2|-1$, set $x_i := U_i \backslash (U_i\cap Y)$ so that $X\backslash Y = \{x_1, x_2\}$. A choice of ordering of $U_1$ and $U_2$ is equivalent to a choice of ordering of the set $\{x_1, x_2\}$. Thus, the sign appearing in defining $\lfb^2 (\cpd)$ requires the usage of $\sgn_2$.  
\end{proof}

\subsection{Modules over Brauer categories associated to non-unital cyclic operads}

Recall that $\db$ is the downward Brauer category, with $\kk$-linearization $\kk \db$;  $\kdbmm$ is the twisted variant, as  defined in  Section \ref{sect:twist}.

Using that these are homogeneous quadratic $\kk$-linear categories (see Proposition \ref{prop:kub_twist_homogeneous_quadratic}),  we have the following (in which the restriction to characteristic $\neq 2$ is due to the appeal to Lemma \ref{lem:split_F_odot_F}):

\begin{prop}
\label{prop:kdb_modules}
Suppose that $\kk$ is a field of characteristic $\neq 2$.
\begin{enumerate}
\item 
A $\kk \db$-module structure with underlying $\kk\fb$-module $M$ is uniquely determined by the structure morphism 
 $ \psi : M \rightarrow M \odot \triv_2$ corresponding to the action of the degree one morphisms.
 
Conversely, given such a structure morphism, this defines a $\kk \db$-module structure on $M$ if and only if the composite  
\[
M \stackrel{\psi}{\rightarrow}
 M \odot \triv_2 
\stackrel{\psi \odot \id} {\rightarrow}
M \odot \triv_2 \odot \triv_2 
\rightarrow 
M \odot \lfb^2 (\triv_2)
\]
is zero, using $\triv _2 \odot \triv_2 \rightarrow \lfb^2 (\triv_2)$,  the projection furnished by Lemma \ref{lem:split_F_odot_F}.
\item 
A $\kdbmm$-module structure with underlying $\kk\fb$-module $N$ is uniquely determined by the structure morphism 
 $ \xi : N \rightarrow N \odot \sgn_2$ corresponding to the action of the degree one morphisms.
 
Conversely, given such a structure morphism, this defines a $\kdbmm$-module structure on $N$ if and only if the composite  
\[
N \stackrel{\xi}{\rightarrow}
 N \odot \sgn_2 
\stackrel{\xi \odot \id} {\rightarrow}
N \odot \sgn_2 \odot \sgn_2 
\rightarrow 
N \odot \sfb^2 (\sgn_2)
\]
is zero, using $\sgn _2 \odot \sgn_2 \rightarrow \sfb^2 (\sgn_2)$,  the projection furnished by Lemma \ref{lem:split_F_odot_F}.
\end{enumerate}
\end{prop}

\begin{proof}
In both cases, the first statement follows from the fact that the $\kk$-linear category is generated over $\kk \fb$ by the degree $1$ morphisms, together with the identification of these (cf. Lemma \ref{lem:kub_bimodule_special_case}, which treats different twists). The converse then follows by taking into account the quadratic relations; these correspond to imposing the appropriate (signed) commutativity corresponding to the orientation sign associated to the order of chords. 
\end{proof}

Putting together Lemmas \ref{lem:cpd_composite_map} and \ref{lem:cpd_composite_map_lfb} with Lemma \ref{lem:induced structure morphisms}, if $\cpd$ is a non-unital cyclic operad, then the composition morphism induces natural morphisms
\begin{eqnarray}
\label{eqn:psi}
\psi &\colon & \sfb^* (\cpd) \rightarrow \sfb^*(\cpd) \odot \triv_2 \\
\label{eqn:xi}
\xi & \colon & \lfb^* (\cpd) \rightarrow \lfb^* (\cpd) \odot \sgn_2.
\end{eqnarray}

\begin{thm}
\label{thm:cpd_db_kdbm}
For $\cpd$ a non-unital cyclic operad,  
\begin{enumerate}
\item 
$\sfb^{*} (\cpd)$ is a  $\kk \db$-module with structure morphism (\ref{eqn:psi}); 
\item 
$\lfb^{*} (\cpd)$ is a  $\kdbmm$-module with structure morphism (\ref{eqn:xi}).
\end{enumerate}
These define functors from $\nuco$ to $\kk\db\dash\modules$ and $\kdbmm\dash\modules$ respectively.
\end{thm}

\begin{proof}
This result is proved by applying Proposition \ref{prop:kdb_modules}. In each case, we require to show that the specified structure morphism satisfies the `quadratic relation' condition given in that Proposition.

For the case of $\sfb^* (\cpd)$, this is a direct consequence of the condition for composing structure maps for a cyclic operad, given as \cite[Definition 6.1, equation (6.3)]{MR4199072}. 

For the case of $\lfb^* (\cpd)$, one has to take into account the `Koszul-type' signs that are introduced by working with $\lfb^* (\cpd)$. Notably, these intervene when considering the coproduct and the product of $\lfb^* (\cpd)$ used in Lemma \ref{lem:induced structure morphisms}. This accounts for the anti-symmetric behaviour of the permutation of chords. 
(This can be checked explicitly by using the  description of the structure morphisms, as in Example \ref{exam:extended_morphism_n=3}. The explicit argument is encapsulated by Example \ref{exam:lfb-module_explicit} below.)
\end{proof}

\subsection{The case of algebras with involution}
\label{subsect:alg_inv}

Recall that an algebra with involution is an associative algebra $B$ equipped with an involution $\sigma : B \stackrel{\cong}{\rightarrow} B\op$ (i.e., $\sigma$ is an isomorphism of the underlying abelian groups such that $\sigma^2 = \id$ and $\sigma (b_1b_2) = \sigma (b_2)\sigma (b_1)$). These will be referred to as $\kk$-algebras with involution.

\begin{rem}
The algebra is unital (as an algebra with involution) if there is a unit  $1 \in B$ such that $\sigma (1) =1$. Working with $\kk$-algebras, we require that $\sigma$ is $\kk$-linear and that $\kk \subset B$ (induced by the unit) is central.  
\end{rem}

The $\kk$-algebras with involution form a category; in the unital case, morphisms must respect the unit. The following is well-known in the unital case; here we do not require a unit.

\begin{prop}
\label{prop:cyclic_operads_alg_inv}
There is an equivalence of categories between $\kk$-algebras with involution and the category of non-unital cyclic operads in $\kk$-modules  that are supported on $\mathbf{2}$.
\end{prop}

\begin{proof}
Given a non-unital cyclic operad $\cpd$ supported on $\mathbf{2}$ (i.e., $\cpd (\n) = 0$ if $n \neq 2$), set $(B,\sigma)$ to be the pair $(\cpd (\mathbf{2}), \cpd (\tau))$, were $\tau  \in \sym_2$ is the transposition. Now, the cyclic operad composition defines 
\[
\cpd (\{ 1, 2\} ) \otimes \cpd (\{ 3, 4 \}) \stackrel{\iota_{2,3}}{\rightarrow } \cpd (\{1, 4\}).
\] 
Identifying $\cpd (\{3, 4\})$ and $\cpd (\{1, 4\})$ with $B$ via the order preserving bijections $\mathbf{2} \cong \{3,4\}$ and $\mathbf{2} \cong \{1 , 4\}$, this defines
$$
B \otimes B 
\rightarrow 
B.
$$
Using the axioms of a cyclic operad, one checks that this makes $(B,\sigma)$ into a $\kk$-algebra with involution.

One also has the converse: given $(B,\sigma)$ a $\kk$-algebra with involution, there is a non-unital cyclic operad $\cpd_{(B,\sigma)}$ supported on $\mathbf{2}$ with $\cpd _{(B,\sigma)} (\mathbf{2})= B$ and $\kk \sym_2$-module structure given by $\sigma$; the composition is induced by the algebra structure of $B$ so as to be compatible with the previous construction.

One checks directly that these constructions induce an equivalence of categories, as required.
\end{proof}

The above proof was sketched so as to stress the fact that the identification of $\cpd_{(B,\sigma)}(\{x,y\})$ with $B$ depends on a choice of bijection $\mathbf{2} \cong \{x, y\}$; this can be thought of as a choice of order of $\{x, y\}$. This gives a way of representing elements of $\cpd_{(B,\sigma)}(\{x,y\})$ by labelled directed chords, such as 
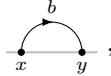
\begin{center}
\begin{tikzpicture}[scale = .2]
\begin{scope}
    \clip (-3,0) rectangle (3,3);
    \draw (0,0) circle(2);
\end{scope}
\draw [-latex](0,2) -- (.1,2);
\node at (90:3) {$\scriptstyle{b}$};
\draw [lightgray, thick] (-3,0) -- (3,0);
\draw [fill=black] (-2,0) circle (0.2);
\node [below] at (-2,0) {$\scriptstyle{x}$};
\draw [fill=black] (2,0) circle (0.2);
\node [below] at (2,0) {$\scriptstyle{y}$};
\node [right] at (3,0) {,};
\end{tikzpicture}
\end{center}
for an element $b \in B$. The orientation of the chord specifies the choice of order of $\{x, y\}$. This chord is defined to be equivalent to 
\begin{center}
\begin{tikzpicture}[scale = .2]
\begin{scope}
    \clip (-3,0) rectangle (12,3);
    \draw (0,0) circle(2);
    \draw (9,0) circle(2);
\end{scope}
\draw [-latex](0,2) -- (-.1,2);
\node at (90:3) {$\scriptstyle{\sigma(b)}$};
\draw [lightgray, thick] (-3,0) -- (3,0);
\draw [fill=black] (-2,0) circle (0.2);
\node [below] at (-2,0) {$\scriptstyle{x}$};
\draw [fill=black] (2,0) circle (0.2);
\node [below] at (2,0) {$\scriptstyle{y}$};

\node at (4.5, 1) {$\scriptstyle{=}$}; 

\draw [-latex](9,2) -- (9.1,2);
\node at (9,3) {$\scriptstyle{\sigma (b)}$};
\draw [lightgray, thick] (6,0) -- (12,0);
\draw [fill=black] (7,0) circle (0.2);
\node [below] at (7,0) {$\scriptstyle{y}$};
\draw [fill=black] (11,0) circle (0.2);
\node [below] at (11,0) {$\scriptstyle{x}$};
\node [right] at (11.5,0) {,};
\end{tikzpicture}
\end{center}
i.e., changing the orientation of the chord invokes $\sigma$ (the equality reflects the fact that the diagram is independent of the planar embedding). This relation will be referred to as the {\em $\sigma$-chord orientation relation}. Once this relation is taken into account, there is a bijection between equivalence classes of such diagrams and $B$. 

 The  $\kk$-module structure on such diagrams is the obvious one, by imposing $\kk$-linearity with respect to $b\in B$.

Using this diagrammatic interpretation, the multiplication of $B$ corresponds (for elements $b, c \in B$) to 
\begin{equation}
\label{eqn:iota_uv}
\begin{tikzpicture}[scale = .2]
\begin{scope}
    \clip (-3,0) rectangle (20,3);
    \draw (0,0) circle(2);
     \draw (7,0) circle(2);
         \draw (16,0) circle(2);
\end{scope}
\draw [-latex](0,2) -- (.1,2);
\node at (0,3) {$\scriptstyle{b}$};
\draw [lightgray, thick] (-3,0) -- (3,0);
\draw [fill=black] (-2,0) circle (0.2);
\node [below] at (-2,0) {$\scriptstyle{x}$};
\draw [fill=black] (2,0) circle (0.2);
\node [below] at (2,0) {$\scriptstyle{u}$};

\node at (3.5, 1) {$\scriptstyle{\otimes}$}; 

\draw [-latex](7,2) -- (7.1,2);
\node at (7,3) {$\scriptstyle{c}$};
\draw [lightgray, thick] (4,0) -- (10,0);
\draw [fill=black] (5,0) circle (0.2);
\node [below] at (5,0) {$\scriptstyle{v}$};
\draw [fill=black] (9,0) circle (0.2);
\node [below] at (9,0) {$\scriptstyle{y}$};

\draw [|->](10,1) -- (13,1);
\node at (11.5,2) {$\scriptstyle{\iota_{u,v}}$};

\draw [-latex](16,2) -- (16.1,2);
\node at (16,3) {$\scriptstyle{bc}$};
\draw [lightgray, thick] (13,0) -- (19,0);
\draw [fill=black] (14,0) circle (0.2);
\node [below] at (14,0) {$\scriptstyle{x}$};
\draw [fill=black] (18,0) circle (0.2);
\node [below] at (18,0) {$\scriptstyle{y}$};
\node [right] at (19,0) {,};
\end{tikzpicture}
\end{equation}
in which the two nodes contracted are adjacent and all the  chords have compatible orientations, as indicated.  
All other contraction structure maps are reduced to this case by exploiting the $\sigma$-chord orientation relation (for this, one can use Lemma \ref{lem:normalization} below). 

Recall from Section \ref{sect:fiord} that, for $X$ a set of cardinality $2t$, for $0<t \in \nat$, we consider $\fb (\mathbf{2t}, X)$ as the  set of decorated chord diagrams on $X$,  in which each chord is directed and the chords are ordered. Explicitly, the $i$th chord (for $1 \leq i \leq t$) is given by the ordered pair $(f(2i-1), f(2i))$. 
 There is a free right action of $\sym_2 \wr \sym_t \subset \sym_{2t}$ by precomposition, and the quotient set $\fb (\mathbf{2t}, X)/ \sym_2 \wr \sym_t$ identifies as the set of  undecorated chord diagrams, which also identifies as  
$$
\ub (\mathbf{0}, X) = \fb (\mathbf{2t}, X)/ \sym_2\wr \sym_t. 
$$

The following is clear; it is stated so as  to distinguish the two cases: 

\begin{lem}
\label{lem:cases}
For $f\in \fb (\mathbf{2t}, X)$ as above and $(u,v)$ an ordered pair  of distinct elements of $X$, precisely one of the following holds: 
\begin{enumerate}
\item 
$\exists j^f_{u,v} \in \mathbf{t}$ such that $\{ f(2j^f_{u,v}-1) , f(2j^f_{u,v}) \} = \{ u, v\}$ (as unordered sets); 
\item 
$\exists i^f_u \neq i^f_v \in \mathbf{t}$ such that $u \in \{ f(2i^f_u-1) , f(2i^f_u) \}$ and $v \in \{ f(2i^f_v-1) , f(2i^f_v) \}$.
\end{enumerate}
\end{lem}

Using the notation of this lemma, we define the following elements of $\{0, 1\}$ (viewed as the group $\zed/2$): 
\begin{enumerate}
\item 
in the first case, $\epsilon^f_{u,v}$ is $0$ if $u=  f(2j^f_{u,v}-1) $ and $1$ otherwise; 
\item 
in the second case, $\epsilon^f_u$ is $0$ if $u = f(2i^f_u) $ and $1$ otherwise; $\epsilon^f_v$ is $0$ if $v = f(2i^f_v-1) $ and $1$ otherwise.
\end{enumerate}

Recall that $\sym_2 \wr \sym_t$ is the semi-direct product $\sym_2^{\times t} \rtimes \sym_t$, in particular contains the distinguished subgroups $\sym_2 ^{\times t}$ and $\sym_t$. 
We introduce the following group elements:
\begin{enumerate}
\item 
For $i \in \mathbf{t}$, let $\tau_i \in \sym_2^{\times t}$ be the element given by the transposition $\tau \in \sym_2$ on the $i$th factor, the identity element elsewhere.
\item 
For $j \in \mathbf{t}$, let $\rho_j \in \sym_t$ be the shuffle given by $t \mapsto j$ and the order preserving map $\mathbf{t-1} \rightarrow \mathbf{t} \backslash \{j\}$. 
\item 
For $i_1 \neq  i_2 \in \mathbf{t}$, let $\psi_{i_1,i_2}$ be the shuffle given by $t-1 \mapsto i_1$, $t \mapsto i_2$ and the order preserving map $\mathbf{t-2} \rightarrow \mathbf{t} \backslash \{i_1, i_2\}$.  
\end{enumerate}

\begin{defn}
\label{defn:group_elements}
For $f\in \fb (\mathbf{2t}, X)$ and  $(u,v)$ an ordered pair of distinct elements in $X$, define the following elements of $\sym_2 \wr \sym_t$:
\begin{enumerate}
\item 
in case (1) of Lemma \ref{lem:cases}, $\alpha_{u,v}^f :=  \rho_{j^f_{u,v}}\tau_t^{\epsilon^f_{u,v}}$; 
\item 
in case (2) of Lemma \ref{lem:cases}, $\beta_{u,v}^f :=  \psi_{i^f_u,i^f_v}\tau_ {t-1}^{\epsilon^f_u} \tau_t ^{\epsilon^f _v}$. 
\end{enumerate}
\end{defn}

By construction of $\alpha_{u,v}^f$ and $\beta_{u,v}^f$, the following normalization result holds, designed to normalize to the `contraction' $\iota_{u,v}$ as illustrated in (\ref{eqn:iota_uv}) above.

\begin{lem}
\label{lem:normalization}
For $X$ a set of cardinality $2t>0$ and $f \in \fb (\mathbf{2t}, X)$ and an ordered pair $(u,v)$ of distinct elements of $X$: 
\begin{enumerate}
\item 
in the first case of Lemma \ref{lem:cases},  $f\circ \alpha^f_{u,v} (2t-1) = u$ and $f \circ \alpha^f _{u,v} (2t)=v$; 
\item 
in the second case of Lemma \ref{lem:cases}, $f\circ \beta^f_{u,v} (2t-2) = u$ and $f \circ \beta^f _{u,v} (2t-1)=v$.
\end{enumerate}
\end{lem}

 We now proceed to analyse $\sfb^* (\cpd_{(B,\sigma)})$ and $\lfb^* (\cpd_{(B,\sigma)})$. For this, we extend  
the above graphical representation to decorated chord diagrams, with nodes labelled by elements of a fixed set. Taking three chords with endpoints labelled by $\mathbf{6}$, one has for example,

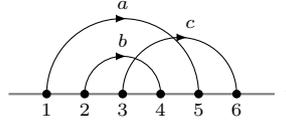
\begin{center}
\begin{tikzpicture}[scale = .5]

\begin{scope}
    \clip (0,-7) rectangle (7,-4.5);
    \draw (3,-7) circle(2);
    \draw (3,-7) circle (1); 
    \draw (4.5, -7) circle (1.5);
\end{scope}

\draw [-latex](3,-5) -- (3.1,-5);
\draw [-latex](3,-6) -- (3.1,-6);

\draw [-latex] (4.6,-5.5) -- (4.7,-5.5);

\draw [gray, thick] (0,-7) -- (7,-7);
\draw [fill=black] (1,-7) circle (0.1);
\node [below] at (1,-7) {$\scriptstyle{1}$};
\draw [fill=black] (2,-7) circle (0.1);
\node [below] at (2,-7) {$\scriptstyle{2}$};
\draw [fill=black] (3,-7) circle (0.1);
\node [below] at (3,-7) {$\scriptstyle{3}$};
\draw [fill=black] (4,-7) circle (0.1);
\node [below] at (4,-7) {$\scriptstyle{4}$};
\draw [fill=black] (5,-7) circle (0.1);
\node [below] at (5,-7) {$\scriptstyle{5}$};
\draw [fill=black] (6,-7) circle (0.1);
\node [below] at (6,-7) {$\scriptstyle{6}$};

\node [above] at (3,-5) {$\scriptstyle{a}$};
\node [above] at (3,-6) {$\scriptstyle{b}$};
\node [above right] at (4.4, -5.5) {$\scriptstyle{c}$};
\node [right] at (7,-7) {,};
\end{tikzpicture}
\end{center} 
 for $a, b , c \in B$. (Here the ordering of $\mathbf{6}$ has been used to induce the standard choice of orientation of the chords and can be used to order the chords using the order of their left hand endpoints.)

 These diagrams are subject to the obvious $\kk$-multilinearity condition and the $\sigma$-chord orientation relation.
 The automorphism group (in the above example, $\sym_6$) acts  by relabelling the nodes and then rewriting the diagram using the canonical order on $\mathbf{6}$, keeping track of the orientation of the chords and their order. The $\sigma$-chord orientation relation can be used to ensure that all chords have the `clockwise' orientation with respect to this planar representation given by the canonical order on $\mathbf{6}$.
 
To state the identification of the underlying $\kk \fb$-modules in Lemma \ref{lem:sfb_lfb_algebra_inv} below, 
we recall the following:

\begin{lem}
\label{lem:action_wreath_product}
Suppose that $M$ is a $\kk \sym_2$-module and $t$ is a positive integer, then there is a canonical action of $\sym_2 \wr \sym_t$ on $M^{\otimes t}$ characterized by the following:
\begin{enumerate}
\item 
the subgroup $\sym_t\subset \sym_2 \wr \sym_t$ acts via place permutations on $M^{\otimes t}$; 
\item 
restricted to the subgroup $\sym_2 ^{\times t} \subset \sym_2 \wr \sym_t$, the module identifies as the $t$-fold exterior tensor product $M^{\boxtimes t}$. 
\end{enumerate}
\end{lem}

This allows the $\kk \fb$-module structure of $\sfb^* (\cpd_{(B, \sigma)})$ and  of $\lfb^* (\cpd_{(B, \sigma)})$ to be described, formalizing the chord diagram approach sketched above.

 \begin{lem}
 \label{lem:sfb_lfb_algebra_inv}
 For $(B,\sigma)$ a $\kk$-algebra with involution and $\cpd_{(B,\sigma)}$ the associated non-unital cyclic operad, the underlying $\kk \sym_n$-modules of $\sfb^* (\cpd_{(B, \sigma)})$ and $\lfb^* (\cpd_{(B, \sigma)})$, for $n \in \nat$, are given respectively by  
\begin{eqnarray*}
\sfb^* (\cpd_{(B,\sigma)})(\n) 
&= &
\left\{ 
\begin{array}{ll}
0 & n \mbox{ odd}
\\
\fb (\mathbf{2t} , \n) \otimes _{\kk \sym_2 \wr \sym_t} B^{\otimes t} & 
n =2t;
\end{array} 
\right. 
 \\
\lfb^* (\cpd_{(B,\sigma)})(\n) 
&= &
\left\{ 
\begin{array}{ll}
0 & n \mbox{ odd}
\\
\fb (\mathbf{2t} , \n)\otimes _{\kk \sym_2 \wr \sym_t} (B^{\otimes t} \otimes \kk_{(+;-)}^{[t]})& 
n = 2t,
\end{array} 
\right.  
\end{eqnarray*} 
where $\sym_2 \wr \sym_t$ acts diagonally upon $B^{\otimes t} \otimes \kk_{(+;-)}^{[t]}$, where the superscript $^{[t]}$ references the group $\sym_2 \wr \sym_t$.
\end{lem}
 
\begin{proof}
The identification of the underlying representations follows directly from an analysis of the construction of the $t$-fold iterated Day convolution product and the action of the group $\sym_t$ on this. 
\end{proof}

So as to state Proposition \ref{prop:sfb_lfb_module_Bsigma} below, we introduce the following notation. 

\begin{nota}
\label{nota:overline_f_b}
For $f \in \fb (\mathbf{2t} , X)$ and $\underline{b} \in B^{\otimes t}$, write
\begin{enumerate}
\item 
$\overline{[f] \otimes \underline{b}}$ for the class of $[f] \otimes \underline{b}$ in 
$\sfb^* (\cpd_{(B,\sigma)})(\mathbf{2t}) = \fb (\mathbf{2t} , X)\otimes _{\kk \sym_2 \wr \sym_t} B^{\otimes t} $;
\item 
$\overline{[f] \otimes (\underline{b}\otimes 1)}$ for the class of $[f] \otimes (\underline{b}\otimes 1)$ in 
$\lfb^* (\cpd_{(B,\sigma)})(\mathbf{2t}) = \fb (\mathbf{2t} , X)\otimes _{\kk \sym_2 \wr \sym_t} (B^{\otimes t} \otimes \sgn_t)$.
\end{enumerate}
\end{nota}

Thus, any element of  $\sfb^* (\cpd_{(B,\sigma)})(\mathbf{2t}) $ can be written as a linear combination of elements of the form $\overline{[f] \otimes \underline{b}}$. Similarly for 
$\lfb^* (\cpd_{(B,\sigma)})(\mathbf{2t}) $.

To describe the  $\kk \db$-module  structure of $\sfb^* (\cpd_{(B, \sigma)})$ and the $(\kk \db)_{(-;-)}$-module structure of $\lfb^* (\cpd_{(B, \sigma)})$, we introduce the following notation for  `generators' for the morphisms.

\begin{nota}
\label{nota:g_uv}
For $0<t \in \nat$ and  $1 \leq u<v\leq 2t$, denote by
\begin{enumerate}
\item 
$g_{u,v}$  the morphism of $\db (\mathbf{2t}, \mathbf{2(t-1)})$ corresponding to the order-preserving bijection 
$\mathbf{2(t-1)} \cong \mathbf{2t} \backslash \{u, v \}$; 
\item 
$[g_{u,v}]$ the corresponding element of $ \kk \db (\mathbf{2t}, \mathbf{2(t-1)})$;
\item 
$[g_{u,v}]'$ the corresponding element of  $(\kk\db)_{(-;-)}(\mathbf{2t}, \mathbf{2(t-1)})$, using the given order $u<v$  (compare Lemma \ref{lem:basis_twisted_kub});
\item 
$\mathfrak{i}_t : \mathbf{2(t-1)} \hookrightarrow \mathbf{2t}$ the order preserving injection with image $\mathbf{2t} \backslash \{ 2t-2, 2t-1 \}$.
\end{enumerate} 
\end{nota}

Lemma \ref{lem:sfb_lfb_algebra_inv} has already determined the underlying $\kk \fb$-module structures of $\sfb^* (\cpd_{(B, \sigma)})$ and  of $\lfb^* (\cpd_{(B, \sigma)})$. 
To determine the full module structures, since $\kk\db$ and $(\kk \db)_{(-;-)}$ are both homogeneous quadratic $\kk$-linear categories, in particular generated in degree one, it suffices to specify for each $t, u,v$ as above, the action of $[g_{u,v}]$ (respectively $[g_{u,v}]'$).

In the following statement,  $\mu_B : B \otimes B \rightarrow B$ is the product of $B$ and $\beta^f_{u,v}$ is as in Definition \ref{defn:group_elements}

\begin{prop}
\label{prop:sfb_lfb_module_Bsigma}
Let $(B,\sigma)$ be a $\kk$-algebra with involution and  $\cpd_{(B,\sigma)}$ be the associated non-unital cyclic operad;  
fix  $0 <t \in \nat$ and $1 \leq u<v \leq 2t$.
\begin{enumerate}
\item 
The action of $[g_{u,v}]$ on $\sfb^* (\cpd_{(B,\sigma)}) (\mathbf{2t})$ is determined by the following.
For an element $\overline{[f] \otimes \underline{b}}$ of $\sfb^* (\cpd_{(B,\sigma)}) (\mathbf{2t})$, 
\begin{enumerate}
\item 
in  case (1) of Lemma \ref{lem:cases}, $[g_{u,v}] \big(\overline{ [f] \otimes \underline{b}}\big)=0$; 
\item 
in case (2) of Lemma \ref{lem:cases}, $[g_{u,v}] \big(\overline{ [f] \otimes \underline{b}}\big)  = \overline{[f'] \otimes \underline{b}'}$, 
where $f' = f \circ \beta^f_{u,v} \circ \mathfrak{i}_t$ and $\underline{b}'= (\id_{B^{\otimes t-2}} \otimes \mu_B)\circ (\beta^f _{u,v})^{-1} (\underline{b})$.
\end{enumerate}
\item 
The action of $[g_{u,v}]'$ on $\lfb^* (\cpd_{(B,\sigma)}) (\mathbf{2t})$ is determined by the following.
For an element $\overline{[f] \otimes (\underline{b}\otimes 1)}$ of $\lfb^* (\cpd_{(B,\sigma)}) (\mathbf{2t})$, 
\begin{enumerate}
\item 
in  case (1) of Lemma \ref{lem:cases}, $[g_{u,v}]' \big(\overline{ [f] \otimes (\underline{b}\otimes 1)}\big)=0$; 
\item 
in case (2) of Lemma \ref{lem:cases}, $[g_{u,v}]' \big(\overline{ [f] \otimes (\underline{b}\otimes 1)}\big)  = \overline{[f'] \otimes (\underline{b}'\otimes 1)}$, 
where $f' = f \circ \beta^f_{u,v} \circ \mathfrak{i}_t$ and $\underline{b}'\otimes 1= (\id_{B^{\otimes t-2}} \otimes \mu_B \otimes \id_{\sgn_t})\circ (\beta^f _{u,v})^{-1} (\underline{b}\otimes 1)$.
\end{enumerate}
\end{enumerate}
\end{prop}

\begin{proof}
This follows from the construction of the group element $\beta^f_{u,v}$, the normalization result Lemma \ref{lem:normalization}, together with the construction of the respective module structures as given by Theorem \ref{thm:cpd_db_kdbm}. 
\end{proof}

\begin{exam}
\label{exam:initial_unital_alg_inv}
For $(B, \sigma) = (\kk, \id)$, the initial unital $\kk$-algebra with involution, one has the identifications of $\kk \sym_n$-modules.
\begin{eqnarray*}
\sfb^* (\cpd_{(\kk,\id)})(\n) 
&\cong & 
\kk \ub (\mathbf{0}, \n) 
\\
\lfb^* (\cpd_{(\kk,\id)})(\n) 
&\cong & 
(\kk \ub) _{(+;-)} (\mathbf{0}, \n).
\end{eqnarray*}

The identification of the  $\kk \db$-module (respectively $(\kk \db)_{(-;-)}$-module) structures is straightforward. For the first case $\sfb^* (\cpd_{(\kk,\id)})$, suppose that $n=2t$,  and consider an undecorated chord diagram with $t$ chords on $\mathbf{2t}$. Let $g_{u,v} \in \db (\mathbf{2t}, \mathbf{2(t-1)})$ be as in Notation \ref{nota:g_uv}. Then there are two possibilities: 
\begin{enumerate}
\item 
the pair $u<v$ is a chord; 
\item 
the pair $u<v$ is not a chord. 
\end{enumerate}

The action of $[g_{u,v}]$ on the generator given by the chord diagram sends this to zero in the first case; in the second case, it sends the generator to that corresponding to the chord diagram obtained by concatenating the chords containing $u$ and $v$ and relabelling the remaining endpoints by elements of $\mathbf{2(t-1)}$  preserving the order.

The analysis in the case of $\lfb^* (\cpd_{(\kk,\id)})$ is similar; the only difference being that there is an orientation sign that may intervene.
\end{exam}

\begin{exam}
\label{exam:lfb-module_explicit}
Consider $\lfb^3 (\cpd _{(B,\sigma)} )$ for a $\kk$-algebra with involution $(B,\sigma)$ and three elements $a$, $b$, and $c$ of $B$, decorating a chord diagram on the set $\{ s, t, x, u, v, y \}$ as illustrated below, considering the iteration first contracting $(u,v)$ and then contracting $(t,x)$. This gives:

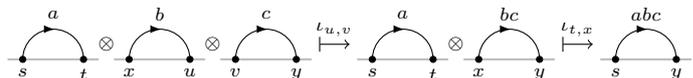
\begin{center}
\begin{tikzpicture}[scale = .2]
\begin{scope}
    \clip (-10,0) rectangle (40,3);
    \draw (-7,0) circle (2);
    \draw (0,0) circle(2);
    \draw (7,0) circle(2);
    \draw (16,0) circle(2);
	\draw (23,0) circle (2);          
    \draw (32,0) circle (2); 
\end{scope}

\draw [-latex](-7,2) -- (-6.9,2);
\node at (-7,3) {$\scriptstyle{a}$};
\draw [lightgray, thick] (-10,0) -- (-4,0);
\draw [fill=black] (-9,0) circle (0.2);
\node [below] at (-9,0) {$\scriptstyle{s}$};
\draw [fill=black] (-5,0) circle (0.2);
\node [below] at (-5,0) {$\scriptstyle{t}$};

\draw [-latex](0,2) -- (.1,2);
\node at (0,3) {$\scriptstyle{b}$};
\draw [lightgray, thick] (-3,0) -- (3,0);
\draw [fill=black] (-2,0) circle (0.2);
\node [below] at (-2,0) {$\scriptstyle{x}$};
\draw [fill=black] (2,0) circle (0.2);
\node [below] at (2,0) {$\scriptstyle{u}$};

\node at (3.5, 1) {$\scriptstyle{\otimes}$}; 
\node at (-3.5, 1) {$\scriptstyle{\otimes}$}; 
\node at (19.5, 1) {$\scriptstyle{\otimes}$}; 

\draw [-latex](7,2) -- (7.1,2);
\node at (7,3) {$\scriptstyle{c}$};
\draw [lightgray, thick] (4,0) -- (10,0);
\draw [fill=black] (5,0) circle (0.2);
\node [below] at (5,0) {$\scriptstyle{v}$};
\draw [fill=black] (9,0) circle (0.2);
\node [below] at (9,0) {$\scriptstyle{y}$};

\draw [|->](10.5,1) -- (12.5,1);
\node at (11.5,2) {$\scriptstyle{\iota_{u,v}}$};
\draw [|->](26.5,1) -- (28.5,1);
\node at (27.5,2) {$\scriptstyle{\iota_{t,x}}$};

\draw [-latex](16,2) -- (16.1,2);
\node at (16,3) {$\scriptstyle{a}$};
\draw [lightgray, thick] (13,0) -- (19,0);
\draw [fill=black] (14,0) circle (0.2);
\node [below] at (14,0) {$\scriptstyle{s}$};
\draw [fill=black] (18,0) circle (0.2);
\node [below] at (18,0) {$\scriptstyle{t}$};

\draw [-latex](23,2) -- (23.1,2);
\node at (23,3) {$\scriptstyle{bc}$};
\draw [lightgray, thick] (20,0) -- (26,0);
\draw [fill=black] (21,0) circle (0.2);
\node [below] at (21,0) {$\scriptstyle{x}$};
\draw [fill=black] (25,0) circle (0.2);
\node [below] at (25,0) {$\scriptstyle{y}$};

\draw [-latex](32,2) -- (32.1,2);
\node at (32,3) {$\scriptstyle{abc}$};
\draw [lightgray, thick] (29,0) -- (35,0);
\draw [fill=black] (30,0) circle (0.2);
\node [below] at (30,0) {$\scriptstyle{s}$};
\draw [fill=black] (34,0) circle (0.2);
\node [below] at (34,0) {$\scriptstyle{y}$};

\node [right] at (34.5,0) {.};
\end{tikzpicture}
\end{center}

\noindent
In this case, no signs arise, since we do not require to reorder the chords.

If we reverse the order of the contractions, then a sign arises, similarly to the appearance of signs in the Chevalley-Eilenberg complex:

\begin{center}
\begin{tikzpicture}[scale = .2]
\begin{scope}
    \clip (-10,0) rectangle (40,3);
    \draw (-7,0) circle (2);
    \draw (0,0) circle(2);
    \draw (7,0) circle(2);
    \draw (16,0) circle(2);
	\draw (23,0) circle (2);          
    \draw (32,0) circle (2); 
\end{scope}

\draw [-latex](-7,2) -- (-6.9,2);
\node at (-7,3) {$\scriptstyle{a}$};
\draw [lightgray, thick] (-10,0) -- (-4,0);
\draw [fill=black] (-9,0) circle (0.2);
\node [below] at (-9,0) {$\scriptstyle{s}$};
\draw [fill=black] (-5,0) circle (0.2);
\node [below] at (-5,0) {$\scriptstyle{t}$};

\draw [-latex](0,2) -- (.1,2);
\node at (0,3) {$\scriptstyle{b}$};
\draw [lightgray, thick] (-3,0) -- (3,0);
\draw [fill=black] (-2,0) circle (0.2);
\node [below] at (-2,0) {$\scriptstyle{x}$};
\draw [fill=black] (2,0) circle (0.2);
\node [below] at (2,0) {$\scriptstyle{u}$};

\node at (3.5, 1) {$\scriptstyle{\otimes}$}; 
\node at (-3.5, 1) {$\scriptstyle{\otimes}$}; 
\node at (19.5, 1) {$\scriptstyle{\otimes}$}; 

\draw [-latex](7,2) -- (7.1,2);
\node at (7,3) {$\scriptstyle{c}$};
\draw [lightgray, thick] (4,0) -- (10,0);
\draw [fill=black] (5,0) circle (0.2);
\node [below] at (5,0) {$\scriptstyle{v}$};
\draw [fill=black] (9,0) circle (0.2);
\node [below] at (9,0) {$\scriptstyle{y}$};

\draw [|->](10.5,1) -- (12.5,1);
\node at (11.5,2) {$\scriptstyle{\iota_{t,x}}$};
\draw [|->](26.5,1) -- (28.5,1);
\node at (27.5,2) {$\scriptstyle{\iota_{u,v}}$};

\draw [-latex](16,2) -- (16.1,2);
\node at (16,3) {$\scriptstyle{ab}$};
\node at (13.5,1) {$\scriptstyle{-}$};
\draw [lightgray, thick] (13,0) -- (19,0);
\draw [fill=black] (14,0) circle (0.2);
\node [below] at (14,0) {$\scriptstyle{s}$};
\draw [fill=black] (18,0) circle (0.2);
\node [below] at (18,0) {$\scriptstyle{u}$};

\draw [-latex](23,2) -- (23.1,2);
\node at (23,3) {$\scriptstyle{c}$};
\draw [lightgray, thick] (20,0) -- (26,0);
\draw [fill=black] (21,0) circle (0.2);
\node [below] at (21,0) {$\scriptstyle{v}$};
\draw [fill=black] (25,0) circle (0.2);
\node [below] at (25,0) {$\scriptstyle{y}$};

\draw [-latex](32,2) -- (32.1,2);
\node at (32,3) {$\scriptstyle{abc}$};
\node at (29.5,1) {$\scriptstyle{-}$};
\draw [lightgray, thick] (29,0) -- (35,0);
\draw [fill=black] (30,0) circle (0.2);
\node [below] at (30,0) {$\scriptstyle{s}$};
\draw [fill=black] (34,0) circle (0.2);
\node [below] at (34,0) {$\scriptstyle{y}$};

\node [right] at (35,0) {.};
\end{tikzpicture}
\end{center}

\noindent
Namely, to apply $\iota_{t,x}$ we first reorder the chords so that the labels are in the order $c$, $a$, $b$ (this introduces no sign) and contracting yields the chords with labels in the order $c$, $ab$. Reordering so that $u$ is to the left of $v$ (as illustrated) introduces the sign. 

This illustrates the fact that $\lfb^* (\cpd)$ is a $(\kk \db) _{(-;-)}$-module, with orientation signs arising both from the directions of the chords (corresponding to $(u,v)$ and $(t,x)$ above) and from the ordering of these. 
\end{exam}

\section{The Conant-Vogtmann-Kontsevich Lie algebra}
\label{sect:kontsevich_lie}

The purpose of this section is to show how the $(\kk \db)_{(-;-)}$-module associated to a cyclic operad $\cpd$ given by Theorem \ref{thm:cpd_db_kdbm} relates to the Lie algebra structure introduced by Kontsevich and generalized to all cyclic operads by Conant and Vogtmann. This uses the generalized Schur functor construction applied to the  Koszul complex of the module, together with its relation with the   Chevalley-Eilenberg complex of the Lie algebra. This approach  is motivated by the observation that the above structures should be related by Brauer-Schur-Weyl duality (as presented in Section \ref{sect:forms}). 

Throughout this section, we work over a field $\kk$ of characteristic zero.

\subsection{Comparison with the Chevalley-Eilenberg complex}
\label{subsect:compare_CE}

Kontsevich \cite{MR1247289,MR1341841} exploited Lie algebras arising in his formal symplectic geometry; the construction was generalized by Conant and Vogtmann \cite{MR2026331}, who showed that, for $\cpd$ a cyclic operad and $(V, \omega)$ a symplectic vector space, the evaluation of its Schur functor $\cpd (V)$ is equipped with a canonical Lie algebra structure induced by the cyclic operad structure map and the symplectic form. The association $(V, \omega) \mapsto \cpd (V)$ is a functor from $\vsp$ to the category of Lie algebras over $\kk$ (using the notation of Section \ref{sect:forms}).

\begin{exam}
\label{exam:symplectic_matrices}
Let $(B, \sigma)$ be a $\kk$-algebra with involution and $(V,\omega)$ be a symplectic vector space. Thus one has the associated (non-unital) cyclic operad $\cpd_{(B,\sigma)}$ and the associated Schur functor $\cpd_{(B,\sigma)}(V)$  given by 
$
V^{\otimes 2} \otimes_{\kk\sym_2} B$; this is equipped with  the Conant-Vogtmann Lie algebra structure. 

For example, taking $(B, \sigma)$ to be the initial unital  $\kk$-algebra with involution $(\kk , \id)$, this gives 
$
\cpd_{(\kk, \id)} (V) = S^2 (V),
$ 
which identifies as a Lie algebra as $\mathfrak{sp}(V)$ (this is often written as $\mathfrak{sp} (2n)$, where $\dim V = 2n$). 

This identification generalizes, using the Lie algebra of symplectic matrices associated to a $\kk$-algebra with involution, as in \cite[Section 10.5]{MR1600246} for example. This is a sub Lie algebra $sp_V (B)$ of $gl_V (B)$ defined as in \cite[Section 10.5.3]{MR1600246} (modifying the notation to index by the symplectic vector space $(V, \omega)$ rather than its dimension). 

There is an isomorphism of Lie algebras 
\[
\cpd_{(B,\sigma)}(V) \cong sp_{ V} (B)
\]
that is natural with respect to the symplectic space $(V,\omega)$ and with respect to the $\kk$-algebra with involution $(B, \sigma)$. 
 This is proved by using the  isomorphism $V \cong V^\sharp$ induced by the symplectic form $\omega$.  Using a symplectic basis $\{x_i\}$ of $(V,\omega)$ such that,  for $1 \leq i \leq n= \frac{1}{2}\dim V$,  $\omega (x_i, x_{i+n}) =1$ and $\omega (x_i, x_j)=0$ otherwise, this is encoded by the matrix 
$$
\begin{pmatrix}
0 & -I_n 
\\
I_n & 0,
\end{pmatrix}
$$
which is the inverse of the matrix $J_n$ in {\em loc. cit.} that is used in defining $sp_V (B)$.

The above isomorphism induces  $V \otimes V \cong V \otimes  V^\sharp \cong \mathrm{End}(V)$. Using that, in characteristic zero, invariants and coinvariants for $\zed/2$ coincide, one shows that this induces an isomorphism of Lie algebras as stated.
\end{exam}

Let us return to the general case of a non-unital cyclic operad $\cpd$.  One can form the associated Chevalley-Eilenberg complex:
\[
(\Lambda^*\cpd (V), d_\mathsf{CE}).
\]
This is natural with respect to both  $\cpd \in \nuco$ and the symplectic vector space $(V, \omega)$. The underlying graded object $\Lambda^ * (\cpd (V))$ only depends on the underlying $\kk \fb$-module of $\cpd$ and on the underlying $\kk$-vector space $V$. The differential depends upon both the symplectic structure $(V, \omega)$ and the cyclic operad composition map of $\cpd$.

\begin{rem}
The differential $d_\mathsf{CE}$ encodes the Lie algebra structure of $\cpd (V)$: the component $\Lambda^2 (\cpd (V)) \rightarrow \cpd (V)$ is the Lie bracket; the fact that $d^2_\mathsf{CE}$ is zero is equivalent to the Jacobi identity. Thus, one can recover  the Lie algebra structure on $\cpd (V)$ from the structure of the Chevalley-Eilenberg complex.
\end{rem}

We show how to recover these structures by using  the $(\kk \db)_{(-;-)}$-module $\lfb^*(\cpd)$ given by Theorem \ref{thm:cpd_db_kdbm}. 

\begin{rem}
We consider $\lfb^*(\cpd)$ as a graded $(\kk \db)_{(-;-)}$-module using the degree of $\lfb^*$. This is compatible with the length grading of $(\kk \db)_{(-;-)}$ in an obvious sense (compare the discussion of gradings in Remark \ref{rem:grading_vker_otimes_Koszul}). 
\end{rem}

We form the Koszul complex 
\[
\mathscr{K}_- \otimes_{(\kk \db)_{(-;-)}} \lfb^* (\cpd),
\]
which is a complex in $(\kk \ub)_{(-;+)}$-modules.  Then, as in Proposition \ref{prop:V_omega_complex}, we have the complex 
\begin{eqnarray}
\label{eqn:vker_K-_lfb}
\spt^\bullet \otimes_{(\kk\ub)_{(-;+)} }\mathscr{K}_- \otimes _{(\kk \db)_{(-;-)}} \lfb^* (\cpd).
\end{eqnarray}
Evaluated on $(V, \omega)$, this has underlying object $(\lfb^* \cpd)(V)$, by {\em loc. cit.}. As observed in Remark \ref{rem:schur_sfb_lfb}, there  is a natural isomorphism of functors of the vector space $V$:  
\[
(\lfb^* \cpd)(V)
\cong 
\Lambda^* (\cpd (V)).
\]
This is the first step in establishing the following:

\begin{thm}
\label{thm:vker_K-_versus_CE}
For $\cpd$ a non-unital cyclic operad and $(V,\omega)$ a symplectic vector space, the complex (\ref{eqn:vker_K-_lfb}) is naturally isomorphic to the Chevalley-Eilenberg complex $(\Lambda^*\cpd (V), d_\mathsf{CE})$.
\end{thm}

\begin{proof}
In the preceding discussion, we have already established that the underlying graded objects are naturally isomorphic. It remains to show that the differentials correspond across this isomorphism. 

First we consider the respective differentials on the  second exterior powers. For the Chevalley-Eilenberg complex, this is the Lie bracket 
$ 
\Lambda^2 (\cpd (V)) \rightarrow \cpd (V).
$ 
For the complex (\ref{eqn:vker_K-_lfb}), the differential is induced by the structure map $\lfb^2 (\cpd) \rightarrow \cpd \odot \sgn_2$, which leads again to $\Lambda^2 (\cpd (V))\rightarrow \cpd (V)$, by Proposition \ref{prop:V_omega_complex}. One shows that these coincide by a direct comparison of the  Conant-Vogtmann construction of the Lie bracket of $\cpd (V)$ and the morphism given by Proposition \ref{prop:V_omega_complex}. 

It remains to show that this implies that the differentials coincide in higher degree.
For this, one compares the construction of (\ref{eqn:xi}) using Lemma \ref{lem:induced structure morphisms} with the construction of the Chevalley-Eilenberg differential.

The Chevalley-Eilenberg differential for a Lie algebra $\mathfrak{g}$ can be defined as the following composite: 
\[
\Lambda^* \mathfrak{g}
\rightarrow 
\Lambda^{*-2} \mathfrak{g}
\otimes
\Lambda^2 \mathfrak{g}
\rightarrow 
\Lambda^{*-2} \mathfrak{g}
\otimes \mathfrak{g}
\rightarrow 
\Lambda^{*-1} \mathfrak{g},
\]
where the first map is given by the coproduct of $\Lambda^*$, the second by the Lie bracket $\Lambda^2 \mathfrak{g}\rightarrow \g$ and the last by the product of $\Lambda^*$.
Thus, taking $\mathfrak{g}= \cpd(V)$, one sees that applying the classical Schur functor to (\ref{eqn:xi}) yields the Chevalley-Eilenberg differential, using the fact that the Schur functor is symmetric monoidal and the product and coproduct of $\lfb^*$ correspond under this to the respective structures on the classical exterior algebra.
\end{proof}

\begin{rem}
\label{rem:original_approach}
The author's original approach to proving such a  result was to consider the induced $(\kk \ub)_{(-;+)}$-module  
$ 
(\kk \ub)_{(-;+)} \otimes_{\kk \fb} \cpd
$ 
and to show that this has a natural Lie algebra structure. This then induces the Conant-Vogtmann Lie algebra structure on $\cpd (V)$ (for $(V,\omega)$ a symplectic vector space) on applying the generalized
 Schur functor $\spt^\bullet \otimes_{(\kk \ub)_{(-;+)} }-$ and evaluating on $(V, \omega)$.

Kazuo Habiro and Mai Katada informed the author in June 2025 that they have  results to this approach in their work in progress \cite{HK}.
\end{rem}

\subsection{The complexes and their homology}
\label{subsect:complexes_homology}
In Section \ref{subsect:compare_CE} we explained that the generalized Schur functor relates the Koszul complex 
\begin{eqnarray}
\label{eqn:complex_ext_K-}
\mathscr{K}_- \otimes_{(\kk \db)_{(-;-)}} \lfb^* (\cpd).
\end{eqnarray}
and the Chevalley-Eilenberg complex of the Lie algebra $\cpd (V)$. The homology of this complex $(\Lambda^* \cpd (-), d_\mathsf{CE})$, considered as a functor of $(V, \omega)$, is denoted $H_*^\mathsf{CE} ( \cpd (-))$.

Hence Proposition \ref{prop:identify_complexes} in conjunction with Lemma \ref{lem:weak_stabilization_homology} implies the following:

\begin{prop}
\label{prop:weakly_stabilize_CE_complex}
For $\cpd$ a non-unital cyclic operad,  there is a natural isomorphism of complexes of $(\kk \db)_{(-;+)}$-modules:
$$
\hom_{\f (\vsp)} ((\Lambda^* \cpd (-), d_\mathsf{CE}), \spt^\ast)
\cong 
\big( (\kk \ub)_{(-;+)}^\sharp \otimes_{(\kk \ub)_{(-;+)} } \mathscr{K}_- \otimes_{(\kk \db)_{(-;-)} } \lfb^* (\cpd) \big)^\sharp.
$$
In homology, there is a natural isomorphism of graded $(\kk \db)_{(-;+)}$-modules:
$$
\hom_{\f (\vsp)} (H_*^\mathsf{CE} ( \cpd (-)), \spt^\ast)
\cong 
H_* \big( (\kk \ub)_{(-;+)}^\sharp \otimes_{(\kk \ub)_{(-;+)} } \mathscr{K}_- \otimes_{(\kk \db)_{(-;-)} } \lfb^* (\cpd) \big)^\sharp.
$$
\end{prop}

This establishes the significance of the complex of $(\kk \ub)_{(-;+)}$-modules:
\begin{eqnarray}
\label{eqn:complex_tor_K-}
(\kk\ub)_{(-;+)}^\sharp  \otimes_{(\kk\ub)_{(-;+)}} \mathscr{K}_- \otimes_{(\kk\db)_{(-;-)}} \lfb^* (\cpd).
\end{eqnarray}

\begin{rem}
The underlying objects of (\ref{eqn:complex_ext_K-}) and (\ref{eqn:complex_tor_K-}) are  isomorphic respectively  to 
\begin{eqnarray*}
&&
(\kk \ub)_{(-;+)} \otimes_{\kk \fb} \lfb^* (\cpd)
\\
&&
(\kk\ub)_{(-;+)}^\sharp  \otimes_{\kk \fb} \lfb^* (\cpd)
.
\end{eqnarray*}
These are equipped with the respective Koszul differentials induced by that of $\mathscr{K}_-$.
\end{rem}

Theorem  \ref{thm:ext_tor_variant} gives the following  description of the (co)homology of these complexes in terms of the homological algebra of $(\kk \db)_{(-;-)}$-modules. 

\begin{cor}
\label{cor:homology_complexes_cpd}
For $\cpd$ a non-unital cyclic operad, 
\begin{enumerate}
\item 
the cohomology of (\ref{eqn:complex_ext_K-})
 is isomorphic to $\ext^*_{(\kk \db)_{(-;-)} } (\kk \fb , \lfb^*(\cpd))$;
\item 
the  homology of 
(\ref{eqn:complex_tor_K-})
is isomorphic to $\tor_*^{(\kk \db)_{(-;-)}}(\kk \fb, \lfb^* (\cpd))$.
\end{enumerate}
Moreover, the natural $(\kk \ub)_{(-;+)}$-module structures on the (co)homologies of these complexes identify  with the $\ext^* _{(\kk \db)_{(-;-)}}(\kk \fb, \kk \fb)$-module structures given respectively by the Yoneda product and the cap product.
\end{cor}

For the $\kk \db$-module $\sfb^* (\cpd)$, we have the counterparts:
\begin{eqnarray}
\label{eqn:complex_ext_K+}
&&\mathscr{K}_+ \otimes_{\kk \db} \sfb^*(\cpd) \\
\label{eqn:complex_tor_K+}
&&
(\kk \ub)_{(+;-)} ^\sharp \otimes _{(\kk \ub) _{(+;-)}} \mathscr{K}_+ \otimes_{\kk \db} \sfb^*(\cpd)
\end{eqnarray}
that have underlying objects 
\begin{eqnarray*}
&&
(\kk \ub)_{(+;-)} \otimes_{\kk \fb} \sfb^*(\cpd) \\
&&
(\kk \ub)_{(+;-)} ^\sharp \otimes _{\kk \fb} \sfb^*(\cpd) .
\end{eqnarray*}

Theorem \ref{thm:ext_tor} gives the following counterpart of Corollary \ref{cor:homology_complexes_cpd}:

\begin{cor}
For $\cpd$ a  non-unital cyclic operad, 
\begin{enumerate}
\item 
the cohomology of (\ref{eqn:complex_ext_K+})
 is isomorphic to $\ext^*_{\kk \db } (\kk \fb , \sfb^*(\cpd))$;
\item 
the  homology of 
(\ref{eqn:complex_tor_K+})
is isomorphic to $\tor_*^{\kk \db}(\kk \fb, \sfb^* (\cpd))$.
\end{enumerate}
\end{cor}

\subsection{Relating to Lie algebra homology}

For $\cpd$ a non-unital cyclic operad, 
Theorem \ref{thm:vker_K-_versus_CE}  leads to the question of the relationship between the homology of $
\mathscr{K}_- \otimes_{(\kk \db)_{(-;-)}} \lfb^* (\cpd)
$ and the homology of $(\Lambda^* (\cpd (V)), d_\mathsf{CE})$ for $(V, \omega)$ a symplectic vector space, since the latter is obtained from the former by applying the generalized Schur functor $\spt^\bullet \otimes_{(\kk \ub)_{(-;+)}} -$ and evaluating on $(V,\omega)$. Whilst the evaluation functor is exact, the generalized Schur functor is only right exact; thus a universal coefficients spectral sequence intervenes, as outlined below. 

Using the identification given by Corollary \ref{cor:homology_complexes_cpd}, there is a universal coefficients spectral sequence  (compare  Section \ref{subsect:relate_ext_tor}):
\[
\tor^{(\kk \ub)_{(-;+)} } _* \big(
\spt^\bullet, \ext^*_{(\kk \db)_{(-;-)} } (\kk \fb , \lfb^*(\cpd)) 
\big) 
\Rightarrow 
H^\mathsf{CE} _*(\cpd (-))
\]
as functors on $\vsp$, where $H^\mathsf{CE} _*(\cpd (-))$ is the Lie algebra homology of the Conant-Vogtmann Lie algebra $\cpd (-)$. (Here one must put an appropriate grading on $\ext^*_{(\kk \db)_{(-;-)} } (\kk \fb , \lfb^*(\cpd))$ and $H^\mathsf{CE} _*(\cpd (-))$.)

In particular, there is an edge homomorphism:
\[
\spt^\bullet
\otimes_{(\kk \ub)_{(-;+)} }
\ext^*_{(\kk \db)_{(-;-)} } (\kk \fb , \lfb^*(\cpd)) 
\rightarrow 
H^\mathsf{CE} _*(\cpd (-)).
\]
The domain can be viewed as a first approximation to calculating $H^\mathsf{CE} _*(\cpd (-))$ as a functor on $\vsp$.

\begin{rem}
In order to implement this strategy, we require not only to calculate $\ext^*_{(\kk \db)_{(-;-)} } (\kk \fb , \lfb^*(\cpd))  $ but also to 
 understand its  $(\kk \ub)_{(-;+)} $-module structure. A natural first case to consider  is when  $\cpd$ is the non-unital cyclic operad associated to a $\kk$-algebra with involution (see Section \ref{sect:examples_alg_inv}).

In Section \ref{subsect:non-torsion_Ext} we consider $\ext^*_{(\kk \db)_{(-;-)} } (\kk \fb , \lfb^*(\cpd_{(\kk, \id)}))$, for the initial unital $\kk$-algebra with involution $(\kk ,\id)$. In particular, we show that this $\ext^*$ is not torsion. This provides some evidence that the above edge homomorphism contains useful information.
\end{rem}

\section{An acyclicity result}
\label{sect:acyclicity}

A cornerstone of Kontsevich's result relating the stable Lie algebra homology of $\cpd (V)$, for $\cpd$ a cyclic operad with unit, is the fact that the stabilization 
$$
\stabvsp H^\mathsf{CE}_* (\cpd (-))
$$ 
has trivial $\spgp_\infty$-action. This is based on the fact that, for any Lie algebra $\mathfrak{g}$, the adjoint action of $\mathfrak{g}$ on its homology $H^\mathsf{CE}_* (\mathfrak{g})$ is trivial (see, for instance, \cite[Proposition 10.1.7]{MR1600246}). This is then applied at the level of the chain complexes (as in \cite[Proposition 10.1.8]{MR1600246}).

This result can be reinterpreted by using Theorem \ref{thm:stabilization_trivial_action}. Combined with the result of Corollary \ref{cor:indentify_homology}, this implies that the complex 
$$
(\kk\ub)_{(-;+)}^\sharp  \otimes_{(\kk\ub)_{(-;+)}} \mathscr{K}_- \otimes_{(\kk\db)_{(-;-)}} \lfb^{*} (\cpd)
$$
of $(\kk\ub)_{(-;+)}$-modules has zero homology when evaluated on a non-empty finite set. 

The purpose of this section is to sketch a direct proof of this result, without appealing to the Lie algebra argument.  This also applies to the complex associated to the $\kk \db$-module $\sfb^{*} (\cpd)$.

\begin{thm}
\label{thm:acyclicity}
Let $\cpd$ be a cyclic operad with unit. Then the chain complexes 
\begin{eqnarray*}
&&(\kk\ub)_{(-;+)}^\sharp  \otimes_{(\kk\ub)_{(-;+)}} \mathscr{K}_- \otimes_{(\kk\db)_{(-;-)}} \lfb^{*} (\cpd)
\\
&&(\kk \ub)_{(+;-)}^\sharp \otimes_{(\kk \ub)_{(+;-)} } \mathscr{K}_+ \otimes_{\kk \db} \sfb^{*} (\cpd)
\end{eqnarray*}
both have zero homology when evaluated on a non-empty finite set.
\end{thm}

\begin{proof}
We consider the second case; the proof of the first is similar. 

After evaluating on an object $\n$ (considered as an object of $(\kk \ub)_{(+;-)}$)  the underlying graded object of the complex  $(\kk \db)_{(+;-)}^\sharp \otimes_{(\kk \ub)_{(+;-)} } \mathscr{K}_+ \otimes_{\kk \db} \sfb^{*} (\cpd)$ identifies as
$$
(\kk \ub)_{(+;-)}^\sharp (-, \n)  \otimes_{\kk \fb}  \sfb^{*} (\cpd),
$$
with homological grading given by the degree of $\sfb^*$. (Here, $(\kk \ub)_{(+;-)}^\sharp (-, \n) = (\kk \ub)_{(+;-)} ( \n, -) ^\sharp$, which is naturally a right $(\kk \ub)_{(+; -)}$-module, hence a right $\kk \fb$-module, by restriction.)

Fix homological degree $d$, then the above can be written as 
$$
\bigoplus_{n\leq t \in \nat} 
(\kk \ub)_{(+;-)} ( \n, \mathbf{t})^\sharp  \otimes_{\kk \sym_t}  \sfb^d (\cpd) (\mathbf{t}).
$$
(Recall that $ \sfb^d (\cpd) (\mathbf{t})$ is described explicitly by Lemma \ref{lem:sfb_lfb}, which applies since $\cpd (\mathbf{0})=0$.) Henceforth we suppose that $n>0$, which implies that $t>0$ in the above expression; in particular, for this to be non-zero we require that $d>0$.

The result is proved by exhibiting a chain nulhomotopy. This is built out of morphisms of the form:
\begin{eqnarray}
\label{eqn:chain_homotopy}
(\kk \ub)_{(+;-)} ( \n, \mathbf{t})^\sharp  \otimes_{\kk \sym_t}  \sfb^d (\cpd) (\mathbf{t})
\rightarrow 
(\kk \ub)_{(+;-)} ( \n, \mathbf{t+2})^\sharp  \otimes_{\kk \sym_{t+2}}  \sfb^{d+1} (\cpd) (\mathbf{t+2}).
\end{eqnarray}

The basic ingredient of the construction is the following. Choose $x \in \mathbf{t}$ (this is possible, since $t>0$, by the hypothesis $n>0$). Then there is a $\kk$-linear map:
\begin{eqnarray}
\label{eqn:unit_map}
\cpd (\mathbf{t}) 
\rightarrow 
\cpd (\{x, t+1\}) \otimes \cpd (\{t+2\} \amalg \mathbf{t} \backslash \{x \} ) \subset \sfb^2 (\cpd) (\mathbf{t+2}).
\end{eqnarray}
This is defined by using the bijection $\cpd (\mathbf{t}) \stackrel{\cong}{\rightarrow} \cpd (\{t+2\} \amalg \mathbf{t} \backslash \{x \} )$ induced by the bijection $\mathbf{t} \stackrel{\cong}{\rightarrow} \{t+2\} \amalg \mathbf{t} \backslash \{x \}$ defined by sending $x$ to $t+2$, together with the element of $\cpd(\{ x, t+1 \})$ provided by the unit of $\cpd$. This map is $\aut (\mathbf{t} \backslash \{x\})$-equivariant. 

This is used as follows. Fix an element of the dual basis of $(\kk \ub)_{(+;-)} ( \n, \mathbf{t})^\sharp $ using the basis of $(\kk \ub)_{(+;-)} ( \n, \mathbf{t})$ provided by Lemma \ref{lem:basis_twisted_kub};  this is indexed by a map $f$ of $\ub (\n, \mathbf{t})$ (together with the standard ordering of the chords). We take $x:= f(1) \in \mathbf{t}$. Also, set $\tilde{f}$ to be the element of $\ub (\n , \mathbf{t+2})$ obtained by composing with $\mathbf{t} \subset \mathbf{t+2}$.

Now,  we consider $\bigotimes _{i=1}^d \cpd (U_i)$, one of the direct summands of $\sfb^d (\cpd) (\mathbf{t})$, 
 where $\amalg U_i = \mathbf{t}$ corresponds to  an unordered decomposition into non-empty subsets (see Lemma \ref{lem:sfb_lfb}). Without loss of generality we may suppose that $x \in U_1$.
Adapting (\ref{eqn:unit_map}) by replacing $\mathbf{t}$ by $U_1$, we obtain the $\kk$-linear map
$$
\xymatrix{
\bigotimes _{i=1}^d \cpd (U_i)
\ar[r]
\ar@{^(->}[d]
&
\cpd (\{ x, t+1\}) \otimes \cpd (\{t+2 \} \amalg U_1 \backslash \{x\} ) \otimes  \bigotimes _{i=2}^d \cpd (U_i)
\ar@{^(->}[d]
\\
\sfb^d (\cpd) (\mathbf{t}) 
&
\sfb^{d+1} (\cpd) (\mathbf{t+2}).
}
$$
Here, we consider the domain as being associated to the dual basis element corresponding to $f$ and the codomain to the dual basis element corresponding to $\tilde{f}$.

Assembling these, this construction induces the required morphism (\ref{eqn:chain_homotopy}). It remains to check that this defines a chain homotopy as required, using the properties of the unit of a cyclic operad. This is analogous to the proof that the two-sided bar construction for a unital associative algebra $A$ gives a resolution of the algebra $A$. 
 For this, one must pay attention to the orientation signs arising from the order of the `chords', corresponding to the fact that the complex is expressed using $(\kk \ub) _{(+;-)}$. Details are left to the reader.
\end{proof}

\begin{rem}
As stated above, the proof for the first complex is similar. The essential difference is that, in this case the,  orientation signs arise from the direction of chords and the ordering of the factors appearing in $\lfb^* (\cpd)$.    
\end{rem}

\begin{rem}
\label{rem:avoid_acyclicity}
There are standard ways of avoiding the acyclicity that is exhibited by Theorem \ref{thm:acyclicity}:
\begin{enumerate}
\item 
If the cyclic operad $\cpd$ is augmented, then one can replace $\cpd$ by the `augmentation ideal' $\overline{\cpd}$, which has the structure of a cyclic operad without unit.
\item 
In general (in particular, if $\cpd$ is not augmented), one can consider $\cpd_{\geq 3} \subset \cpd$, the subobject supported on sets of cardinality greater than two. This inherits the structure of a cyclic operad without unit.
\end{enumerate}
\end{rem}

\section{Examples from $\kk$-algebras with involution}
\label{sect:examples_alg_inv}

Recall from Section \ref{subsect:alg_inv} that, if $(B, \sigma)$ is a $\kk$-algebra with involution, we have the corresponding non-unital cyclic operad $\cpd_{(B,\sigma)}$ and thus the $\kk \db$-module $\sfb^*(\cpd_{(B,\sigma)})$ and the $(\kk\db)_{(-;-)}$-module $\lfb^* (\cpd_{(B,\sigma)})$.  Here we  focus upon the second case and the two associated Koszul complexes 
\begin{eqnarray}
\label{eqn:K-Bsigma}
&&\mathscr{K}_- \otimes_{(\kk \db)_{(-;-)}} \lfb^* (\cpd_{(B, \sigma)}) 
\\
\label{eqn:sharp_K-Bsigma}
&&
(\kk\ub)_{(-;+)}^\sharp\otimes_{(\kk \ub)_{(-;+)}} \mathscr{K}_-\otimes_{(\kk \db)_{(-;-)}}\lfb^* (\cpd_{(B, \sigma)}) .
\end{eqnarray}

\subsection{The first complex}

The underlying object of (\ref{eqn:K-Bsigma}) is isomorphic to 
\begin{eqnarray}
\label{eqn:underlying_K-Bsigma}
(\kk \ub)_{(-;+)} \otimes _{\kk \fb} \lfb^* (\cpd_{(B, \sigma)})
\end{eqnarray}
and the underlying $\kk \fb$-module of $\lfb^* (\cpd_{(B, \sigma)})$ was identified in Lemma \ref{lem:sfb_lfb_algebra_inv}.

\begin{lem}
\label{lem:underlying_K-Bsigma}
For $(B, \sigma)$ a $\kk$-algebra with involution, 
the underlying $\kk \fb$-module of $(\kk \ub)_{(-;+)} \otimes _{\kk \fb} \lfb^* (\cpd_{(B, \sigma)})$ is supported on sets of even parity. 

For $n \in \nat$, there is an isomorphism of $\kk \sym_{2n}$-modules
\[
\big(
(\kk \ub)_{(-;+)} \otimes _{\kk \fb} \lfb^* (\cpd_{(B, \sigma)})
\big) 
(\mathbf{2n}) 
\cong 
\bigoplus_{t \leq n} 
(\kk \ub)_{(-;+)}(\mathbf{2t}, \mathbf{2n})  \otimes _{\kk \sym_2 \wr \sym_t} (B^{\otimes t} \otimes \kk_{(+;-)}^{[t]}),
\]
where the term indexed by $t$ is placed in cohomological degree $n-t$. 
The right hand side is isomorphic to 
$$
\bigoplus_{t \leq n} 
\kk \sym_{2n}  \otimes _{\kk (\sym_2 \wr \sym_t\  \times \ \sym_2\wr\sym_{n-t} )} \big( (B^{\otimes t} \otimes \kk_{(+;-)}^{[t]})\boxtimes \kk^{[n-t]}_{(-;+)}\big),
$$
where we use the bijection $\mathbf{2n} \cong  \mathbf{2t} \amalg \mathbf{2 (n-t)}$ to identify $\sym_2 \wr \sym_t\  \times \ \sym_2\wr\sym_{n-t}$ as a subgroup of $\sym_{2n}$; 
 $\kk_{(-;+)}^{[n-t]}$ indicates the representation $\kk_{(-;+)}$ of $\sym_2\wr\sym_{n-t}$.
\end{lem}

\begin{proof}
By Lemma \ref{lem:sfb_lfb_algebra_inv}, the underlying $\kk \fb$-module of $\lfb^* (\cpd_{(B, \sigma)})$ is supported on sets of even parity, hence the same is true of (\ref{eqn:underlying_K-Bsigma}). The first identification  follows directly from Lemma \ref{lem:sfb_lfb_algebra_inv}; the second then follows from the definition of $(\kk \ub)_{(-;+)}$.
\end{proof}

Using this Lemma, the differential from cohomological degree $n-t$ to $n-t+1$ is of the form:
$$
(\kk \ub)_{(-;+)}(\mathbf{2t}, \mathbf{2n})  \otimes _{\kk \sym_2 \wr \sym_t } (B^{\otimes t} \otimes \kk_{(+;-)}^{[t]})
\rightarrow 
(\kk \ub)_{(-;+)}(\mathbf{2(t-1)}, \mathbf{2n})  \otimes _{\kk \sym_2 \wr  \sym_{t-1} } (B^{\otimes t-1} \otimes \kk_{(+;-)}^{[t-1]})
$$
This can be rewritten as the $\kk \sym_n$-module morphism:
$$
\kk \sym_{2n}  \otimes _{\kk (\sym_2 \wr \sym_{t}\  \times \ \sym_2\wr\sym_{n-t} )} \big( (B^{\otimes t} \otimes \kk_{(+;-)}^{[t]})\boxtimes \kk^{[n-t]}_{(-;+)}\big)
\rightarrow 
\kk \sym_{2n}  \otimes _{\kk (\sym_2 \wr \sym_{t-1}\ \times \ \sym_2\wr\sym_{n-t+1} )} \big( (B^{\otimes t-1} \otimes \kk_{(+;-)}^{[t-1]})\boxtimes \kk^{[n-t+1]}_{(-;+)}\big).
$$

\begin{rem}
\label{rem:identify_first_complex_B_sigma}
In cohomological  degree $n-t$, the term of the complex (\ref{eqn:K-Bsigma}) evaluated on $\mathbf{2n}$ is 
$$\kk \sym_{2n} \otimes _{\kk (\sym_2 \wr \sym_t\  \times \ \sym_2\wr\sym_{n-t} )} \big( (B^{\otimes t} \otimes \kk_{(+;-)}^{[t]})\boxtimes \kk_{(-;+)}^{[n-t]}\big).
$$

This makes it clear how elements can be represented by generalized decorated chord diagrams on the set $\mathbf{2n}$, modulo the appropriate relations. There are two types of chords: 
\begin{enumerate}
\item 
$t$ chords decorated by elements of $B$; there is an orientation sign associated to the order of the chords; the $\sigma$-chord orientation relation applies to these; 
\item 
$n-t$ unlabelled directed chords, with an orientation sign associated to the direction of each chord (but no orientation sign associated to the order of the chords). 
\end{enumerate}
The differential decreases the number of $B$-decorated chords by one and increases the number of unlabelled directed chords by one. The differential is best illustrated by an example, as below. 
\end{rem}

\begin{exam}
\label{exam:chord_complex}
For $n=3$ and cohomological degree $1$ (i.e., $t=2$), the following is an example of such a generalized decorated chord diagram:

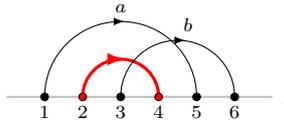
\begin{center}
\begin{tikzpicture}[scale = .5]

\begin{scope}
    \clip (0,-7) rectangle (7,-4.5);
    \draw (3,-7) circle(2);
    \draw [very thick, red] (3,-7) circle (1); 
    \draw (4.5, -7) circle (1.5);
\end{scope}

\draw [very thick, -latex, red](3,-6) -- (3.1,-6);
\draw [-latex](3,-5) -- (3.1,-5);

\draw [-latex] (4.6,-5.5) -- (4.7,-5.5);

\draw [lightgray, thick] (0,-7) -- (7,-7);
\draw [fill=black] (1,-7) circle (0.1);
\node [below] at (1,-7) {$\scriptstyle{1}$};
\draw [fill=red] (2,-7) circle (0.1);
\node [below] at (2,-7) {$\scriptstyle{2}$};
\draw [fill=black] (3,-7) circle (0.1);
\node [below] at (3,-7) {$\scriptstyle{3}$};
\draw [fill=red] (4,-7) circle (0.1);
\node [below] at (4,-7) {$\scriptstyle{4}$};
\draw [fill=black] (5,-7) circle (0.1);
\node [below] at (5,-7) {$\scriptstyle{5}$};
\draw [fill=black] (6,-7) circle (0.1);
\node [below] at (6,-7) {$\scriptstyle{6}$};

\node [above] at (3,-5) {$\scriptstyle{a}$};
\node [above right] at (4.4, -5.5) {$\scriptstyle{b}$};
\node [right] at (7,-7) {.};
\end{tikzpicture}
\end{center} 
The black chords are labelled by elements of $B$ and are ordered so that the chord labelled $a$ precedes that labelled by $b$ (recall that an orientation sign is associated to the ordering). The thick red chord is directed (an orientation sign is associated to this datum). The diagram distinguishes the subset $\{1, 3, 5, 6\} \subset \mathbf{6}$ of black nodes, which correspond to the `image' of $\mathbf{2t}$ in $\mathbf{2n}$.

In this case, the differential is given by the sum of the operation described below associated to ordered pairs $u<v$ with $u, v \in \{1, 3, 5, 6\}$. Thus there are six possibilities, namely the ordered pairs $(1,3)$, $(1,5)$, $(1,6)$, $(3,5)$, $(3,6)$, $(5,6)$. For a given $(u,v)$, if these are the endpoints of a single (black) chord, the corresponding operation is zero.  (Here this applies to the ordered pairs $(1,5)$ and $(3,6)$, leaving the four ordered pairs $(1,3)$, $(1,6)$, $(3,5)$, $(5,6)$ to consider.)
 Otherwise one applies the following steps:
\begin{enumerate}
\item 
reorder the black chords so that $v$ appears as a node of the last chord and $u$ as a node of the penultimate chord; this potentially introduces an orientation sign; 
\item 
use the $\sigma$-chord orientation relation to ensure that $u$ is the `exit' node of its chord and $v$ the `entry' node of its chord (using the specified direction of the chord); 
\item 
replace these two black chords by the chord corresponding to the product operation in $B$; 
\item 
add a new red chord from $u$ to $v$.
\end{enumerate}

For example, taking $(u,v) = (1,3)$, the two chords are already ordered correctly. However, we must apply the $\sigma$-chord orientation relation
 to ensure that $1$ is the exit node, thus replacing the diagram by 
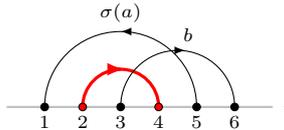
\begin{center}
\begin{tikzpicture}[scale = .5]

\begin{scope}
    \clip (0,-7) rectangle (7,-4.5);
    \draw (3,-7) circle(2);
    \draw [very thick, red] (3,-7) circle (1); 
    \draw (4.5, -7) circle (1.5);
\end{scope}

\draw [very thick, -latex, red](3,-6) -- (3.1,-6);
\draw [-latex](3.1,-5) -- (3,-5);

\draw [-latex] (4.6,-5.5) -- (4.7,-5.5);

\draw [lightgray, thick] (0,-7) -- (7,-7);
\draw [fill=black] (1,-7) circle (0.1);
\node [below] at (1,-7) {$\scriptstyle{1}$};
\draw [fill=red] (2,-7) circle (0.1);
\node [below] at (2,-7) {$\scriptstyle{2}$};
\draw [fill=black] (3,-7) circle (0.1);
\node [below] at (3,-7) {$\scriptstyle{3}$};
\draw [fill=red] (4,-7) circle (0.1);
\node [below] at (4,-7) {$\scriptstyle{4}$};
\draw [fill=black] (5,-7) circle (0.1);
\node [below] at (5,-7) {$\scriptstyle{5}$};
\draw [fill=black] (6,-7) circle (0.1);
\node [below] at (6,-7) {$\scriptstyle{6}$};

\node [above] at (3,-5) {$\scriptstyle{\sigma(a)}$};
\node [above right] at (4.4, -5.5) {$\scriptstyle{b}$};
\node [right] at (7,-7) {.};
\end{tikzpicture}
\end{center} 

\noindent
Then, multiplying and adding the new red chord gives:
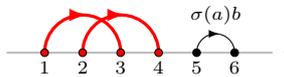
\begin{center}
\begin{tikzpicture}[scale = .5]

\begin{scope}
    \clip (0,-7) rectangle (7,-4.5);
    \draw [very thick, red] (2,-7) circle (1); 
    \draw [very thick, red] (3,-7) circle (1); 
    \draw (5.5, -7) circle (.5);
\end{scope}

\draw [very thick, -latex, red](3,-6) -- (3.1,-6);
\draw [very thick, -latex, red](2,-6) -- (2.1,-6);

\draw [-latex] (5.5,-6.5) -- (5.6,-6.5);

\draw [lightgray, thick] (0,-7) -- (7,-7);
\draw [fill=red] (1,-7) circle (0.1);
\node [below] at (1,-7) {$\scriptstyle{1}$};
\draw [fill=red] (2,-7) circle (0.1);
\node [below] at (2,-7) {$\scriptstyle{2}$};
\draw [fill=red] (3,-7) circle (0.1);
\node [below] at (3,-7) {$\scriptstyle{3}$};
\draw [fill=red] (4,-7) circle (0.1);
\node [below] at (4,-7) {$\scriptstyle{4}$};
\draw [fill=black] (5,-7) circle (0.1);
\node [below] at (5,-7) {$\scriptstyle{5}$};
\draw [fill=black] (6,-7) circle (0.1);
\node [below] at (6,-7) {$\scriptstyle{6}$};

\node [above] at (5.5,-6.5) {$\scriptstyle{\sigma(a)b}$};
\node [right] at (7,-7) {.};
\end{tikzpicture}
\end{center} 

As a second example, consider the case $(3,5)$. In this case, there will be an orientation sign $-1$ resulting from reordering the two black chords; moreover, we require to apply the $\sigma$-chord orientation relation to both black chords, resulting in 

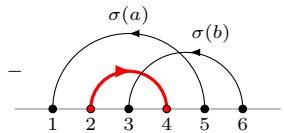
\begin{center}
\begin{tikzpicture}[scale = .5]

\begin{scope}
    \clip (0,-7) rectangle (7,-4.5);
    \draw (3,-7) circle(2);
    \draw [very thick, red] (3,-7) circle (1); 
    \draw (4.5, -7) circle (1.5);
\end{scope}

\draw [very thick, -latex, red](3,-6) -- (3.1,-6);
\draw [-latex](3.1,-5) -- (3,-5);

\node at (0,-6) {$\scriptstyle{-}$}; 

\draw [-latex] (4.7,-5.5) -- (4.6,-5.5);

\draw [lightgray, thick] (0,-7) -- (7,-7);
\draw [fill=black] (1,-7) circle (0.1);
\node [below] at (1,-7) {$\scriptstyle{1}$};
\draw [fill=red] (2,-7) circle (0.1);
\node [below] at (2,-7) {$\scriptstyle{2}$};
\draw [fill=black] (3,-7) circle (0.1);
\node [below] at (3,-7) {$\scriptstyle{3}$};
\draw [fill=red] (4,-7) circle (0.1);
\node [below] at (4,-7) {$\scriptstyle{4}$};
\draw [fill=black] (5,-7) circle (0.1);
\node [below] at (5,-7) {$\scriptstyle{5}$};
\draw [fill=black] (6,-7) circle (0.1);
\node [below] at (6,-7) {$\scriptstyle{6}$};

\node [above] at (3,-5) {$\scriptstyle{\sigma(a)}$};
\node [above right] at (4.4, -5.5) {$\scriptstyle{\sigma (b)}$};
\node [right] at (7,-7) {.};
\end{tikzpicture}
\end{center} 

\noindent
Then, multiplying and adding the new red chord gives the following:
\begin{center}
\begin{tikzpicture}[scale = .5]

\begin{scope}
    \clip (0,-7) rectangle (7,-4.5);
    \draw (3.5,-7) circle(2.5);
    \draw [very thick, red] (3,-7) circle (1); 
    \draw [very thick, red] (4,-7) circle (1); 
\end{scope}

\draw [very thick, -latex, red](3,-6) -- (3.1,-6);
\draw [very thick, -latex, red](4,-6) -- (4.1,-6);
\draw [-latex](3.5,-4.5) -- (3.4,-4.5);

\node at (0,-6) {$\scriptstyle{-}$}; 

\draw [lightgray, thick] (0,-7) -- (7,-7);
\draw [fill=black] (1,-7) circle (0.1);
\node [below] at (1,-7) {$\scriptstyle{1}$};
\draw [fill=red] (2,-7) circle (0.1);
\node [below] at (2,-7) {$\scriptstyle{2}$};
\draw [fill=red] (3,-7) circle (0.1);
\node [below] at (3,-7) {$\scriptstyle{3}$};
\draw [fill=red] (4,-7) circle (0.1);
\node [below] at (4,-7) {$\scriptstyle{4}$};
\draw [fill=red] (5,-7) circle (0.1);
\node [below] at (5,-7) {$\scriptstyle{5}$};
\draw [fill=black] (6,-7) circle (0.1);
\node [below] at (6,-7) {$\scriptstyle{6}$};

\node [above] at (3.5,-4.5) {$\scriptstyle{\sigma(b)\sigma(a)}$};
\node [right] at (7,-7) {.};
\end{tikzpicture}
\end{center} 

\noindent
Finally, one can  apply the $\sigma$-chord orientation relation to recover the following:
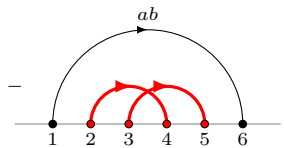
\begin{center}
\begin{tikzpicture}[scale = .5]

\begin{scope}
    \clip (0,-7) rectangle (7,-4.5);
    \draw (3.5,-7) circle(2.5);
    \draw [very thick, red] (3,-7) circle (1); 
    \draw [very thick, red] (4,-7) circle (1); 
\end{scope}

\draw [very thick, -latex, red](3,-6) -- (3.1,-6);
\draw [very thick, -latex, red](4,-6) -- (4.1,-6);
\draw [-latex](3.4,-4.5) -- (3.5,-4.5);

\node at (0,-6) {$\scriptstyle{-}$};

\draw [lightgray, thick] (0,-7) -- (7,-7);
\draw [fill=black] (1,-7) circle (0.1);
\node [below] at (1,-7) {$\scriptstyle{1}$};
\draw [fill=red] (2,-7) circle (0.1);
\node [below] at (2,-7) {$\scriptstyle{2}$};
\draw [fill=red] (3,-7) circle (0.1);
\node [below] at (3,-7) {$\scriptstyle{3}$};
\draw [fill=red] (4,-7) circle (0.1);
\node [below] at (4,-7) {$\scriptstyle{4}$};
\draw [fill=red] (5,-7) circle (0.1);
\node [below] at (5,-7) {$\scriptstyle{5}$};
\draw [fill=black] (6,-7) circle (0.1);
\node [below] at (6,-7) {$\scriptstyle{6}$};

\node [above] at (3.5,-4.5) {$\scriptstyle{ab}$};
\node [right] at (7,-7) {.};
\end{tikzpicture}
\end{center} 

This can also  be seen in a different manner: we could have changed the order of $(3,5)$, which means that the new red chord will have the opposite direction (which corresponds to the orientation sign). Then, the black chords are already in the correct order and their directions are compatible, whence one obtains the same result upon multiplying.
\end{exam}

\begin{rem}
The description of the differential given in Example \ref{exam:chord_complex} simplifies when $(B, \sigma)$ satisfies $\sigma = \id$ (this implies that $B$ is commutative). In this case, there is no need to apply the $\sigma$-chord orientation relation (so that we may forget the orientation of the $B$-labelled chords). However, we do still require to keep track of the order of the black chords.
\end{rem}

\begin{exam}
\label{exam:initial_example}
For the case of the initial unital $\kk$-algebra with involution, $(\kk , \id)$, the complex $(\kk \ub)_{(-;+)} \otimes _{\kk \fb} \lfb^* (\cpd_{(\kk, \id)})$ is the precursor of the Chevalley-Eilenberg complex of the symplectic Lie algebras $\mathfrak{sp} (V)$ (for $(V, \omega)$ a symplectic vector space):
\[
(\Lambda^* (\mathfrak{sp} (V)), d_\mathsf{CE}).
\] 
Heuristically, the idea is that each black chord corresponds to a copy of $\mathfrak{sp}(V)$; the red chords are `phantom';  they are `contracted away' by using the symplectic form $\omega$.
\end{exam}

\subsection{The first complex for $(\kk,\id)$}
\label{subsect:non-torsion_Ext}

In this section we expand upon Example \ref{exam:initial_example}. The main aim is to establish Theorem  \ref{thm:non-torsion}, which shows that the torsion-free part of $\ext^* _{(\kk \db)_{(-;-)} }(\kk \fb, \lfb^* (\cpd_{(\kk, \id)}) )$ is  non-trivial.

We exploit the Schur correspondence between $\kk \fb$-modules and functors on $\kk$-vector spaces. Thus, the decomposition of a $\kk \fb$-module as a direct sum of simples corresponds to the decomposition of the associated Schur functor as a direct sum of Schur functors of the form $\schur_\lambda (-)$ (the Schur functor associated to the simple $\kk \sym_d$-module $S_\lambda$ indexed by a partition $\lambda \vdash d$). In this context,  Schur's lemma gives for any two partitions $\lambda$, $\mu$
$$
\hom (\schur_\lambda(-), \schur_\mu(-)) 
=
\left\{ 
\begin{array}{ll}
0 & \lambda \neq \mu \\
\kk & \lambda = \mu,
\end{array}
\right.
$$
where the left hand side indicates natural transformations between Schur functors. (Our indexing convention is such that $S_{(d)}$ corresponds to $\triv_d$, the trivial representation.)

When writing functors on vector spaces below, $F(W)$ will  be used as shorthand for the functor $F: W \mapsto F(W)$. Recall that a  Schur functor of the form $W^{\otimes d} \otimes_{\kk \sym_d} M$, for a $\kk \sym_d$-module $M$, is said to have (homogeneous) polynomial degree $d$.

Using the identification given in Remark \ref{rem:identify_first_complex_B_sigma}, we have the following:

\begin{lem}
\label{lem:underlying_schur_functor}
The Schur functor associated to the underlying $(\kk \ub)_{(-;+)} \otimes _{\kk \fb} \lfb^* (\cpd_{(\kk, \id)})$ is isomorphic to 
\[
W \mapsto \Lambda^* (S^2 (W)) \otimes S^* (\Lambda^2 (W)).
\]
Here the cohomological degree is given by the degree in the symmetric power term $S^* (\Lambda^2 (W))$. 
\end{lem}

\begin{proof}
The underlying $\kk \fb$-module of $(\kk \ub)_{(-;+)} \otimes _{\kk \fb} \lfb^* (\cpd_{(\kk, \id)})$ is supported on sets of even parity, so we restrict to these. Namely, it suffices to consider the even polynomial degree part of the associated Schur functor.

By Remark \ref{rem:identify_first_complex_B_sigma}, the polynomial degree $2n$ part is given in cohomological degree $n-t$ by 
$$
W^{\otimes 2n} \otimes_{\sym_{2n}} \kk \sym_{2n}\otimes _{\kk (\sym_2 \wr \sym_t\  \times \ \sym_2\wr\sym_{n-t} )} \big( \kk_{(+;-)}^{[t]} \boxtimes \kk_{(-;+)}^{[n-t]}\big)
\cong 
W^{\otimes 2n} \otimes _{\kk (\sym_2 \wr \sym_t\  \times \ \sym_2\wr\sym_{n-t} )} \big( \kk_{(+;-)}^{[t]} \boxtimes \kk_{(-;+)}^{[n-t]}\big),
$$
where $\kk_{(+;-)}^{[t]}$ denotes $\sgn_t$ considered as a $\sym_2 \wr \sym_t$-module. This  is isomorphic to 
$$
\big(
W^{\otimes 2t} \otimes _{\kk \sym_2 \wr \sym_t} \kk_{(+;-)}^{[t]}
\big) 
\ \otimes \ 
\big(
W^{\otimes 2(n-t)} \otimes_{\kk\sym_2\wr\sym_{n-t} }  \kk_{(-;+)}^{[n-t]}
\big) .
$$
By standard results on plethysm, the first term is isomorphic to $\Lambda^t (S^2 (W)) $ and the second to $S^{n-t} (\Lambda^2 (W))$. The result follows.
\end{proof}

\begin{rem}
\label{rem:identify_plethysms}
For $n \in \nat$, the Schur functors $\Lambda^n (S^2 (W))$ and $S^n (\Lambda^2 (W))$  can be decomposed  as follows:
\begin{eqnarray}
S^n (\Lambda^2 W) &\cong & \bigoplus_{\substack{\lambda \vdash 2n \\ \lambda'_i \equiv 0 \mod (2) } } \schur_\lambda (W)
\\
\Lambda^n (S^2 W) & \cong & \bigoplus_{\substack{\lambda \in Q_1(2n)}} \schur _\lambda (W).
\end{eqnarray}
Here, $\lambda'$ is the conjugate partition of $\lambda$;   $Q_1(2n)$ is the set of partitions of $2n$ such that $\lambda_i -i = \lambda'_i- i+1$, whenever $\lambda'_i \geq i$.
(These results are given in \cite[Propositions 2.3.8 and 2.3.9]{MR1988690} for example (note that Weyman indexes by the conjugate partition) and also in \cite{MR3443860}.) 

In particular, these functors are multiplicity free: for any partition $\lambda \vdash 2n$, the Schur functor $\schur_\lambda$ occurs with multiplicity at most one.
\end{rem}

In polynomial degree $2n$, the complex at the level of Schur functors has the form
\begin{eqnarray}
\label{eqn:complex_schur_functors}
\Lambda^n (S^2 (W) ) 
\rightarrow 
\Lambda^{n-1} (S^2 (W))
\otimes 
\Lambda^2 (W) 
\rightarrow 
\Lambda^{n-2} (S^2 (W))
\otimes 
S^2 (\Lambda^2 (W))
\rightarrow 
\ldots
\rightarrow 
S^n (\Lambda^2 (W)),
\end{eqnarray}
concentrated in cohomological degrees $[0, n]$. Moreover, the final differential in the complex is always zero.

\begin{rem}
\label{rem:differential}
The differential can be described explicitly  as follows. The fundamental ingredient is  the case $n=2$. This is given by the composite
$$
\Lambda^2 (S^2 (W) ) 
\hookrightarrow 
S^2 (W) \otimes S^2 (W) 
\hookrightarrow 
W^{\otimes 4} 
\twoheadrightarrow 
W \otimes \Lambda^2 (W) \otimes W \cong W \otimes W \otimes \Lambda^2 (W) 
\twoheadrightarrow 
S^2(W) \otimes \Lambda ^2 (W), 
$$
where the codomain is understood to be  $\Lambda^1 (S^2 (W)) \otimes S^1 (\Lambda^2 (W))$;  the first map  is induced by the natural inclusion $\Lambda^2 (V) \hookrightarrow V^{\otimes 2}$ (taking $V= S^2 (W)$); the second is induced by the natural inclusion $S^2 (W) \hookrightarrow W^{\otimes 2}$; the third map is given by applying the projection $W^{\otimes 2} \twoheadrightarrow \Lambda^2 (W)$ to the middle two factors and the subsequent isomorphism reorders  the tensor factors; the final map uses the projection $W^{\otimes 2} \twoheadrightarrow S^2 (W)$.

Decomposing these functors as Schur functors, this identifies as a natural transformation 
$$
\schur_{(3,1) } (W) \rightarrow 
\schur_{(3, 1)} (W) \oplus \schur_{(2,1,1)} (W),
$$
using the identification recalled in Remark \ref{rem:identify_plethysms} for the domain and the Pieri formula for the codomain. Wince it is easily seen to be non-zero, by Schur's lemma, it is the unique such natural transformation,  up to non-zero scalar multiple. 

Using this, the differential $\Lambda^{n-t}  (S^2 (W) ) \otimes S^t (\Lambda^2(W)) 
\rightarrow 
\Lambda^{n-t-1}  (S^2 (W) ) \otimes S^{t+1} (\Lambda^2(W))$ can be described as follows. It is zero if $n-t <2$ and, otherwise, it is the composite:
\begin{eqnarray*}
\Lambda^{n-t}  (S^2 (W) ) \otimes S^t (\Lambda^2(W)) 
\rightarrow 
\Lambda^{n-t-2}  (S^2 (W) )\otimes \Lambda^2 (S^2 (W)) \otimes S^t (\Lambda^2(W)) 
\rightarrow 
\\
\Lambda^{n-t-2}  (S^2 (W) )\otimes S^2 (W)\otimes \Lambda^2 (W)  \otimes S^t (\Lambda^2(W)) 
\rightarrow 
\Lambda^{n-t-1}  (S^2 (W) ) \otimes S^{t+1} (\Lambda^2(W)),
\end{eqnarray*}
in which the first map is induced by the coproduct $\Lambda^{n-t} \rightarrow \Lambda^{n-t-2} \otimes \Lambda^2$, the second is the morphism $\Lambda^2 (S^2 (W))\rightarrow S^2(W) \otimes \Lambda^2(W) $ described above, and the last is induced by the products $\Lambda^{n-t-2} \otimes \Lambda^1 \rightarrow \Lambda^{n-t-1}$ and $S^1 \otimes S^t \rightarrow S^{t+1}$.
\end{rem}

\begin{rem}
In \cite[Section 4.3]{MR3376738}, Sam and Snowden give a model for  $\mathrm{Rep} (\mathbf{Sp})$ in terms of modules over $\mathrm{Sym}(\Lambda^2)$. From our point of view, using the Schur correspondence, this is equivalent to the fact that the category of $(\kk \ub)_{(-;+)} $-modules is equivalent to the category of $S^* (\Lambda^2 (W))$-modules (functorially with respect to $W$). 
\end{rem}

The following Proposition corresponds to the fact that the complex $(\kk \ub)_{(-;+)} \otimes _{\kk \fb} \lfb^* (\cpd_{(\kk, \id)})$ is a complex of $(\kk \ub)_{(-;+)}$-modules. 

\begin{prop}
\label{prop:S_Lambda2_module_complex}
The complex $\Lambda^* (S^2 (W)) \otimes S^* (\Lambda^2 (W))$  is a complex of $S^* (\Lambda^2 (W))$-modules, naturally with respect to $W$.
\end{prop}

\begin{exam}
\label{exam:cohomology}
\ 
\begin{enumerate}
\item 
For $n=1$, the complex reduces to $S^2(W) \stackrel{0}{\rightarrow }\Lambda^2 (W)$; in particular the cohomology in degree zero is isomorphic to $S^2 (W) = \schur_{(2)} (W)$.
\item 
For $n=2$, the complex reduces to  $\Lambda^2 (S^2 (W)) \rightarrow S^2 (W) \otimes \Lambda^2 (W) \rightarrow S^2 (\Lambda^2 (W))$, which identifies as 
\[
\schur_{(3,1)} (W)
\hookrightarrow 
\schur_{(3, 1)} (W) \oplus \schur_{(2,1,1)} (W)
\stackrel{0}{\rightarrow}
\schur_{(1^4)} (W) \oplus \schur _{(2, 2) } (W).
\]
In particular, the cohomology in degree zero is zero.
\item 
For $n=3$, the complex reduces to 
\begin{eqnarray*}
\schur_{(3,3)} (W) \oplus \schur_{(4,1,1)} (W)
\rightarrow 
\schur _{(3,1,1,1)} (W) \oplus \schur_{(3,2,1)} (W) \oplus \schur_{(4,1,1)} (W) \oplus \schur _{(4, 2) } (W)
\rightarrow 
\ldots 
\\
\schur_{(2, 1^4)} (W) \oplus \schur_{(2^3)}(W) \oplus \schur_{(3,1^3)} (W) \oplus \schur _{(3,2,1)} (W) \oplus \schur_{(4,2)} (W)
\stackrel{0}{\rightarrow}
\schur_{(1^6)}(W) \oplus \schur_{(2^2, 1^2)} (W) \oplus \schur_{(3,3)}(W).
\end{eqnarray*}
In particular, the cohomology in degree zero is isomorphic to $\schur_{(3,3)} (W)$.

\end{enumerate}
\end{exam}

Extending these examples, we have the following non-triviality result:

\begin{prop}
\label{prop:non_trivial_cohomology}
For $2 \leq \ell \in \nat$, the cohomology of (\ref{eqn:complex_schur_functors}) for $2n = \ell (\ell -1)$ in cohomological degree zero contains 
$\schur_{(\ell^{\ell-1})}(W)$. Hence there is an inclusion of $\kk \sym_{2n}$-modules for $2n = \ell (\ell -1)$:
$$
S_{(\ell^{\ell-1})} \subseteq \ext^0_{(\kk \db)_{(-;-)} }(\kk \fb, \lfb^* (\cpd_{(\kk, \id)}) ) (\mathbf{2n})
.
$$
\end{prop}

\begin{proof}
The final statement follows from the first by the Schur correspondence, so it suffices to establish the first.

One checks that $\schur_{(\ell^{\ell-1})}(W)$ occurs in $\Lambda^n (S^2 (W))$ for this value of $n$ using the identification recalled in Remark \ref{rem:identify_plethysms}. This follows since the partition $(\ell^{\ell -1}) $ lies in $Q_1 (2n)$, as can easily be checked by induction on $\ell$. 

To conclude it suffices to show that $\schur_{(\ell^{\ell-1})}(W)$ does not occur in $\Lambda^{n-1} (S^2 (W))\otimes \Lambda^2 (W)$. The case $n=2$ is treated in the example above. 
For $n>2$, one uses the Pieri formula and the identification of $Q_1 (2(n-1))$. Namely, by inspection, the only partition $\mu \in Q_1 (2(n-1))$ with $\mu_1 \leq \ell$ is the partition $(\ell^{\ell -2}, \ell -2)$. The Pieri formula shows that $\schur_{(\ell^{\ell -2}, \ell -2)} (W) \otimes \Lambda^2 (W)$ does not contain a summand $\schur_{(\ell^{\ell -1})}(W)$.
\end{proof}

An interesting phenomenon is exhibited by Example \ref{exam:cohomology}.  Namely, for $n=1$ we have the identification of the cohomology in degree zero as $\schur_{(2)} (W)$ (as in Proposition \ref{prop:non_trivial_cohomology}). This propagates using the $S^* (\Lambda^2 (W))$-module structure to give the following cohomology:
\begin{enumerate}
\item 
for $n=2$, $\schur_{(2,1,1)} (W)$ in cohomological degree $1$;
\item 
for $n=3$,  $\schur_{(2,1^4)}(W)$, in cohomological degree $2$.
\end{enumerate}

This pattern extends: 

\begin{thm}
\label{thm:non-torsion}
For $2 \leq \ell \in \nat$ and $2n := \ell (\ell -1)$, the  cohomology given by the $\kk \sym_{2n}$-module
$$
S_{(\ell^{\ell -1})} \subseteq 
\ext^0_{(\kk \db)_{(-;-)} }(\kk \fb, \lfb^* (\cpd_{(\kk, \id)}) )(\mathbf{2n})
$$
is not torsion with respect to the  $(\kk \ub)_{(-;+)}$-module structure. 
\end{thm}

\begin{proof}
We work at the level of the associated Schur functors and the $S^* (\Lambda^2 (W))$-module structure, which corresponds to the $(\kk \ub)_{(-;+)}$-module structure (as explained above). Recall that, by Proposition \ref{prop:S_Lambda2_module_complex}, the complex $\Lambda^* (S^2 (W)) \otimes S^* (\Lambda^2 (W))$  is a complex of $S^* (\Lambda^2 (W))$-modules.

Since $\Lambda^n (S^2 (W))$ is multiplicity free, it contains a unique copy of $\schur_{(\ell^{\ell -1})}(W)$ and this lies in the cocycles, by Proposition \ref{prop:non_trivial_cohomology}. 

Take $d \in \nat$. Using the $S^* (\Lambda^2 (W))$-module structure on the complex yields the following subobject of the cohomological degree $d$ cochains
\[
\schur_{(\ell^{\ell -1})}(W) \otimes S^d (\Lambda^2 (W)) \subseteq \Lambda^n (S^2 (W)) \otimes S^d (\Lambda^2 (W)) 
\]
and this subobject lies in the cocycles. 

Now, $S^d (\Lambda^2 (W))$ contains a direct summand $\Lambda^{2d} (W)$ and, by the Pieri formula, $\schur_{(\ell^{\ell -1})}(W) \otimes \Lambda^{2d} (W)$ contains a (unique) direct summand $\schur_{(\ell^{\ell -1}, 1^{2d})} (W)$. (In fact, $\schur_{(\ell^{\ell -1}, 1^{2d})} (W)$ occurs with multiplicity one in $\Lambda^n (S^2 (W)) \otimes S^d (\Lambda^2 (W)) $; this follows from the fact that the only partition $\lambda$ in $Q_1(2n)$ with $\lambda_1= \ell $ is $\lambda  =(\ell ^{\ell -1})$,  together with elementary properties of the Littlewood-Richardson rule.) 

We claim  that $\schur_{(\ell^{\ell -1}, 1^{2d})} (W)$ does not lie in the coboundaries. The  cochains in cohomological degree $d-1$ and of the correct polynomial degree are given by  
$$
 \Lambda^{n+1} (S^2 (W)) \otimes S^{d-1} (\Lambda^2 (W)). 
$$
Now, $\Lambda^{n+1} (S^2 (W))$ is identified as in Remark \ref{rem:identify_plethysms} as the direct sum of Schur functors indexed by $Q_1 (2(n+1))$. One checks by inspection that, for $2n=\ell (\ell -1)$,   all partitions $\mu \in Q_1 (2(n+1))$ satisfy $\mu_1 >\ell$. It follows, again  from elementary properties of the Littlewood-Richardson rule, 
 that $\schur_{(\ell^{\ell -1}, 1^{2d})} (W)$ does not occur as a direct summand in these cochains. This gives the claimed result.

This implies that, for the given $n$  and $d$, the $(\kk \ub) _{(-;+)}$-module action map
\[ 
(\kk \ub) _{(-;+)} (\mathbf{2n} , \mathbf{2(n+d)}) 
\otimes 
\ext^0_{(\kk \db)_{(-;-)} }(\kk \fb, \lfb^* (\cpd_{(\kk, \id)}) )(\mathbf{2n})
\rightarrow 
\ext^0_{(\kk \db)_{(-;-)} }(\kk \fb, \lfb^* (\cpd_{(\kk, \id)}) )(\mathbf{2(n+d)})
\]
restricted to $S_{(\ell^{\ell -1})}$ is non-trivial. By the definition of torsion (see Section \ref{sect:torsion}), it follows that all the elements in $S_{(\ell^{\ell -1})}$ are non-torsion.
\end{proof}

\subsection{The second complex}
The complex (\ref{eqn:sharp_K-Bsigma}) should seem more familiar; it is related to the complex calculating dihedral homology in characteristic zero.

The underlying object of (\ref{eqn:sharp_K-Bsigma}) is isomorphic to 

\begin{eqnarray}
\label{eqn:sharp_underlying_K-Bsigma}
(\kk \ub)_{(-;+)}^\sharp \otimes _{\kk \fb} \lfb^* (\cpd_{(B, \sigma)}).
\end{eqnarray}

\begin{lem}
\label{lem:B_sigma_second_complex}
For $(B, \sigma)$ a $\kk$-algebra with involution, 
the underlying $\kk \fb$-module of (\ref{eqn:sharp_underlying_K-Bsigma}) is supported on sets of even parity. 
 For $s\in \nat$, there are identifications of $\kk \sym_{2s}$-modules:
\begin{eqnarray*}
\big(
(\kk \ub)_{(-;+)}^\sharp \otimes _{\kk \fb} \lfb^* (\cpd_{(B, \sigma)})
\big) 
(\mathbf{2s}) 
&= & 
\bigoplus_{t \geq s} 
(\kk \ub)_{(-;+)} (\mathbf{2s}, \mathbf{2t}) ^\sharp
\otimes_{\kk (\sym_2 \wr \sym_t)} 
(B^{\otimes t} \otimes \kk_{(+;-)}^{[t]})
\\
&\cong &
 \bigoplus_{t \leq s} 
 \kk_{(-;+)}^{[t-s]} \otimes _{\kk (\sym_2\wr \sym_{t-s})} \kk \sym_{2t} \otimes _{\kk (\sym_2 \wr \sym_t)} (B^{\otimes t} \otimes \kk_{(+;-)}^{[t]}).
\end{eqnarray*}
Here, the term indexed by $s$ is placed in homological degree $t-s$; the left $\kk \sym_{2s}$-action on the $t$th term of the second expression comes from the inclusion $\sym_{2s} \times (\sym_2 \wr \sym_{t-s}) \subset \sym_{2s} \times \sym_{2 (t-s)} \subset \sym_{2t}$.
\end{lem}

There is a  geometric description in terms of suitable graphs, as explained below.  This is related to the general case of the hairy graph complexes given in Section \ref{sect:complexes}. However, in this section, we have `blown up' the vertices of the usual hairy graph complexes to give `black edges';  this allows us to contract hairs, giving rise to `white vertices'.

Fix $t\geq s$; we consider $\kk$-linear combinations of the following labelled directed graphs (modulo relations):
\begin{itemize}
\item 
$t$ directed black edges labelled by elements of $B$, subject to the $\sigma$-chord orientation relation and with an order (controlling  orientation signs);
\item 
$t-s$ directed red edges; changing the direction gives a sign; 
\item 
$2 (t-s)$ gray vertices; these are bivalent with one black and one red edge attached; 
\item 
$2s$  univalent white vertices, labelled bijectively by elements of $\mathbf{2s}$,  with one black edge attached. 
\end{itemize}

As usual, one can focus upon the connected components of such graphs; there are two possibilities: 
\begin{enumerate}
\item 
a `string' with endpoints given by two white vertices; 
\item 
a `wheel' with no white vertices. 
\end{enumerate} 

For instance, (ignoring the ordering of the black edges and omitting labels of the white vertices): 
\begin{center}
\begin{tikzpicture} [scale=.4]
\draw (0,0)-- (6,0);
\draw [-latex] (0,0) -- (1.3,0);
\draw [-latex] (4,0)-- (5.3,0);
\draw [very thick, -latex, red](2,0) -- (3.5,0);
\draw [very thick, red](3,0) -- (4,0);
\draw [fill = white] (0,0) circle (0.2);
\draw [fill = gray] (2,0) circle (0.2);
\draw [fill = gray] (4,0) circle (0.2);
\draw [fill = white] (6,0) circle (0.2);
\node [above] at (1,0) {$\scriptstyle{a}$};
\node [above] at (5,0) {$\scriptstyle{b}$};
\node [right] at (6,0) {.};
\end{tikzpicture}
\end{center}
The direction of the edges can be changed, modulo the $\sigma$-chord orientation relation for black edges, or up to orientation sign for red edges.

Or (again ignoring the ordering of the black edges):
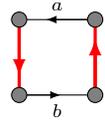
\begin{center}
\begin{tikzpicture} [scale=.4]
\draw (1,1) -- (-1,1) -- (-1,-1) -- (1,-1) -- (1,1);
\draw [very thick, red] (1,-1) -- (1,1);
\draw [very thick, red] (-1,1) -- (-1,-1);
\draw [very thick, -latex, red] (1,-1) -- (1,.5);
\draw [very thick, -latex,  red] (-1,1) -- (-1,-.5);
\draw [-latex] (1,1)-- (-.3,1); 
\draw [-latex ] (-1,-1) -- (.1, -1);
\draw [fill = gray] (1,1) circle (0.2);
\draw [fill = gray] (1,-1) circle (0.2);
\draw [fill = gray] (-1,-1) circle (0.2);
\draw [fill = gray] (-1,1) circle (0.2);
\node [above] at (0,1) {$\scriptstyle{a}$};
\node [below] at (0,-1) {$\scriptstyle{b}$};
\node [right] at (1,-1) {.};
\end{tikzpicture}
\end{center}

The differential  is induced by contracting red edges, using multiplication in $B$ to label the new black edge obtained by concatenation. 
In the case  of a bigon (i.e., one black, one red edge, connected at their two endpoints), the contraction yields zero.

There are the usual technical requirements:
\begin{enumerate}
\item 
black edges must be reordered so that the operation is on the last two edges, with order of these dictated by that of the red edge (this may introduce a sign); 
\item 
the directions of the black and red edges involved in the contraction must be the same; adjustment involves the $\sigma$-chord orientation relation and the orientation sign for red edges.
\end{enumerate}
Clearly, for each connected component, one can fix once and for all a choice of compatible orientations (there are only two possibilities); this is already done in the two illustrated examples. 

\begin{rem}
\ 
\begin{enumerate}
\item 
In the case $s=0$, all connected components are `wheels'. The homological degree is one half the number of gray vertices. 
 One then reduces to considering each connected component separately. The associated complex corresponds to the dihedral complex  (see \cite[Section 10.5.4]{MR1600246}, which is based on \cite{MR937318}). Indeed Loday and Procesi's proof essentially contains this identification.
\item 
More generally, the relationship between hairy graph complexes and dihedral complexes has been used and investigated in \cite{MR3347586} for example.
\end{enumerate}
\end{rem}

\begin{rem}
Using the $\kk \db$-module $\sfb^* (\cpd_{(B,\sigma)})$ in place of $\lfb^* (\cpd_{(B,\sigma)})$, together with the appropriate Koszul complex,  gives rise to an {\em odd} variant of the above.
\end{rem}

\section{Hairy graph complexes}
\label{sect:complexes}

Throughout this section, $\kk$ is a field of characteristic zero and $\cpd$ is a non-unital cyclic operad such that $\cpd (\mathbf{0})=0$.
We are interested in the complexes 
\begin{eqnarray}
\label{eqn:complex_K-}
(\kk\ub)_{(-;+)}^\sharp  \otimes_{(\kk\ub)_{(-;+)}} \mathscr{K}_- \otimes_{(\kk\db)_{(-;-)}} \lfb^* (\cpd)
&\cong &
(\kk\ub)_{(-;+)}^\sharp  \otimes_{\kk \fb} \lfb^* (\cpd)
\\
\label{eqn:complex_K+}
(\kk \ub)_{(+;-)} ^\sharp \otimes _{(\kk \ub) _{(+;-)}} \mathscr{K}_+ \otimes_{\kk \db} \sfb^*(\cpd)
&\cong &
(\kk \ub)_{(+;-)} ^\sharp \otimes _{\kk \fb} \sfb^*(\cpd) .
\end{eqnarray}
These can be considered  as even (respectively odd) versions of the same construction, in the same way that graph complexes come in even and odd flavours. 

More specifically, we relate the above complexes to {\em hairy graph complexes}, extending Kontsevich's theorem (as generalized by Conant and Vogtmann \cite{MR2026331}), which concerns the case with no hairs  (see \cite[Theorem 1]{MR2026331}). The relationship between the methods used here and the work of Kontsevich and Conant-Vogtmann can be explained by using Proposition \ref{prop:symplectic_invariants}.  

\begin{rem}
If $\cpd$ has a unit, to avoid the acyclicity property of Theorem \ref{thm:acyclicity}, one should seek to exclude the unit. If $\cpd$ is augmented, this can be done by restricting to the `augmentation ideal' $\overline{\cpd}$. Otherwise, one can truncate, working with the subobject $\cpd_{\geq 3}$ supported on finite sets of cardinality at least three. The latter is (implicitly) the approach taken in \cite{MR3029423}, where bivalent vertices are not allowed in the graph complexes. 
\end{rem}

In considering {\em hairy} graph complexes, our approach is related to that of Conant, Kassabov, and Vogtmann in \cite{MR3029423}. In particular, our complexes are homological, using edge contraction, rather than the dual using edge expansion.

\begin{rem}
The study of hairy graph complexes can be subsumed (up to duality) in that of the {\em Feynman transform} of Getzler and Kapranov \cite{MR1601666} for modular operads, which provides the appropriate general framework. Indeed, Stoll proves the more general analogue of the result considered here: using his result characterizing modular operads as modules over the Brauer properad, in \cite[Theorem 4.5.4]{MR4541945} he relates the properadic (co)bar construction with the Feynman transform.

Moreover, Stoll explains how the orientation data in the properadic framework gives rise to {\em hyperoperads} in the sense of Getzler and Kapranov; these are used in the definition of the Feynman transform. Our  approach, restricted to cyclic operads,  allows us to avoid such considerations.
\end{rem}

We show that, by introducing  a suitable category of graphs with hairs using  the downward Brauer category $\db$ (see Section \ref{subsect:review_graphs}), the fact that  
the complexes (\ref{eqn:complex_K-}) and (\ref{eqn:complex_K+}) are respectively even and odd hairy graph complexes is transparent. The result is stated in Section \ref{subsect:hairy_graph_thm} as Theorem \ref{thm:hairy-graph_complexes}, where a proof is sketched.

\subsection{A quick review of graphs}
\label{subsect:review_graphs}

 We present a variant of the definition of the category of graphs given in \cite[Appendix A]{MR3636409}, using the category $\db$ (or, equivalently, $\ub$) rather than involutions to encode edges. 

\begin{rem}
All graphs that we consider are finite (have finitely many half edges and vertices), and have no isolated vertices. They are not necessarily connected. 
\end{rem}

\begin{nota}
\label{nota:fs}
Write $\fs$ to denote the category of finite sets and surjections, a wide subcategory of the category of finite sets and all maps,  $\fin$. 
\end{nota}

\begin{defn}
\label{defn:graphs}
A graph $\Gamma$ is  a triple of finite sets $(V_\Gamma, X_\Gamma, L_\Gamma)$ (corresponding respectively to the vertices, the half edges, and the hairs (or legs) of the graph), together with structure maps 
\begin{eqnarray*}
p_\Gamma &\in & \fs (X_\Gamma, V_\Gamma)\\
f_\Gamma &\in&  \db (X_\Gamma, L_\Gamma) = \ub (L_\Gamma, X_\Gamma)
.
\end{eqnarray*}

Let $\Gamma'$ be a second graph with $L_\Gamma = L_{\Gamma'}$. A morphism $\Phi : \Gamma \rightarrow \Gamma'$  is given by a pair of maps $\Phi^X \in \db (X_\Gamma, X_{\Gamma'})$ and $\Phi^V \in \fs (V_\Gamma, V_{\Gamma'})$ such that:
\begin{enumerate}
\item 
 the following diagram commutes in $\db$:
\begin{eqnarray}
\label{diag:Phi^X}
\xymatrix{
X_\Gamma 
\ar[rd]_{f_\Gamma}
\ar[rr]^{\Phi^X}
&&
X_{\Gamma'} 
\ar[ld]^{f_{\Gamma'}}
\\
&
L_{\Gamma}; 
}
\end{eqnarray}
\item 
the following diagram commutes in $\fin$:
\begin{eqnarray}
\label{diag:Phi^V}
\xymatrix{
X_{\Gamma'} 
\ar@{^(->}[r]^{\widetilde{\Phi^X}} 
\ar@{->>}[d]_{p_{\Gamma'}}
&
X_\Gamma 
\ar@{->>}[d]^{p_{\Gamma}}
\\
V_{\Gamma'}
&
V_\Gamma 
\ar@{->>}[l]^{\Phi^V},
}
\end{eqnarray}
where $\widetilde{\Phi^X}$ is the injective map underlying $\Phi^X$;
\item
for a `chord' of $\Phi^X$, corresponding to a pair $\{h_1, h_2\} \subset X_\Gamma$, we have $\Phi^V p_\Gamma(h_1) = \Phi^V p_\Gamma (h_2)$.
\end{enumerate} 
\end{defn}

\begin{rem}
\ 
\begin{enumerate}
\item 
The map $p_\Gamma$ determines to which vertex a half edge is attached and $f_\Gamma$ labels the hairs and determines the edges, which correspond to the `chords' given by a morphism in $\db$.
\item 
The hypothesis that $p_\Gamma$ is surjective is equivalent to the hypothesis that the graph $\Gamma$ has no isolated vertices (vertices with no half edges attached).
\item 
Above we only consider morphisms of graphs that preserve the labelling of the hairs.  
\item 
The final condition in the definition of a morphism ensures that the edges of $\Gamma$ that do not correspond to edges of $\Gamma'$ are `contracted': the endpoints of such an edge are identified.
\item 
If $V_\Gamma$ is empty, then the surjectivity of $p_\Gamma$ implies that $X_\Gamma$ and hence $L_\Gamma$ are both empty. Conversely, if $X_\Gamma$ is empty, then so are $L_\Gamma$ and $V_\Gamma$. 
\item 
We may take $X_\Gamma$, $V_\Gamma$, and $L_\Gamma$ to belong to the standard skeleton of $\fin$. 
\end{enumerate}
\end{rem}

From the definition, we have:

\begin{lem}
\label{lem:set_of_graphs}
For given finite sets $(X, V, L)$, the set of graphs $\Gamma$ with $X_\Gamma = X$, $V_\Gamma = V$, and $L_\Gamma = L$ is $\db (X, L) \times \fs (X, V)$. 
In particular, this set is finite. Hence, the set of isomorphism classes of graphs $\Gamma$ with specified $|X_\Gamma|$, $|V_\Gamma|$, $|L_\Gamma|$ is finite. 

Moreover, for fixed $(X,L)$, the set of graphs with $X_\Gamma =X$, $L_\Gamma= L$ and $V_\Gamma = \mathbf{v}$, for some $v \in \nat$, is finite.
\end{lem}

\begin{proof}
The first statement is clear from the definition. Since $\fs (X,\mathbf{v})$ is empty if $v > |X|$, the final statement follows.
\end{proof}

The previous lemma does not take into account isomorphisms between graphs. These are characterized by the following:

\begin{lem}
\label{lem:iso_graphs}
For $\Gamma$, $\Gamma'$ two graphs such that $L_\Gamma= L_{\Gamma'}$, a morphism $(\Phi^X, \Phi^V)$ is an isomorphism if and only if both $\Phi^X$ and $\Phi^V$ are isomorphisms. 
Moreover, in this case, $\Phi^V$ is determined by $\Phi^X$. 
\end{lem}

\begin{proof}
The first statement is clear. The fact that $\Phi^X$ is an isomorphism implies, in particular, that no edge is contracted. Since we have supposed that $p_\Gamma$ is surjective and, by the above, $\widetilde{\Phi^X}$ is a bijection, the diagram (\ref{diag:Phi^V}) shows that $\Phi^V$ is determined by $\widetilde{\Phi^X}$, whence the result.
\end{proof}

\begin{rem}
\label{rem:graph_automorphisms}
Consider a graph $\Gamma$ and the group of automorphisms $\aut (\Gamma)$ (preserving the leg structure). An automorphism of $\aut (X_\Gamma)$ defines an automorphism of $\Gamma$ if and only if the following conditions are both satisfied. 
\begin{enumerate}
\item 
It permutes the edges, corresponding to the commutativity in $\db$ of the diagram (\ref{diag:Phi^X}); if the graph $\Gamma$ has $t$ edges, the edge structure can be represented  by an element of $\fb (\mathbf{2t}, X_\Gamma \backslash L_\Gamma)$ (abusively writing $ X_\Gamma \backslash L_\Gamma$ for the set of half edges that are not legs). The automorphism of $X_\Gamma$ fixes $L_\Gamma$, hence reduces to an automorphism of $\aut (X_\Gamma \backslash L_\Gamma) \cong \sym_{2t}$; the edge condition implies that the automorphism must live in $\sym_2 \wr \sym_t \subset \sym_{2t}$. 
\item 
It is compatible with the vertex structure, as encoded by the commutativity of (\ref{diag:Phi^V}).
\end{enumerate}
 This identifies $\aut (\Gamma)$ as a subgroup of $\sym_2 \wr \sym_t \subset \aut (X_\Gamma\backslash L_\Gamma)$.
\end{rem}

\begin{exam}
\label{exam:contraction}
Let $\Gamma$ be a graph and suppose that $\Gamma$ has at least one edge, i.e., $L_\Gamma$ is not isomorphic to $X_\Gamma$. Then for a `chord' appearing in $f_\Gamma \in \db (X_\Gamma, L_\Gamma)$, define $X_{\Gamma'}$ to be the subset of $X_\Gamma$ obtained by omitting the corresponding pair of elements $\{h_1,h_2\}$ and $\Phi^X \in \db (X_\Gamma, X_{\Gamma'})$ to be the morphism defined by the inclusion.  Define $V_{\Gamma'}:= V_\Gamma/ _{h_1 \sim h_2}$ and $\Phi^V$ to be the canonical surjection. Take $p_{\Gamma'}$ to be the map $X_{\Gamma'} \rightarrow V_{\Gamma'}$ that makes (\ref{diag:Phi^V}) commute. 

Then, if $p_{\Gamma'}$ is surjective, there is a unique graph structure on $(X_{\Gamma'},V_{\Gamma'}, L_\Gamma)$ such that $\Phi^X$ and $\Phi^V$ induce a morphism of graphs $\Phi : \Gamma \rightarrow \Gamma'$. This defines the {\em contraction} of the edge $\{h_1, h_2\}$.

(If one relaxes the requirement that the graph has no isolated vertices, then the requirement that $p_{\Gamma'}$ be surjective can be dropped.)
\end{exam}

\subsection{Hairy graph complexes}
\label{subsect:hairy_graph_thm}

Fix finite sets $X, V, L$; we may take these to be in the skeleton of $\fin$, hence the triple can be written $(\mathbf{x}, \mathbf{v}, \mathbf{l})$ for natural numbers $x$, $v$, $l$. By Lemma \ref{lem:set_of_graphs}, a graph $\Gamma$ on these sets is given by an element  $(f_\Gamma,p_\Gamma)\in  \db (\mathbf{x}, \mathbf{l}) \times \fs (\mathbf{x}, \mathbf{v})$.

Now, given  $\cpd \in \ob \nuco$, we associate to  the map $p_\Gamma$ the vector space:
$
\bigotimes_{i\in \mathbf{v}} 
\cpd (p_\Gamma^{-1} (i))$.
 After summing over all possible $p_\Gamma \in \fs (\mathbf{x}, \mathbf{v})$, we have the identification
\[
\bigoplus_{p_\Gamma \in \fs (\mathbf{x}, \mathbf{v})} \bigotimes_{i\in \mathbf{v}} 
\cpd (p_\Gamma^{-1} (i))
\cong 
\cpd^{\odot v} (\mathbf{x}),
\]
by the definition of the convolution product $\odot$ (also using that $\cpd (\mathbf{0})=0$). 
 Then, summing over all graphs gives 
$ 
\kk \db (\mathbf{x}, \mathbf{l}) \otimes _\kk \cpd^{\odot v} (\mathbf{x})$. Letting $v$ vary, then yields 
\[
\kk \db (\mathbf{x}, \mathbf{l}) \otimes _\kk \cpd^{\odot \bullet} (\mathbf{x}).
\]

This is the first step towards the construction of the hairy graph complex. However, there is a missing ingredient: we have not introduced the appropriate orientation signs. For these, there are two possibilities: 
\begin{enumerate}
\item 
the even case: an orientation sign is associated to the order of the vertices and edges have an orientation sign associated to their direction; 
\item 
the odd case: an orientation sign is associated to the order of the edges.
\end{enumerate}

In each case, this requires replacing $\kk \db$ by the appropriate twisted variant; in the first case, one also introduces  the sign $\Lambda^v (\kk \mathbf{v})$. 

Having introduced these orientation signs, for each $v$, we pass to the quotient by the action of $\aut (\mathbf{v})$; this corresponds to forgetting the labelling of the vertices (whilst taking into account the orientation data). This yields respectively:
\begin{eqnarray*}
&&
(\kk \db)_{(-;+)} (\mathbf{x}, \mathbf{l}) \otimes _\kk \lfb^\bullet (\cpd) (\mathbf{x})
\\
&& 
(\kk \db)_{(+;-)} (\mathbf{x}, \mathbf{l}) \otimes _\kk   \sfb^\bullet (\cpd) (\mathbf{x}).
\end{eqnarray*}

Then we form the quotient by the action of the groupoid of isomorphisms between such graphs (by Lemma \ref{lem:iso_graphs}, the action on vertices is determined by $\aut(\mathbf{x})$);  this corresponds to forming the coinvariants for the action of $\aut (\mathbf{x})$. This gives respectively:
\begin{eqnarray*}
&&
(\kk \db)_{(-;+)} (\mathbf{x}, \mathbf{l}) \otimes _{\kk \aut(\mathbf{x})} \lfb^\bullet (\cpd) (\mathbf{x})
\\
&& 
(\kk \db)_{(+;-)} (\mathbf{x}, \mathbf{l}) \otimes _{\kk \aut(\mathbf{x})} \sfb^\bullet (\cpd) (\mathbf{x}).
\end{eqnarray*}

Now, as representations of $\aut (\mathbf{x})$, we have isomorphisms 
\begin{eqnarray*}
(\kk \db)_{(-;+)} (\mathbf{x}, \mathbf{l})
&\cong &(\kk \db)_{(-;+)} (\mathbf{x}, \mathbf{l})^\sharp
;
\\
(\kk \db)_{(+;-)} (\mathbf{x}, \mathbf{l}) 
& \cong &(\kk \db)_{(+;-)} (\mathbf{x}, \mathbf{l}) ^\sharp
\end{eqnarray*}
(adjusting variance as usual). Therefore, upon summing over $x \in \nat$, this gives 
\begin{eqnarray*}
(\kk \db)_{(-;+)} (-, \mathbf{l})^\sharp \otimes _{\kk \fb} \lfb^\bullet (\cpd)
\\
(\kk \db)_{(+;-)} (-, \mathbf{l})^\sharp \otimes _{\kk \fb} \sfb^\bullet (\cpd)
\end{eqnarray*}
respectively. These are the underlying objects of the respective hairy graph complexes; by inspection, these identify with those of the complexes (\ref{eqn:complex_K-}) and (\ref{eqn:complex_K+}).

Now, the respective hairy graph complex differentials are defined by edge contraction, using the composition in the cyclic operad. Tracing through the above constructions, one shows that these correspond to the respective Koszul complex differentials. 

\begin{rem}
\label{rem:CKV_hairy}
The above presentation of the hairy graph complexes can be compared with the complex of hairy $\mathscr{O}$-graphs introduced by Conant, Kassabov, and Vogtmann in \cite[Section 3]{MR3029423}, where the cyclic operad is denoted by $\mathscr{O}$. We underline the following differences: 
\begin{enumerate}
\item 
The authors work with graphs with no bivalent vertices; in our framework, this corresponds to replacing $\mathscr{C}$ by $\mathscr{C}_{\geq 3}$. 
\item 
The authors label the hairs by elements of a symplectic vector space $(W,\omega)$,  sometimes restricting to labels in the symplectic basis $\mathcal{B}$ (the notation  $W$ is used here rather than $V$ so as not to confuse with our notation for vertices). However, the authors point out that the symplectic form is irrelevant (see \cite[Section 3.4]{MR3029423}, for example). Hence, using the Schur correspondence, one can reduce to labelling legs {\em bijectively} by elements of a finite set, as we do. 
\item 
The authors only work with the even case, as exemplified by their orientation datum given by \cite[Definition 3.1]{MR3029423}. (Note that, our vertices are equivalent to the internal vertices of {\em loc. cit.}, similarly for the edges.)
\end{enumerate}

Conant, Kassabov, and Vogtmann define their hairy $\mathscr{O}$-graph complex 
$$
\mathcal{H}_W = \bigoplus_ k C_k \mathcal{H}_W,
$$
where $k$ corresponds to the number of vertices of the graph. The differential is described in \cite[Section 3.2]{MR3029423}, corresponding to the usual contraction differential. \end{rem}

The above discussion has provided a sketch proof of the following theorem, in which we have replaced (this is only a question of notation) $(\kk \db)_{(-;+)}^\sharp$ by $(\kk \ub)_{(-;+)}^\sharp$ and  $(\kk \db)_{(+;-)}^\sharp$ by $(\kk \ub)_{(+;-)}^\sharp$. 

\begin{thm}
\label{thm:hairy-graph_complexes}
For $\cpd \in \ob \nuco$,  the complexes
\begin{eqnarray*}
(\kk\ub)_{(-;+)}(\mathbf{l}, -) ^\sharp  \otimes_{(\kk\ub)_{(-;+)}} \mathscr{K}_- \otimes_{(\kk\db)_{(-;-)}} \lfb^* (\cpd_{\geq 3})
&\cong &
(\kk\ub)_{(-;+)}(\mathbf{l}, -) ^\sharp  \otimes_{\kk \fb} \lfb^* (\cpd_{\geq 3})
\\
(\kk \ub)_{(+;-)}(\mathbf{l}, -) ^\sharp \otimes _{(\kk \ub) _{(+;-)}} \mathscr{K}_+ \otimes_{\kk \db} \sfb^*(\cpd_{\geq 3})
&\cong &
(\kk \ub)_{(+;-)} (\mathbf{l}, -)^\sharp \otimes _{\kk \fb} \sfb^*(\cpd_{\geq 3}) 
\end{eqnarray*}
are the even and odd hairy graph complexes with legs labelled by $\mathbf{l}$.
\end{thm}

\begin{rem}
\ 
\begin{enumerate}
\item 
The restriction to $\cpd_{\geq 3}$ is only imposed so as to give compatibility with the definition given in \cite{MR3029423}, which uses graphs with no bivariant vertices.
\item 
We have forgotten structure as compared to the Feynman transform, in the same way that considering $\sfb^*(\cpd)$ as a $\kk \db$-module (respectively $\lfb^* (\cpd)$ as a $(\kk\db)_{(-;-)}$-module) forgets that these structures are derived from $\cpd$ considered as an algebra of the downward Brauer properad.   
\item 
The complexes appearing in Theorem \ref{thm:hairy-graph_complexes} arise from  complexes of $(\kk \ub)_{(-;+)}$-modules and $(\kk \ub)_{(+;-)}$ respectively, encoding functoriality with respect to $\mathbf{l}$.  This may be compared to constructions of Conant, Kassabov, and Vogtmann \cite[Section 4]{MR3029423}, for example.
\end{enumerate}
\end{rem}


\end{document}